\documentclass[a4paper,10pt,leqno]{amsart}
        \title{Conformal nets V: dualizability}       
       \author{Arthur Bartels}
      \address{Westf\"alische Wilhelms-Universit\"at M\"unster\\
               Mathematisches Institut\\
               Einsteinstr.~62,
               D-48149 M\"unster, Deutschland}
        \email{bartelsa@wwu.de}
      \urladdr{http://www.math.uni-muenster.de/u/bartelsa}
       \author{Christopher L. Douglas} 
      \address{Mathematical Institute\\ Radcliffe Observatory Quarter\\ Woodstock Road\\ Oxford\\ OX2 6GG\\ United Kingdom}
        \email{cdouglas@maths.ox.ac.uk}
      \urladdr{http://people.maths.ox.ac.uk/cdouglas}
       \author{Andr{\'e} Henriques}
      \address{Mathematical Institute\\ Radcliffe Observatory Quarter\\ Woodstock Road\\ Oxford\\ OX2 6GG\\ United Kingdom}
        \email{andre.henriques@maths.ox.ac.uk}
      \urladdr{http://www.andreghenriques.com}  
\usepackage{hyperref}
\usepackage{enumerate,amssymb}
\usepackage[arrow,curve,matrix,tips,2cell]{xy}
  \SelectTips{eu}{10} \UseTips
  \UseAllTwocells
\usepackage{tikz}                \usetikzlibrary{calc}               \usetikzlibrary{matrix}		\usetikzlibrary{patterns}
\usepackage{pdfcolmk}
\usepackage{calc}
\usepackage{multirow}
\usepackage{stmaryrd}
       \newcommand{\cala}{\mathcal{A}}
       \newcommand{\calb}{\mathcal{B}}
  \newcommand{\IC}{\mathbb{C}}     \newcommand{\calc}{\mathcal{C}}
  \newcommand{\ID}{\mathbb{D}}     \newcommand{\cald}{\mathcal{D}}

  \newcommand{\IN}{\mathbb{N}}     \newcommand{\caln}{\mathcal{N}}

  \newcommand{\bfB}{{\mathbf B}}

  \newcommand{\cK}{\mathcal K}
  \newcommand{\cA}{\mathcal A}
  \newcommand{\cB}{\mathcal B}

  \newcounter{commentcounter}

  \definecolor{AHcolor}{rgb}{0.5,0.0,0.5}   % Textcolor for AB
  \definecolor{CDcolor}{rgb}{0.7,0.0,0.3}   % Textcolor for CD
  \definecolor{ABcolor}{rgb}{0.2,0.8,0.2}   % Textcolor for AH
  \newcommand{\AB}[1]{\marginpar{\raggedright\tiny\color{ABcolor}{ #1}}}
  
  \newcommand{\CD}[1]{\marginpar{\raggedright\tiny\color{CDcolor}{ #1}}}

  \newcommand{\tikzmath}[2][]
     {\vcenter{\hbox{\begin{tikzpicture}[#1]#2
                     \end{tikzpicture}}}
     }

  \newcommand{\textscale}{.03}
  
  \newcommand{\squarescale}{.07}
  \newcommand{\displscale}{.05}

  \definecolor{spacecolor}{gray}{.7}
  \definecolor{antispacecolor}{gray}{.45}

  \theoremstyle{plain}
  \newtheorem{theorem}{Theorem}[section]
  
  \newtheorem{lemma}[theorem]{Lemma}
  \newtheorem{corollary}[theorem]{Corollary}
  \newtheorem{proposition}[theorem]{Proposition}

  \newtheorem*{theorem*}{Theorem}
  \newtheorem{introthm}{Theorem}

  \theoremstyle{definition}
  \newtheorem{definition}[theorem]{Definition}

  \theoremstyle{remark}
  
  \newtheorem{remark}[theorem]{Remark}
  \newtheorem{scholium}[theorem]{Scholium}

  \makeatletter\let\c@equation=\c@theorem\makeatother

     {\end{list}}

  \DeclareMathOperator{\Diff}{Diff}

  \DeclareMathOperator{\id}{id}

  \newcommand{\INT}{{\mathsf{INT}}}
  
  \newcommand{\CN}{{\mathsf{CN}}}

\def\tworarrow{\hspace{.1cm}{\setlength{\unitlength}{.50mm}\linethickness{.09mm}
	\begin{picture}(8,8)(0,0)\qbezier(0,4)(4,7)(8,4)\qbezier(0,1)(4,-2)(8,1)\qbezier(3.5,4)(3.5,3)(3.5,1.5)
	\qbezier(4.5,4)(4.5,3)(4.5,1.5)\qbezier(4,0.8)(4.5,1.7)(5.5,2)\qbezier(4,0.8)(3.5,1.7)(2.5,2)
	\qbezier(8,1)(7.4,.2)(7.7,-.7)\qbezier(8,1)(7,1)(6.5,1.5)\qbezier(8,4)(7.4,4.8)(7.7,5.7)
	\qbezier(8,4)(7,4)(6.5,3.5)\end{picture}\hspace{.1cm}}}

  \newcommand{\alg}{{\mathit{alg}}}
  
  \newcommand{\op}{{\mathit{op}}}

  \newcommand{\x}{{\times}}

  \newcommand{\dd}{{\partial}}
  \newcommand{\e}{{\varepsilon}}

  \newcommand{\ra}{\rightarrow}
  
\newcommand{\nid}{\noindent}

  \def\CC{\mathbb C}

\def\cA{\mathcal A}

\def\rev{\mathrm{rev}}

\def\lemHKK{Lemma A.4}
\def\lemvacvacvacdefects{Lemma 1.15}
\def\lemDtildeHBfinite{Lemma 3.17}
\def\notationnoncanonvacuumdefects{Notation 1.14}
\def\lemNTbetweenmodulecategories{Lemma B.24}
\def\propKLMfinitenessdefects{Prop. 3.18}
\def\lemdualvacuum{Lemma 3.4}
\def\thmnetsarbitrarymanifolds{Thm 1.3}
\def\eqdefectsdisconnected{Eq 1.34}
\def\interchangeiso{Eq 6.25}
\def\lemnoncanonvacuumdefect{Lemma 1.13}
\def\thminterchange{Thm 6.2}
\def\lemNTbetweenmodulecategories{Lemma B.24}
\def\exampleonenotoneone{Example 3.5}
\def\remarkweakidentity{Remark 1.40}

\def\corcontragredientbimodule{Cor 6.12}
\def\thmnetoncircle{Thm 1.20}
\def\propnetoncircle{Prop 1.25}
\def\thmconformalblocks{Thm 2.18}

\def\KLM{Kawahigashi-Longo-Mueger(2001multi-interval)}

\newcommand{\Hilb}{\mathsf{Hilb}}
\newcommand{\vN}{\mathsf{vN}}
\newcommand{\Hom}{\mathrm{Hom}}

  \DeclareRobustCommand{\SkipTocEntry}[5]{}

\begin{document}

\begin{abstract}

We prove that finite-index conformal nets are fully dualizable objects in the $3$-category of conformal nets.  Therefore, assuming the cobordism hypothesis applies, there exists a local framed topological field theory whose value on the point is any finite-index conformal net.  Along the way, we prove a Peter--Weyl theorem for defects between conformal nets, namely that the annular sector of a finite defect is the sum of every sector tensor its dual.
	
\end{abstract}

\maketitle

%====================================================================

% ------------------------------------- 
\tableofcontents

\newcommand{\comment}[1]{}

\section*{Introduction}

A finite-dimensional Hilbert space $H$ is dualizable in the sense that there is a Hilbert space $H^\ast$ together with evaluation and coevaluation morphisms $\mathrm{ev}: H \otimes H^\ast \ra \CC$ and $\mathrm{coev}: \CC \ra H^\ast \otimes H$ such that the identity $\id_H$ can be recovered as the composite $(\mathrm{coev} \otimes \id_H) \circ (\id_H \otimes \,\mathrm{ev})$, and the identity $\id_{H^\ast}$ can be recovered as a similar composite; indeed, every dualizable Hilbert space is finite-dimensional.  

The 2-category $\vN$ of von Neumann algebras deloops the category $\Hilb$ of Hilbert spaces in the sense that $\Hom_\vN(1,1) \cong \Hilb$.  If a von Neumann algebra $A$ is a finite direct sum of type $I$ factors, then it is fully dualizable in the sense that there is a von Neumann algebra $A^\op$ together with evaluation bimodule ${}_{A \otimes A^\op} H_{\CC}$ and coevaluation bimodule ${}_\CC H_{A^\op \otimes A}$ such that the identity bimodule ${}_A L^2(A)_A$ can be recovered as a composite of the evaluation and coevaluation (and the identity bimodule for $A^\op$ can be similarly recovered), and such that the evaluation and coevaluation bimodules themselves admit adjoints.  A fully dualizable von Neumann algebra is in fact necessarily a finite direct sum of type $I$ factors.  More generally, full dualizability functions as a strong finiteness condition on the objects of a higher category.

The 3-category $\CN$ of conformal nets deloops the 2-category $\vN$ of von Neumann algebras, in the sense that $\Hom_{\CN}(1,1) \cong \vN$ \cite[Prop.\,1.22]{BDH(1*1)}.  In this paper, the fifth in a series~\cite{BDH(nets), BDH(modularity), BDH(1*1), BDH(3-category)} concerning the 3-category of conformal nets, we investigate the dualizability properties of conformal nets and their defects and sectors.  Our main result is that a conformal net is fully dualizable if (Theorem \ref{introthm:dualizability} below) and only if (Theorem \ref{introthm:dualizability -- other way} below) it has finite index.

\subsection*{Dualizability}
\addtocontents{toc}{\SkipTocEntry}
Recall that two $i$-morphisms $F: A \ra B$ and $G: B \ra A$ in an $n$-category ($i<n$) are called \emph{adjoint} (or \emph{dual}), denoted $F \dashv G$, if there exist $(i+1)$-morphisms, the unit $s: \id_B \ra G \cdot F$ and the counit $r: F \cdot G \ra \id_A$ such that the composite $(\id_G \cdot r) \cdot (s \cdot \id_G)$ is equivalent to $\id_G$ and the composite $(r \cdot \id_F) \cdot (\id_F \cdot s)$ is equivalent to $\id_F$; we say that $F$ admits $G$ as its right adjoint, or equivalently that $G$ admits $F$ as its left adjoint.  Similarly, two objects $f$ and $g$ in a symmetric monoidal $n$-category are called \emph{dual} if there exist $1$-morphisms, the coevaluation $s: 1 \ra g \otimes f$ and the evaluation $r: f \otimes g \ra 1$, such that the composite $(\id_g \otimes r) \cdot (s \otimes \id_g)$ is equivalent to $\id_g$ and the composite $(r \otimes \id_f) \cdot (\id_f \otimes s)$ is equivalent to $\id_f$.

An $i$-morphism $F: A \ra B$ in an $n$-category ($i<n$) is called \emph{fully dualizable} if there is an infinite chain of adjunctions $\cdots F^{LL} \dashv F^L \dashv F \dashv F^R \dashv F^{RR} \dashv \cdots$ such that every unit and counit morphism in each of the adjunctions in that chain itself admits a similar infinite chain of adjunctions, such that every unit and counit morphism in each of the adjunctions in all of those chains in turn admits an infinite chain of adjunctions, and so on until one reaches a chain of $(n-1)$-morphisms, at which point the conditions stop.  (We refer to an $(n-1)$-morphism that has an infinite chain of left and right adjoints, and is therefore fully dualizable, simply as `dualizable'.)  Similarly, an object in a symmetric monoidal $n$-category is fully dualizable (also called `$n$-dualizable') if it admits a dual and the coevaluation and evaluation morphisms are fully dualizable.  An $n$-category is said to \emph{have all duals} if every object is fully dualizable and every $i$-morphism ($i<n$) is fully dualizable.  (Note that the notions of fully dualizable and of having all duals do not depend on the exact model one chooses for symmetric monoidal $n$-categories, because the dualizability conditions can be phrased entirely in terms of homotopy 2-categories canonically associated to the $n$-category.  For a more detailed discussion of the notion of dualizability, see~\cite[Appendix A]{DSS}.)

The cobordism hypothesis~\cite{baezdolan,Lurie(On-classification-TFT),Ayala-Francis:cobordism-hypothesis} ensures that for any fully dualizable object $c$ in a symmetric monoidal $n$-category $\calc$, there is a local framed topological field theory $F_c : \mathrm{Bord}_n^{\mathrm{fr}} \ra \calc$ whose value on the positively framed point is $c$.\footnote{See the section on `Manifold invariants' below, and in particular Footnote~\ref{foot-ch}, for a discussion of the applicability of the cobordism hypothesis to the symmetric monoidal 3-category of conformal nets.}  In particular, for any such object, there is an associated framed $n$-manifold invariant.

\subsection*{Finiteness}
\addtocontents{toc}{\SkipTocEntry}
We will investigate the dualizability of objects and morphisms in the symmetric monoidal $3$-category of conformal nets.  To that end, we introduce notions of `finiteness' for nets, defects, and sectors, arranged in such a way that finiteness ensures both the existence of a dual (or adjoint) and in turn the finiteness of the coevaluation and evaluation (or unit and counit) morphisms.  We will therefore be able to successively establish that finiteness implies dualizability for sectors, defects, and conformal nets.

Consider the following subintervals of the standard circle:
\[
\begin{split}
S^1_\top:=\{z\in S^1\,|\, \Im\mathrm{m}(z)\ge 0\},
\qquad S^1_\dashv:=\{z\in S^1\,|\, \Re\mathrm{e}(z)\ge 0\},\\
S^1_\bot:=\{z\in S^1\,|\, \Im\mathrm{m}(z)\le 0\},
\qquad S^1_\vdash:=\{z\in S^1\,|\, \Re\mathrm{e}(z)\le 0\}.
\end{split}
\]
Moreover, let $I_1,\ldots, I_4\subset S^1$ be the subintervals indicated here:
  \begin{equation}\label{eq: picture of I_1 ... I_4}
  \tikzmath[scale=.03]
  { \useasboundingbox (-20,-22) rectangle (20,22);
        \draw (0,0) circle (15);
        \draw[thick] (10,10) -- (11.5,11.5)
              (10,-10) -- (11.5,-11.5)
              (-10,10) -- (-11.5,11.5)
              (-10,-10) -- (-11.5,-11.5)       
              (-20,0) node {$\scriptstyle I_2$} 
              (20,0) node {$\scriptstyle I_4$}
              (0,20) node {$\scriptstyle I_1$}
              (0,-20) node {$\scriptstyle I_3$};      }
  \end{equation} 
When appropriate, we equip the standard circle $S^1$ with its standard bicoloring $S^1_\circ=S^1_\vdash$, $S^1_\bullet=S^1_\dashv$, and give $I_1,\ldots, I_4$ the induced bicoloring, so that $I_1$ and $I_3$ are genuinely bicolored, $I_2$ is white, and $I_4$ is black.

We henceforth assume that all conformal nets and defects are \emph{semisimple}, that is finite direct sums of irreducible ones; (a conformal net or defect is irreducible if it does not admit a non-trivial direct sum decomposition).
\begin{definition} \label{def: finiteness for defects} \ \\\vspace{-10pt}
\begin{itemize}
\item[$\circ$] A conformal net $\cala$ is \emph{finite} if
the bimodule ${}_{\cala(I_1\cup I_3)}H_0(\cala)_{\cala(I_2\cup I_4)^\op}$
is dualizable as a morphism in the $2$-category of von Neumann algebras.\footnote{If $\cala$ is irreducible, then this condition is equivalent to the conformal net having finite index, as follows.  Recall from~\cite[Def 3.1]{BDH(nets)} that the index of a conformal net $\cala$ is defined as the minimal index of the inclusion $\cala(I_1 \cup I_3) \subset \cala(I_2 \cup I_4)'$.  By~\cite[Prop 7.5]{BDH(Dualizability+Index-of-subfactors)}, if this minimal index is finite, then the bimodule ${}_{\cala(I_1\cup I_3)}H_0(\cala)_{\cala(I_2\cup I_4)^\op}$ is dualizable.  Conversely, if that bimodule is dualizable, then, by~\cite[Def 5.1]{BDH(Dualizability+Index-of-subfactors)}, its statistical dimension is finite and thus, by~\cite[Def 5.10 \& Prop 7.3]{BDH(Dualizability+Index-of-subfactors)}, the corresponding minimal index is finite.\label{foot-dualizable}}\vspace{3pt}
\item[$\circ$] A defect ${}_\cala D_\calb$ between finite conformal nets is \emph{finite} if
the action of $D(I_1) \otimes_{alg} D(I_3)$ on $H_0(D)$ extends to $D(I_1) \,\bar{\otimes}\, D(I_3)$,
that is, if $H_0(D)$ is split as $D(I_1)$-$D(I_3)^\op$-bimodule.\vspace{3pt}
\item[$\circ$] A $D$-$E$-sector $H$ between defects $D$ and $E$, is \emph{finite} if the the bimodule ${}_{D(S^1_\top)} H_{E(S^1_\bot)^\op}$ is dualizable as a morphism in the $2$-category of von Neumann algebras.
\end{itemize}
\end{definition} 
\nid Note that, because there is a contravariant involution on the 2-morphisms of the 2-category of von Neumann algebras (namely the adjoint map of Hilbert spaces), a left adjoint bimodule is also a right adjoint bimodule and vice versa; thus for a bimodule to be dualizable it suffices that it admit a single adjoint.

\subsection*{Statement of results}
\addtocontents{toc}{\SkipTocEntry}
In order to construct adjunctions for defects, we will need to understand the Hilbert space assigned by a defect to a bicolored annulus.  To that end, we prove the following Peter--Weyl annular decomposition theorem for defects, generalizing the Kawahigashi--Longo--M\"uger theorem for conformal nets~\cite[Thm~9]{\KLM}.  Given a bicolored annulus $A=
\def\coords{
  \coordinate (a) at (0,0);
  \coordinate (b) at (.15,1);
  \coordinate (c) at (-.2,2);
  \coordinate (d) at (0,3);
  \coordinate (e) at (-5,.4);
  \coordinate (f) at (-5,2.6);
  \coordinate (g) at (-1,1.2);
  \coordinate (h) at (-1,2);
  \coordinate (a') at (d);
  \coordinate (b') at (c);
  \coordinate (c') at (b);
  \coordinate (d') at (a);
\begin{scope}[yshift = 85, rotate= 180]
  \coordinate (e') at (-5,.4);
  \coordinate (f') at (-5,2.6);
  \coordinate (g') at (-1,1.2);
  \coordinate (h') at (-1,2);
\end{scope}
}
\tikzmath[scale=.15]{\coordinate (a) at (0,0);\coordinate (b) at (.15,1);\coordinate (c) at (-.2,2);\coordinate (d) at (0,3);\coordinate (e) at (-5,.4);\coordinate (f) at (-5,2.6);\coordinate (g) at (-1,1.2);\coordinate (h) at (-1,2);\coordinate (a') at (d);\coordinate (b') at (c);\coordinate (c') at (b);\coordinate (d') at (a);\begin{scope}[yshift = 85, rotate= 180]\coordinate (e') at (-5,.4);\coordinate (f') at (-5,2.6);\coordinate (g') at (-1,1.2);\coordinate (h') at (-1,2);\end{scope}
\fill[fill = gray!30] (a) to [out = 180, in = -45, looseness=1.1] (e) to [out = -45 + 180, in = 225, looseness=1.1] (f) to [out = 225 + 180, in = 180, looseness=1.1] node (a1) [pos = .37] {} (d) to [out = 0, in = -45+180, looseness=1.1] (e') to [out = -45, in = 225+180, looseness=1.1] (f') to [out = 225, in = 0, looseness=1.1] (d') (c) to [out = 180, in = 45, looseness=1.1]  (h) to [out = 45 + 180, in = -225, looseness=1.1] (g) to [out = -225 + 180, in = 180, looseness=1.1] (b) to [out = 0, in = 45+180, looseness=1.1]  (h') to [out = 45, in = -225+180, looseness=1.1] (g') to [out = -225, in = 0, looseness=1.1] node (b1) [pos = .37] {} (b');
\draw[ultra thick] (5.5,1.5) to [out = 90+180, in = 225+180, looseness=1] (f') to [out = 225, in = 0, looseness=1.1] (d') (a) to [out = 180, in = -45, looseness=1.1] (e) to [out = -45 + 180, in = -90, looseness=1] (-5.5,1.5) (-1.18,1.55) to [out = -90, in = -225, looseness=1] (g)  to [out = -225 + 180, in = 180, looseness=1.1] (c') to [out = 0, in = 45+180, looseness=1.1]  (h') to [out = 45, in = -90, looseness=1] (1.18,1.45);
\fill[pattern = north east lines](5.5,1.5) to [out = 90+180, in = 225+180, looseness=1] (f') to [out = 225, in = 0, looseness=1.1] (d') -- (a) to [out = 180, in = -45, looseness=1.1] (e) to [out = -45 + 180, in = -90, looseness=1] (-5.5,1.5) -- (-1.18,1.55) to [out = -90, in = -225, looseness=1] (g)  to [out = -225 + 180, in = 180, looseness=1.1] (c') to [out = 0, in = 45+180, looseness=1.1]  (h') to [out = 45, in = -90, looseness=1] (1.18,1.45) -- cycle;\draw[->] (a1.center) -- ++ (180:0.01);
\draw (a) to [out = 180, in = -45, looseness=1.1] (e) to [out = -45 + 180, in = 225, looseness=1.1] (f) to [out = 225 + 180, in = 180, looseness=1.1] (d)
to [out = 0, in = -45+180, looseness=1.1] (e') to [out = -45, in = 225+180, looseness=1.1] (f') to [out = 225, in = 0, looseness=1.1] (d');
\draw(c) to [out = 180, in = 45, looseness=1.1]  (h) to [out = 45 + 180, in = -225, looseness=1.1] (g) to [out = -225 + 180, in = 180, looseness=1.1] (b)
to [out = 0, in = 45+180, looseness=1.1]  (h') to [out = 45, in = -225+180, looseness=1.1] (g') to [out = -225, in = 0, looseness=1.1] (b');
\draw[->] (b1.center) -- ++ (0:0.01);}$\;\! and a defect $D$, let $H_{\mathit{ann}}(D)$ denote the associated Hilbert space, considered as a `$\partial A$-sector', that is, a representation of the collection of algebras $\{D(I)\}$ for $I$ a subinterval of the boundary $\partial A$, subject to the following isotony and locality axioms:
\begin{equation}\label{eq:  orange star}
\begin{matrix}
\text{(isotony):} & \text{$I\subset J\Rightarrow \rho_{D(J)}|_{D(I)}=\rho_{D(I)}$,}
\\
\text{(locality):} & \text{$\mathring I\cap \mathring J=\emptyset \Rightarrow [\rho_{D(I)},\rho_{D(J)}]=0$.}
\end{matrix}
\end{equation}

Let $\Delta_D$ be the set of isomorphism classes of irreducible $D$-$D$-sectors, with the sector associated to $\lambda \in \Delta_D$ denoted $H_\lambda$.  Let $\bar \lambda$ denote the dual isomorphism class, and let $H_\lambda \otimes H_{\bar \lambda}$ denote the $\partial A$-sector where one circle acts on $H_\lambda$ and the other circle acts on $H_{\bar \lambda}$.
\begin{introthm}[Peter--Weyl for defects]
For a finite irreducible defect $D$, the annular sector $H_{\mathit{ann}}(D)$ is non-canonically isomorphic to the sum $\bigoplus_{\lambda\in\Delta_D} H_\lambda \otimes H_{\bar \lambda}$ of every sector tensor its dual.
\end{introthm}
\nid This is proven as Theorem~\ref{thm:KLM} in the text.  We may depict this result as
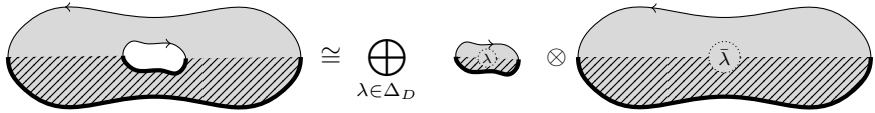
\[
\def\coords{
  \coordinate (a) at (0,0);
  \coordinate (b) at (.15,1);
  \coordinate (c) at (-.2,2);
  \coordinate (d) at (0,3);
  \coordinate (e) at (-5,.4);
  \coordinate (f) at (-5,2.6);
  \coordinate (g) at (-1,1.2);
  \coordinate (h) at (-1,2);
  \coordinate (a') at (d);
  \coordinate (b') at (c);
  \coordinate (c') at (b);
  \coordinate (d') at (a);
\begin{scope}[yshift = 85, rotate= 180]
  \coordinate (e') at (-5,.4);
  \coordinate (f') at (-5,2.6);
  \coordinate (g') at (-1,1.2);
  \coordinate (h') at (-1,2);
\end{scope}
}
\tikzmath[scale=.35]{\coordinate (a) at (0,0);\coordinate (b) at (.15,1);\coordinate (c) at (-.2,2);\coordinate (d) at (0,3);\coordinate (e) at (-5,.4);\coordinate (f) at (-5,2.6);\coordinate (g) at (-1,1.2);\coordinate (h) at (-1,2);\coordinate (a') at (d);\coordinate (b') at (c);\coordinate (c') at (b);\coordinate (d') at (a);\begin{scope}[yshift = 85, rotate= 180]\coordinate (e') at (-5,.4);\coordinate (f') at (-5,2.6);\coordinate (g') at (-1,1.2);\coordinate (h') at (-1,2);\end{scope}
\fill[fill = gray!30] (a) to [out = 180, in = -45, looseness=1.1] (e) to [out = -45 + 180, in = 225, looseness=1.1] (f) to [out = 225 + 180, in = 180, looseness=1.1] node (a1) [pos = .37] {} (d) to [out = 0, in = -45+180, looseness=1.1] (e') to [out = -45, in = 225+180, looseness=1.1] (f') to [out = 225, in = 0, looseness=1.1] (d') (c) to [out = 180, in = 45, looseness=1.1]  (h) to [out = 45 + 180, in = -225, looseness=1.1] (g) to [out = -225 + 180, in = 180, looseness=1.1] (b) to [out = 0, in = 45+180, looseness=1.1]  (h') to [out = 45, in = -225+180, looseness=1.1] (g') to [out = -225, in = 0, looseness=1.1] node (b1) [pos = .37] {} (b');
\draw[ultra thick] (5.5,1.5) to [out = 90+180, in = 225+180, looseness=1] (f') to [out = 225, in = 0, looseness=1.1] (d') (a) to [out = 180, in = -45, looseness=1.1] (e) to [out = -45 + 180, in = -90, looseness=1] (-5.5,1.5) (-1.18,1.55) to [out = -90, in = -225, looseness=1] (g)  to [out = -225 + 180, in = 180, looseness=1.1] (c') to [out = 0, in = 45+180, looseness=1.1]  (h') to [out = 45, in = -90, looseness=1] (1.18,1.45);
\fill[pattern = north east lines](5.5,1.5) to [out = 90+180, in = 225+180, looseness=1] (f') to [out = 225, in = 0, looseness=1.1] (d') -- (a) to [out = 180, in = -45, looseness=1.1] (e) to [out = -45 + 180, in = -90, looseness=1] (-5.5,1.5) -- (-1.18,1.55) to [out = -90, in = -225, looseness=1] (g)  to [out = -225 + 180, in = 180, looseness=1.1] (c') to [out = 0, in = 45+180, looseness=1.1]  (h') to [out = 45, in = -90, looseness=1] (1.18,1.45) -- cycle;\draw[->] (a1.center) -- ++ (180:0.01);
\draw (a) to [out = 180, in = -45, looseness=1.1] (e) to [out = -45 + 180, in = 225, looseness=1.1] (f) to [out = 225 + 180, in = 180, looseness=1.1] (d)
to [out = 0, in = -45+180, looseness=1.1] (e') to [out = -45, in = 225+180, looseness=1.1] (f') to [out = 225, in = 0, looseness=1.1] (d');
\draw(c) to [out = 180, in = 45, looseness=1.1]  (h) to [out = 45 + 180, in = -225, looseness=1.1] (g) to [out = -225 + 180, in = 180, looseness=1.1] (b)
to [out = 0, in = 45+180, looseness=1.1]  (h') to [out = 45, in = -225+180, looseness=1.1] (g') to [out = -225, in = 0, looseness=1.1] (b');
\draw[->] (b1.center) -- ++ (0:0.01);} % tikzmath
  \; \cong \;
  \bigoplus_{\lambda\in\Delta_D}\, \;
\tikzmath[scale=.35]{\coords
\useasboundingbox (-2,.4) rectangle (2,2.6);
\filldraw[fill = gray!30] 
	(c) to [out = 180, in = 45, looseness=1.1] 
	(h) to [out = 45 + 180, in = -225, looseness=1.1]
	(g) to [out = -225 + 180, in = 180, looseness=1.1]
	(c') to [out = 0, in = 45 + 180, looseness=1.1] 
	(h') to [out = 45, in = -225 + 180, looseness=1.1]
	(g') to [out = -225, in = 0, looseness=1.1]                                     node (b1) [pos = .7] {}
	(b');
\filldraw[ultra thick, pattern = north east lines] (-1.18,1.55) to [out = -90, in = -225, looseness=1] (g)  to [out = -225 + 180, in = 180, looseness=1.1] (c') to [out = 0, in = 45+180, looseness=1.1]  (h') to [out = 45, in = -90, looseness=1] (1.18,1.45); \draw[->] (b1.center) -- ++ (0:0.01);
\node[circle, fill = gray!30, inner sep=.5, draw, densely dotted] at (0,1.5)  {$\scriptscriptstyle\lambda$};
} % tikzmath
\otimes
\tikzmath[scale=.35]{
\coords
\filldraw[fill = gray!30]
	(a) to [out = 180, in = -45, looseness=1.1] 
	(e) to [out = -45 + 180, in = 225, looseness=1.1]
	(f) to [out = 225 + 180, in = 180, looseness=1.1] node (a1) [pos = .5] {}
	(a') to [out = 0, in = -45 + 180, looseness=1.1] 
	(e') to [out = -45, in = 225 + 180, looseness=1.1]
	(f') to [out = 225, in = 0, looseness=1.1] (d');
\draw[->] (a1.center) -- ++ (175:0.01);
\filldraw[ultra thick, pattern = north east lines] (5.5,1.5) to [out = 90+180, in = 225+180, looseness=1] (f') to [out = 225, in = 0, looseness=1.1] (a) to [out = 180, in = -45, looseness=1.1] (e) to [out = -45 + 180, in = -90, looseness=1] (-5.5,1.5);
\node[circle, fill = gray!30, inner sep=1.5, draw, densely dotted] at (0,1.5)  {$\scriptstyle\bar \lambda$};
} % tikzmath
\]

Equipped with this and other results about defect annular sectors, we proceed to our main topic of dualizability properties of conformal nets.  We show that finite sectors are dualizable; that finite defects are dualizable with finite unit and counit sectors (and hence are fully dualizable); and that finite conformal nets are dualizable with finite evaluation and coevaluation defects (and hence are fully dualizable).  Altogether this implies that the collection of finite conformal nets, finite defects, finite sectors, and intertwiners forms a sub-3-category of the 3-category of conformal nets, and establishes the following:
\begin{introthm}[Dualizability of finite nets, defects, and sectors] \label{introthm:dualizability}
The $3$-category of finite semisimple conformal nets, finite semisimple defects, finite sectors, and intertwiners has all duals.
\end{introthm}
\nid This result is summarized as Theorem~\ref{thm-duals} in the text, collecting the results of Proposition~\ref{prop:sectoradjoint}, Corollary~\ref{cor:sectordualizable}, Proposition~\ref{prop:defectadjoint}, Corollary~\ref{cor:defectdualizable}, Theorem~\ref{thm:finitenetdualizable}, and Corollary~\ref{cor:netdualizable}.

Having established that finiteness implies full dualizability, we conversely establish that full dualizability ensures finiteness.\footnote{Note that we do not have a 3-category of all not-necessarily-finite conformal nets (because we do not know that the composition of two defects between non-finite nets is again a defect); however the notion of dualizability is still well defined for an arbitrary not-necessarily-finite net (namely as the condition that the canonical evaluation and coevaluation defects both have ambidextrous adjoints with dualizable unit and counit sectors), and therefore it makes sense to claim and prove as we do that a dualizable net is finite.}
\begin{introthm}[Finiteness of dualizable nets, defects, and sectors] \label{introthm:dualizability -- other way}
A fully dualizable conformal net, defect, or sector is necessarily finite.
\end{introthm}
\nid See Corollary~\ref{cor:sectordualizable}, Proposition~\ref{prop:defectfinite}, Theorem~\ref{thm:netfinite}, and Scholium~\ref{end scholium} in the text for the precise statements and proofs.

\subsection*{Manifold invariants}
\addtocontents{toc}{\SkipTocEntry}

By Theorem~\ref{introthm:dualizability} and under the (overwhelmingly plausible but not yet proven) assumption that the cobordism hypothesis applies to the symmetric monoidal $3$-category of conformal nets constructed in~\cite{BDH(3-category)}\footnote{As the cobordism hypothesis applies most immediately to symmetric monoidal $n$-categories modeled as $\Gamma$-objects in complete $n$-fold Segal spaces~\cite{Lurie(On-classification-TFT),Calaque-Scheimbauer:Note-cobordism-category}, this assumption can be made precise in the form of the following conjecture: there exists a $\Gamma$-object in complete $3$-fold Segal spaces $\CN'$ together with an equivalence of tricategories $E: [\CN] \rightarrow [\CN']$;
%inducing an equivalence of tricategories $[(\CN)^\mathrm{fd}] \rightarrow [(\CN')^\mathrm{fd}]$; 
here $\CN$ denotes the symmetric monoidal 3-category of finite conformal nets constructed as an internal dicategory in symmetric monoidal categories~\cite{BDH(3-category)}, and the brackets $[-]$ denote the tricategory associated to either the internal dicategory in symmetric monoidal categories or the $\Gamma$-object in complete $3$-fold Segal spaces.\label{foot-ch}},
%, and the decoration $(-)^\mathrm{fd}$ denotes the full substructure (either internal dicategory or Segal $3$-category) on the dualizable objects.
%Henceforth whenever we refer to the field theory associated to a conformal net $\cN \in \CN$, we implicitly mean the field theory associated to an object $\widetilde{E(\cN)} \in \CN'$ covering $E(\cN) \in [\CN']$.
associated to any finite conformal net there is a $3$-dim\-en\-sion\-al local framed topological field theory whose value on a point is the conformal net.  Naturally, one wonders what manifold invariants are given by this topological field theory.  

For $1$-dimensional manifolds, the conformal net field theory invariants are given, projectively (that is, up to tensoring by an invertible von Neumann algebra), by the extension, constructed in~\cite[\thmnetsarbitrarymanifolds]{BDH(modularity)}, of the conformal net to a functor from $1$-manifolds to the category of von Neumann algebras.  In particular, the invariant of a circle is the direct sum over irreducible representations of the algebra of bounded operators on the underlying representation space (see~\cite[\thmnetoncircle]{BDH(modularity)}).  One may also express the invariant of a circle as the colimit in the category of von Neumann algebras of the value of the conformal net on all the subintervals of the circle (see~\cite[\propnetoncircle]{BDH(modularity)}).

For $2$-dimensional manifolds, the conformal net field theory invariants are given, projectively (that is, up to tensoring by an invertible Hilbert space), by the functor constructed in~\cite[\thmconformalblocks]{BDH(modularity)}, from $2$-manifolds to Hilbert spaces.  In particular, the invariant of a closed $2$-manifold is given, projectively, by the space of conformal blocks associated to that surface.

For any finite-index conformal net $\cala$, under the aforementioned assumption that the cobordism hypothesis applies, our results provide a complex-valued invariant $Z_\cala(M)$ of any closed framed $3$-manifold $M$.  When the conformal net is $\caln_{G,k}$, the one associated to a central extension of the loop group $LG$ (and assuming this net is indeed of finite-index), the category $\mathrm{Rep}(\caln_{G,k})$ of representations of the net is thought to be isomorphic to the category $\mathrm{Rep}(LG,k)$ of representations of the loop group $LG$ at level $k$; see~\cite{CS(pt)} for a discussion of this comparison problem and~\cite[Sec 5.1]{BinGui:CategoricalExtensionsOfConformalNets} for progress towards a solution.  Provided the representation categories of the conformal net and of the loop group are indeed isomorphic as modular tensor categories, then we expect the $3$-manifold invariant $Z_{\caln_{G,k}}(M)$ determined by the conformal-net-valued local field theory is the Reshetikhin--Turaev invariant of $M$ associated to the modular tensor category $\mathrm{Rep}(LG,k)$ of representations of the associated loop group.

\subsection*{Acknowledgments}
\addtocontents{toc}{\SkipTocEntry}

AB was funded in part by the DFG under Germany's Excellence Strategy EXC 2044-390685587.  CD was partially supported by a Miller Research Fellowship and by EPSRC grant EP/S018883/1.  AH was supported in part by grant VP2-2013-005 from the Leverhulme Trust.

\section{Defect algebras acting on annuli and discs}

We will, later in Section~\ref{sec:dualizablenets}, interpret the fusion of a defect and its adjoint as associating an algebra to an interval with not just a single transition point from white to black, but instead two: one from white to black, and then one back to white.  To construct the unit and counit of the adjunction, we will need an action of this larger algebra on the vacuum sector of the original defect.  We will construct such an action by first constructing an action on a Hilbert space associated to an annulus and then ``plugging the hole'' of the annulus with a vacuum sector.

Working up to those constructions, in this section we study the Hilbert space associated to a bicolored annulus; we prove a Peter--Weyl theorem decomposing the defect annular Hilbert space as a sum of tensor products of sectors and their duals, and we define the algebras associated to arbitrary bicolored $1$-manifolds.

\subsection{The Hilbert space for a bicolored annulus}

Given a finite defect $D$ between finite conformal nets,
the bimodule
\begin{equation}\label{eq: h0D is dualizable}
{}_{D(I_1)\bar\otimes D(I_3)}H_0(S,D)_{(\cala(I_2)\bar\otimes \calb(I_4))^\op}
\end{equation}
is always dualizable (see~\cite[\propKLMfinitenessdefects]{BDH(1*1)} and Footnote~\ref{foot-dualizable}).
Here $S$ is a bicolored circle decomposed into intervals $I_1,\ldots, I_4$, as in \eqref{eq: picture of I_1 ... I_4}, and the vacuum sector $H_0(S,D)$ is described in~\cite[\notationnoncanonvacuumdefects]{BDH(1*1)}.  Let $-S$, $-I_1,\ldots,-I_4$ be the same manifolds with the reverse orientations. The following result explicitly identifies the dual, generalizing the corresponding result for conformal nets~\cite[\lemdualvacuum]{BDH(nets)}:

\begin{lemma}\label{lem: H_0(-S) -- for defects}
Under the canonical identifications
$(D(-I_1)\,\bar\otimes\, D(-I_3))^\op\cong D(I_1)\,\bar\otimes\, D(I_3)$ and
$(\cala(-I_2)\,\bar\otimes\, \calb(-I_4))^\op\cong \cala(I_2)\,\bar\otimes\, \calb(I_4)$,
the dual of the bimodule \eqref{eq: h0D is dualizable} 
is given by \medskip
\[
{}_{\cala(-I_2)\bar\otimes \calb(-I_4)}H_0(-S,D)_{(D(-I_1)\bar\otimes D(-I_3))^\op}.
\]
\end{lemma}

\begin{proof}
We assume without loss of generality that $S$ is the standard bicolored circle. 
Let us write $S^1=K_1\cup\ldots \cup K_6$, with $K_1=I_4\cap S^1_\top$, $K_2=I_1$, $K_3=I_2\cap S^1_\top$, $K_4=I_2\cap S^1_\bot$, $K_5=I_3$, $K_6=I_4\cap S^1_\bot$
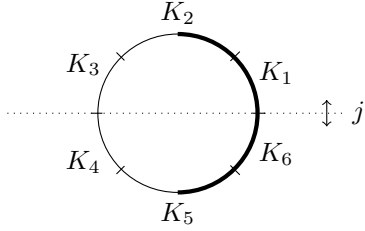
\begin{equation*}
\tikzmath[scale=.07]
{
\draw (0,0) circle (15);\draw[ultra thick] (0,-15) arc (-90:90:15);
\draw(45:14) -- (45:16.5)(-45:14) -- (-45:16.5)(0:14) -- (0:16.5)(180:14.15) -- (180:16.27)(135:14.15) -- (135:16.27)(-135:14.15) -- (-135:16.27)
(90:19) node {$K_2$}(152.5:20) node {$K_3$}(-152.5:20) node {$K_4$}(-90:19) node {$K_5$}(-22.5:20) node {$K_6$}(22.5:20) node {$K_1$};
\draw[dotted] (-32,0)--(31,0);\draw[<->] (28,2.5)--(28,-2.5);\node at (34,0) {$j$};}
\end{equation*} 
and let $j$ be the reflection
that exchanges $S^1_\top$ and $S^1_\bot$.
For any interval $I$, we abbreviate $D(j): D(I) \ra D(j(I))^\op$ by $j_*$ and let $A:=D(S^1_\top)$.
By definition, $H_0(S,D)=L^2(A)$ with actions
\begin{equation}\label{eq:aaxaa -- BIS}
(a_1\otimes a_2\otimes a_3\otimes a_4\otimes a_5\otimes a_6)\cdot\xi := (a_1a_2a_3)\,\xi\, j_*(a_4a_5a_6)^\op, \qquad a_i\in D(K_i).
\end{equation}
Here $a^\op \in A^\op$ is the element $a \in A$ viewed as an element of $A^\op$.  By~\cite[\corcontragredientbimodule]{BDH(Dualizability+Index-of-subfactors)}, the dual of $H_0$ is its complex conjugate $\overline{H_0(S,D)}=\overline{L^2(A)}$, 
with actions $b\cdot{\bar \xi}\cdot a= \overline{a^* \!\cdot\!\xi\!\cdot\! b^*}$ for 
$a\in D(I_1)\,\bar\otimes\, D(I_3)$ and $b\in (\cala(I_2)\,\bar\otimes\, \calb(I_4))^\op$.
We rewrite it as
\begin{equation}\label{eq:aaxaa1 -- BIS}
\begin{split}
(a_1\otimes a_3\otimes a_4\otimes a_6)\cdot\bar\xi\cdot (a_2^\op&\otimes a_5^\op)=\overline{(a_2^{\op*}\otimes a_5^{\op*})\cdot\xi\cdot (a_1^*\otimes a_3^*\otimes a_4^*\otimes a_6^*)}\\
&=\overline{(a_1^{\op*}\otimes a_2^{\op*}\otimes a_3^{\op*}\otimes a_4^{\op*}\otimes a_5^{\op*}\otimes a_6^{\op*})\cdot\xi}\\
&=\overline{(a_1^{\op*}a_2^{\op*}a_3^{\op*})\,\xi\, j_*(a_4^*a_5^*a_6^*)}
\end{split}
\end{equation}
for $a_i\in D(K_i)$.

On the other hand, $H_0(-S,D):=L^2(D(-S^1_\top))\cong L^2(A^\op)$ has actions 
$(b_1\otimes b_2\otimes b_3\otimes b_4\otimes b_5\otimes b_6)\cdot\eta := (b_1b_2b_3)\,\eta\, j_*(b_4b_5b_6)^\op$ for $b_i\in D(K_i)^\op$ and $\eta\in L^2(A^\op)$.
Using the canonical identification $\eta \mapsto \eta^\op$ between $L^2(A^\op)$ and $L^2(A)$ that exchanges the left $A^\op$-module structure with the right $A$-module structure
and the right $A^\op$-module structure with the left $A$-module structure, this becomes
\begin{equation}\label{eq:aaxaa2 -- BIS}
\quad[(a_1\otimes a_2\otimes a_3\otimes a_4\otimes a_5\otimes a_6)\cdot\xi]^\op = j_*(a_4a_5a_6)\,\xi^\op\,(a_1^\op a_2^\op a_3^\op)
\end{equation}
for $a_i\in D(K_i)$ and $\xi\in L^2(A)$.
Finally, the isomorphism intertwining \eqref{eq:aaxaa1 -- BIS} and \eqref{eq:aaxaa2 -- BIS}
is given by the modular conjugation $J:\overline{L^2(A)}\to L^2(A)$.
\end{proof}

%As in Section \ref{sec: Hilbert space for annulus}, 
We now investigate what happens when we glue two vacuum sectors along a pair of intervals.
Instead of viewing the vacuum sector $H_0(S,D)$ as being associated to a bicolored circle $S$ as in~\cite[\notationnoncanonvacuumdefects]{BDH(1*1)}, we shall think of it as being associated to a bicolored disk:
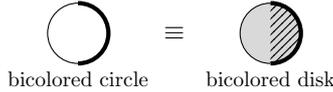
\[
\tikzmath[scale=.4]{
\useasboundingbox (-1.7,-2) rectangle (1.7,0);
\draw[ultra thick](0,0) arc (90:-90:1);
\draw(0,0) arc (90:270:1);
\node[scale=.8] at (0,-2.5) {bicolored circle};
}
\quad\equiv\quad
\tikzmath[scale=.4]{
\useasboundingbox (-1.7,-2) rectangle (1.7,0);
\fill[fill = gray!30] (0,-1) circle (1);
\fill[pattern = north east lines] (0,0) arc (90:-90:1) -- cycle; 
\draw[ultra thick](0,0) arc (90:-90:1);
\draw(0,0) arc (90:270:1);
\node[scale=.8] at (0,-2.5) {bicolored disk};
}\vspace{.5cm}
\]
This is merely a change of notation, not of content.  (Note that, as in~\cite{BDH(1*1)}, the Hilbert space $H_0(S,D)$ is only well defined up to non-canonical isomorphism.)

Given two genuinely bicolored disks $\ID_l$, $\ID_r$, we investigate two ways of gluing them together into annulus.
Decompose each of their boundaries into four intervals $S_l:=\partial \ID_l = I_1\cup \ldots \cup I_4$ and $S_r:=\partial \ID_r = I_5\cup\ldots \cup I_8$,
where $I_1$, $I_3$, $I_5$, $I_7$ are genuinely bicolored, $I_4$, $I_6$ are white, and $I_2$, $I_8$ are black.
If we glue $\ID_l$ to $\ID_r$ along diffeomorphisms $I_1\leftrightarrow I_5$ and $I_3\leftrightarrow I_7$, we get the following bicolored annulus:
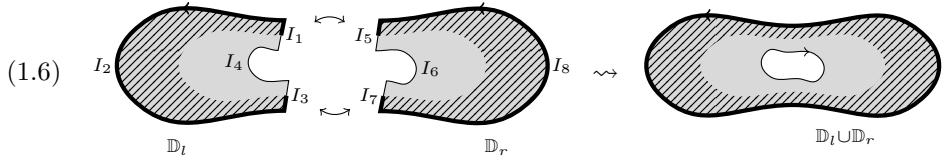
\begin{equation}\label{eq: bicolored annulus no1}
\tikzmath[scale=.4]{\coordinate (a) at (0,0);\coordinate (b) at (.15,1);\coordinate (c) at (-.2,2);\coordinate (d) at (0,3);\coordinate (e) at (-5,.4);\coordinate (f) at (-5,2.6);\coordinate (g) at (-1,1.2);\coordinate (h) at (-1,2);\begin{scope}[xshift = 90, yshift = 85, rotate= 180]\coordinate (a') at (0,0);\coordinate (b') at (.15,1);\coordinate (c') at (-.2,2);\coordinate (d') at (0,3);\coordinate (e') at (-5,.4);\coordinate (f') at (-5,2.6);\coordinate (g') at (-1,1.2);\coordinate (h') at (-1,2);\end{scope}
\fill[fill = gray!30] (b) to node (I3) [right, xshift = -2] {$\scriptstyle I_3$}(a) to [out = 180, in = -45, looseness=1.1] (e) to [out = -45 + 180, in = 225, looseness=1.1] node[left, xshift = 2] {$\scriptstyle I_2$}(f) to [out = 225 + 180, in = 180, looseness=1.1] node (a1) [pos = .37] {} (d) to [looseness=0] node (I1) [right, xshift = -2] {$\scriptstyle I_1$} (c) to [out = 180, in = 45, looseness=1.1] (h) to [out = 45 + 180, in = -225, looseness=1.1] node[left, xshift = 2] {$\scriptstyle I_4$} (g) to [out = -225 + 180, in = 180, looseness=1.1] (b);
\fill[pattern = north east lines] (a) to [out = 180, in = -45, looseness=1.1] (e) to [out = -45 + 180, in = 225, looseness=1.1] (f) to [out = 225 + 180, in = 180, looseness=1.1]
(d) to [looseness=0] ($(c)!.5!(d)$) to [out = 180, in = 45, looseness=1.1] ($(h)+(-2.2,.1)$) to [out = 45 + 180, in = -225, looseness=1.1]
($(g)+(-2.2,-.3)$) to [out = -225 + 180, in = 180, looseness=1.1] ($(a)!.5!(b)$) -- cycle;
\draw[ultra thick] ($(a)!.5!(b)$) -- (a) to [out = 180, in = -45, looseness=1.1] (e) to [out = -45 + 180, in = 225, looseness=1.1] (f) to [out = 225 + 180, in = 180, looseness=1.1] (d) -- ($(c)!.5!(d)$);
\draw($(c)!.5!(d)$) -- (c) to [out = 180, in = 45, looseness=1.1]  (h) to [out = 45 + 180, in = -225, looseness=1.1] (g) to [out = -225 + 180, in = 180, looseness=1.1] (b)--($(a)!.5!(b)$);
\draw[->, thick] (a1.center) -- ++ (180:0.01);
\fill[fill = gray!30] (b') to node (I5) [left, xshift = 2] {$\scriptstyle I_5$} (a') to [out = 0, in = -45 + 180, looseness=1.1] node (a1') [pos = .6] {} (e') to [out = -45, in = 225 + 180, looseness=1.1] node[right, xshift = -2] {$\scriptstyle I_8$} (f') to [out = 225, in = 0, looseness=1.1] (d') to [looseness=0] node (I7) [left, xshift = 2] {$\scriptstyle I_7$} (c') to [out = 0, in = 45 + 180, looseness=1.1] (h') to [out = 45, in = -225 + 180, looseness=1.1] node[right, xshift = -2] {$\scriptstyle I_6$} (g') to [out = -225, in = 0, looseness=1.1] (b');
\fill[pattern = north east lines] (a') to [out = 0, in = -45+180, looseness=1.1] (e') to [out = -45, in = 225+180, looseness=1.1] (f') to [out = 225, in = 0, looseness=1.1]
(d') to [looseness=0] ($(c')!.5!(d')$) to [out = 0, in = 45+180, looseness=1.1] ($(h')+(2.2,-.1)$) to [out = 45, in = -225+180, looseness=1.1]
($(g')+(2.2,.3)$) to [out = -225, in = 0, looseness=1.1] ($(a')!.5!(b')$) -- cycle;
\draw[ultra thick] ($(a')!.5!(b')$) -- (a') to [out = 0, in = -45+180, looseness=1.1] (e') to [out = -45, in = 225+180, looseness=1.1] (f') to [out = 225, in = 0, looseness=1.1] (d') -- ($(c')!.5!(d')$);
\draw($(c')!.5!(d')$) -- (c') to [out = 0, in = 45+180, looseness=1.1]  (h') to [out = 45, in = -225+180, looseness=1.1] (g') to [out = -225, in = 0, looseness=1.1] (b')--($(a')!.5!(b')$);
\draw[->, thick] (a1'.center) -- ++ (180:0.01); \node at (-3.5,-1.2) {$\scriptstyle \ID_l$}; \node at (7,-1.2)  {$\scriptstyle \ID_r$}; \draw (I1) edge [<->,bend left=35] (I5); \draw (I3) edge [<->,bend right=35] (I7);
} % tikzmath
\,\rightsquigarrow\,\,\,
\tikzmath[scale=.35]{\coordinate (a) at (0,0);\coordinate (b) at (.15,1);\coordinate (c) at (-.2,2);\coordinate (d) at (0,3);\coordinate (e) at (-5,.4);\coordinate (f) at (-5,2.6);\coordinate (g) at (-1,1.2);\coordinate (h) at (-1,2);\coordinate (a') at (d);\coordinate (b') at (c);\coordinate (c') at (b);\coordinate (d') at (a);\begin{scope}[yshift = 85, rotate= 180]\coordinate (e') at (-5,.4);\coordinate (f') at (-5,2.6);\coordinate (g') at (-1,1.2);\coordinate (h') at (-1,2);\end{scope}
\fill[fill = gray!30] (a) to [out = 180, in = -45, looseness=1.1] (e) to [out = -45 + 180, in = 225, looseness=1.1] (f) to [out = 225 + 180, in = 180, looseness=1.1] node (a1) [pos = .37] {} (d)
to [out = 0, in = -45+180, looseness=1.1] (e') to [out = -45, in = 225+180, looseness=1.1] (f') to [out = 225, in = 0, looseness=1.1] (d')
(c) to [out = 180, in = 45, looseness=1.1]  (h) to [out = 45 + 180, in = -225, looseness=1.1] (g) to [out = -225 + 180, in = 180, looseness=1.1] (b)
to [out = 0, in = 45+180, looseness=1.1]  (h') to [out = 45, in = -225+180, looseness=1.1] (g') to [out = -225, in = 0, looseness=1.1] node (b1) [pos = .37] {} (b');
\fill[pattern = north east lines] 
(a) to [out = 180, in = -45, looseness=1.1] (e) to [out = -45 + 180, in = 225, looseness=1.1] (f) to [out = 225 + 180, in = 180, looseness=1.1] node (a1) [pos = .37] {} (d)
to [out = 0, in = -45+180, looseness=1.1] (e') to [out = -45, in = 225+180, looseness=1.1] (f') to [out = 225, in = 0, looseness=1.1] (d')
($(c)!.5!(d)$) to [out = 180, in = 45, looseness=1.1]  ($(h)+(-2.2,.1)$) to [out = 45 + 180, in = -225, looseness=1.1] ($(g)+(-2.2,-.3)$) to [out = -225 + 180, in = 180, looseness=1.1] ($(a)!.5!(b)$)
to [out = 0, in = 45+180, looseness=1.1]  ($(h')+(2.2,-.1)$) to [out = 45, in = -225+180, looseness=1.1] ($(g')+(2.2,.3)$) to [out = -225, in = 0, looseness=1.1] ($(a')!.5!(b')$);
\draw[->, thick] (a1.center) -- ++ (180:0.01);
\draw[ultra thick] (a) to [out = 180, in = -45, looseness=1.1] (e) to [out = -45 + 180, in = 225, looseness=1.1] (f) to [out = 225 + 180, in = 180, looseness=1.1] (d)
to [out = 0, in = -45+180, looseness=1.1] (e') to [out = -45, in = 225+180, looseness=1.1] (f') to [out = 225, in = 0, looseness=1.1] (d');
\draw(c) to [out = 180, in = 45, looseness=1.1]  (h) to [out = 45 + 180, in = -225, looseness=1.1] (g) to [out = -225 + 180, in = 180, looseness=1.1] (b)
to [out = 0, in = 45+180, looseness=1.1]  (h') to [out = 45, in = -225+180, looseness=1.1] (g') to [out = -225, in = 0, looseness=1.1] (b');
\draw[->] (b1.center) -- ++ (0:0.01); \node at (2,-1.2)  {$\scriptstyle \ID_l\cup \ID_r$};
} % tikzmath
\end{equation}
If $D$ is a finite defect, then the action of $D(I_1)\otimes_\alg D(I_3)$ on $H_0(S_l,D)$ extends to the spatial tensor product $D(I_1)\,\bar\otimes\, D(I_3)$.
Similarly, the action of $D(I_5)\otimes_\alg D(I_7)$ on $H_0(S_r,D)$ extends to $D(I_5)\,\bar\otimes\, D(I_7)$.
Identifying $D(I_5)\bar\otimes D(I_7)$ with $D(I_1)\bar\otimes D(I_3)^\op$ via the diffeomorphism, we can then associate a Hilbert space to the annulus \eqref{eq: bicolored annulus no1} as follows:
\[
H_0(S_l,D)\!\underset{D(I_5)\bar\otimes D(I_7)}\boxtimes\! H_0(S_r,D)\,\,\,=\,\,\,
\tikzmath{
\node (a) at (0,0) {$
H_0(S_l,D)\underset{D(I_5)}\boxtimes H_0(S_r,D)\underset{D(I_3)}\boxtimes$};
\def\dd{.5}\def\ll{.35}\def\rr{.25}
\draw[dashed, rounded corners = 6] (a.east)++(0,.1) -- ++(\rr,0) -- ++(0,-\dd) -- ($(a.west) + (-\ll,-\dd) + (0,.1)$) -- ++(0,\dd) -- ++(\rr,0);} % tikzmath
\]

Consider now the slightly different situation where $I_2$, $I_4$, $I_6$, $I_8$ are genuinely bicolored, $I_1$, $I_5$ are white, and  $I_3$, $I_7$ are black.
Once again, we glue $\ID_l$ to $\ID_r$ along two diffeomorphisms $I_1\leftrightarrow I_5$ and $I_3\leftrightarrow I_7$
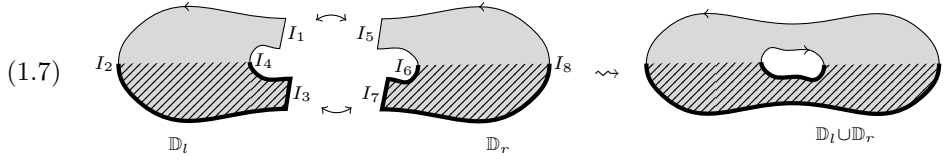
\begin{equation}\label{eq: bicolored annulus no2}
\tikzmath[scale=.4]{\coordinate (a) at (0,0);\coordinate (b) at (.15,1);\coordinate (c) at (-.2,2);\coordinate (d) at (0,3);\coordinate (e) at (-5,.4);\coordinate (f) at (-5,2.6);\coordinate (g) at (-1,1.2);\coordinate (h) at (-1,2);\begin{scope}[xshift = 90, yshift = 85, rotate= 180]\coordinate (a') at (0,0);\coordinate (b') at (.15,1);\coordinate (c') at (-.2,2);\coordinate (d') at (0,3);\coordinate (e') at (-5,.4);\coordinate (f') at (-5,2.6);\coordinate (g') at (-1,1.2);\coordinate (h') at (-1,2);\end{scope}
\fill[fill = gray!30] (b) to node (I3) [right, xshift = -2] {$\scriptstyle I_3$}(a) to [out = 180, in = -45, looseness=1.1] (e) to [out = -45 + 180, in = 225, looseness=1.1] node[left, xshift = 2] {$\scriptstyle I_2$}(f) to [out = 225 + 180, in = 180, looseness=1.1] node (a1) [pos = .37] {}(d) to [looseness=0] node (I1) [right, xshift = -2] {$\scriptstyle I_1$}(c) to [out = 180, in = 45, looseness=1.1] (h) to [out = 45 + 180, in = -225, looseness=1.1] node[right, xshift = -2] {$\scriptstyle I_4$}(g) to [out = -225 + 180, in = 180, looseness=1.1] (b);
\draw(-5.5,1.5) to [out = 90, in = 225, looseness=1] (f) to [out = 225 + 180, in = 180, looseness=1.1] (d) -- (c) to [out = 180, in = 45, looseness=1.1]  (h) 
to [out = 45 + 180, in = 90, looseness=1] (-1.18,1.55);
\filldraw[ultra thick, pattern = north east lines] (-1.18,1.55) to [out = -90, in = -225, looseness=1] (g)  to [out = -225 + 180, in = 180, looseness=1.1] (b) -- (a) to [out = 180, in = -45, looseness=1.1] (e)
to [out = -45 + 180, in = -90, looseness=1] (-5.5,1.5); \draw[->] (a1.center) -- ++ (180:0.01);
\fill[fill = gray!30] (b') to node (I5) [left, xshift = 2] {$\scriptstyle I_5$}(a') to [out = 0, in = -45 + 180, looseness=1.1] node (a1') [pos = .6] {}(e') to [out = -45, in = 225 + 180, looseness=1.1] node[right, xshift = -2] {$\scriptstyle I_8$}(f') to [out = 225, in = 0, looseness=1.1](d') to [looseness=0] node (I7) [left, xshift = 2] {$\scriptstyle I_7$}(c') to [out = 0, in = 45 + 180, looseness=1.1] (h') to [out = 45, in = -225 + 180, looseness=1.1] node[left, xshift = 2] {$\scriptstyle I_6$}(g') to [out = -225, in = 0, looseness=1.1] (b');
\draw (4.35,1.45) to [out = 90, in = -225+180, looseness=1] (g')  to [out = -225, in = 0, looseness=1.1] (b') -- (a') to [out = 0, in = -45+180, looseness=1.1] (e')
to [out = -45, in = 90, looseness=1] (8.67,1.5);\filldraw[ultra thick, pattern = north east lines] (8.67,1.5) to [out = 90+180, in = 225+180, looseness=1] (f') to [out = 225, in = 0, looseness=1.1] (d') -- (c') to [out = 0, in = 45+180, looseness=1.1]  (h') to [out = 45, in = -90, looseness=1] (4.35,1.45); \draw[->] (a1'.center) -- ++ (180:0.01);
\node at (-3.5,-1.2)  {$\scriptstyle \ID_l$};\node at (7,-1.2)  {$\scriptstyle \ID_r$};\draw (I1) edge [<->,bend left=35] (I5);\draw (I3) edge [<->,bend right=35] (I7);} % tikzmath
\,
\rightsquigarrow\,\,
\,
\tikzmath[scale=.35]{\coordinate (a) at (0,0);\coordinate (b) at (.15,1);\coordinate (c) at (-.2,2);\coordinate (d) at (0,3);\coordinate (e) at (-5,.4);\coordinate (f) at (-5,2.6);\coordinate (g) at (-1,1.2);\coordinate (h) at (-1,2);\coordinate (a') at (d);\coordinate (b') at (c);\coordinate (c') at (b);\coordinate (d') at (a);\begin{scope}[yshift = 85, rotate= 180]\coordinate (e') at (-5,.4);\coordinate (f') at (-5,2.6);\coordinate (g') at (-1,1.2);\coordinate (h') at (-1,2);\end{scope}
\fill[fill = gray!30] (a) to [out = 180, in = -45, looseness=1.1] (e) to [out = -45 + 180, in = 225, looseness=1.1] (f) to [out = 225 + 180, in = 180, looseness=1.1] node (a1) [pos = .37] {} (d) to [out = 0, in = -45+180, looseness=1.1] (e') to [out = -45, in = 225+180, looseness=1.1] (f') to [out = 225, in = 0, looseness=1.1] (d') (c) to [out = 180, in = 45, looseness=1.1]  (h) to [out = 45 + 180, in = -225, looseness=1.1] (g) to [out = -225 + 180, in = 180, looseness=1.1] (b) to [out = 0, in = 45+180, looseness=1.1]  (h') to [out = 45, in = -225+180, looseness=1.1] (g') to [out = -225, in = 0, looseness=1.1] node (b1) [pos = .37] {} (b');
\draw[ultra thick] (5.5,1.5) to [out = 90+180, in = 225+180, looseness=1] (f') to [out = 225, in = 0, looseness=1.1] (d') (a) to [out = 180, in = -45, looseness=1.1] (e) to [out = -45 + 180, in = -90, looseness=1] (-5.5,1.5) (-1.18,1.55) to [out = -90, in = -225, looseness=1] (g)  to [out = -225 + 180, in = 180, looseness=1.1] (c') to [out = 0, in = 45+180, looseness=1.1]  (h') to [out = 45, in = -90, looseness=1] (1.18,1.45);
\fill[pattern = north east lines](5.5,1.5) to [out = 90+180, in = 225+180, looseness=1] (f') to [out = 225, in = 0, looseness=1.1] (d') -- (a) to [out = 180, in = -45, looseness=1.1] (e) to [out = -45 + 180, in = -90, looseness=1] (-5.5,1.5) -- (-1.18,1.55) to [out = -90, in = -225, looseness=1] (g)  to [out = -225 + 180, in = 180, looseness=1.1] (c') to [out = 0, in = 45+180, looseness=1.1]  (h') to [out = 45, in = -90, looseness=1] (1.18,1.45) -- cycle;\draw[->] (a1.center) -- ++ (180:0.01);
\draw (a) to [out = 180, in = -45, looseness=1.1] (e) to [out = -45 + 180, in = 225, looseness=1.1] (f) to [out = 225 + 180, in = 180, looseness=1.1] (d)
to [out = 0, in = -45+180, looseness=1.1] (e') to [out = -45, in = 225+180, looseness=1.1] (f') to [out = 225, in = 0, looseness=1.1] (d');
\draw(c) to [out = 180, in = 45, looseness=1.1]  (h) to [out = 45 + 180, in = -225, looseness=1.1] (g) to [out = -225 + 180, in = 180, looseness=1.1] (b)
to [out = 0, in = 45+180, looseness=1.1]  (h') to [out = 45, in = -225+180, looseness=1.1] (g') to [out = -225, in = 0, looseness=1.1] (b');
\draw[->] (b1.center) -- ++ (0:0.01); \node at (2,-1.2)  {$\scriptstyle \ID_l\cup \ID_r$}; } % tikzmath
\end{equation}
and we associate a Hilbert space to the annulus: % \eqref{eq: bicolored annulus no2}:
\[
H_0(S_l,D)\!\underset{\cala(I_5)\bar\otimes \calb(I_7)}\boxtimes\! H_0(S_r,D)\,\,\,=\,\,\,
\tikzmath{
\node (a) at (0,0) {$
H_0(S_l,D)\underset{\cala(I_5)}\boxtimes H_0(S_r,D)\underset{\calb(I_3)}\boxtimes$};
\def\dd{.5}\def\ll{.35}\def\rr{.25}
\draw[dashed, rounded corners = 6] (a.east)++(0,.1) -- ++(\rr,0) -- ++(0,-\dd) -- ($(a.west) + (-\ll,-\dd) + (0,.1)$) -- ++(0,\dd) -- ++(\rr,0);} % tikzmath
\]

\begin{lemma}\label{lem: H_0 otimes H_0   <  H_Ann  --  with defects}
  Let ${}_\cala D_\calb$ be a finite irreducible defect, and let $S_l$, $S_r$, $I_1,\ldots, I_8$ be either as in \eqref{eq: bicolored annulus no1} or as in \eqref{eq: bicolored annulus no2}.
  Let also $S_b:=I_2\cup I_8$ and $S_m:=I_4\cup I_6$.
  Then $H_0(S_m,D) \otimes H_0(S_b,D)$ is a direct summand of 
  \[
  H_{\mathit{ann}}\,:=\,\,\,
  \tikzmath{
  \node (a) at (0,0) {$
  H_0(S_l,D)\underset{D(I_5)}\boxtimes H_0(S_r,D)\underset{D(I_3)}\boxtimes$};
  \def\dd{.5}\def\ll{.35}\def\rr{.25}
  \draw[dashed, rounded corners = 6] (a.east)++(0,.1) -- ++(\rr,0) -- ++(0,-\dd) -- ($(a.west) + (-\ll,-\dd) + (0,.1)$) -- ++(0,\dd) -- ++(\rr,0);} % tikzmath
  \smallskip\]
  in a way compatible with the actions of $D(J)$ for all $J\subset S_b$ and $J\subset S_m$.
  Moreover, $H_0(S_m,D) \otimes H_0(S_b,D)$ appears with multiplicity $1$ inside $H_{\mathit{ann}}$.  (In the case of situation~\eqref{eq: bicolored annulus no1}, by definition $H_0(S_m,D) = H_0(S_m,\cala)$ and $H_0(S_b,D) = H_0(S_b,\calb)$; in this case, we also require that $\cala$ and $\calb$ be irreducible.)
\end{lemma}

\begin{proof}
Let 
$A:=D(I_2)\,\bar\otimes\,D(I_4)$, $B:=(D(I_1)\,\bar\otimes\, D(I_3))^\op\cong D(I_5)\,\bar\otimes\, D(I_7)$, and $C:=(D(I_6)\,\bar\otimes\,D(I_8))^\op$,
and let us abbreviate
\[
H_l:=H_0(S_l,D),\quad H_r:=H_0(S_r,D),\quad H_b:=H_0(S_b,D),\quad H_m:=H_0(S_m,D).
\]
Since $I_2$, $I_4$, $I_6$, $I_8$ cover $S_m\cup S_b$ and $D$ is (and if needed $\cala$ and $\calb$ are) irreducible, the Hilbert space $H_m\otimes H_b$ is an irreducible $A$-$C$-bimodule.
%In order the show that $H_m\otimes H_b$ is a direct summand of $H_l\boxtimes_B H_r$,
%it is therefore enough to argue that
We need to show that
\begin{equation}\label{eq: hom_A,C}
\hom_{A,C}(H_m\otimes H_b, H_{\mathit{ann}})=
\hom_{A,C}(H_m\otimes H_b, H_l\boxtimes_B H_r)
\end{equation}
%is non-tivial.We will show that it 
is one dimensional.

Since $D$ is a finite defect, the bimodule ${}_A H_{lB}$ is dualizable. 
By Lemma \ref{lem: H_0(-S) -- for defects}, its dual is then $\check H_l:=H_0(-S_l)$.
By the fundamental property of duals (Frobenius reciprocity), we can therefore rewrite \eqref{eq: hom_A,C} as
\[
\hom_{B,C}\big(\check H_l\boxtimes_A (H_m\otimes H_b), H_r\big).
\]
By~\cite[\lemHKK]{BDH(modularity)} and~\cite[\lemvacvacvacdefects]{BDH(1*1)} $\check H_l\boxtimes_A (H_m\otimes H_b)$ is isomorphic to $H_0(S_r)$.
The above expression therefore reduces to $\hom_{B,C}(H_r, H_r)$, which is one dimensional by the irreducibility of the defect $D$.
\end{proof}

\subsection{A Peter--Weyl theorem for defects}

We now prove that there are finitely many isomorphism classes of irreducible $D$-$D$-sectors (also referred to simply as `$D$-sectors') for a finite defect $D$, and that every such irreducible sector is finite.  This is the analog for sectors between defects of the corresponding fact for representations of conformal nets, and the proof follows the structure of the proof for nets~\cite[Thm 3.14]{BDH(nets)}.

Let $S$ be a bicolored circle. Recall that an $S$-sector of $D$ is a Hilbert space equipped with actions of the algebras $D(I)$ for all bicolored subintervals $I$ of $S$, subject to the conditions~\eqref{eq:  orange star}.  As in~\cite[\S1.B]{BDH(nets)}, given a $D$-sector $K$ (on the standard bicolored circle) and a bicolored circle $S$, we write $K(S)$ for the $S$-sector $\varphi^*K$, where $\varphi:S\to S^1$ is any bicolored diffeomorphism from $S$ to the standard circle. % to $S$ determines (up to isomorphism) a corresponding $S$-sector of $D$, denoted $K(S)$.  
This sector is well defined up to non-canonical isomorphism, by the same argument as in the proof of ~\cite[Prop.\,1.14]{BDH(nets)}.

\begin{theorem}     \label{thm: KLM -- all irreducible sectors are finite : with defects}
Let ${}_\cala D_\calb$ be a finite irreducible defect between finite conformal nets. 
Then all $D$-sectors are direct sums of irreducible ones, and all irreducible $D$-sectors are finite.
Moreover, there are only finitely many isomorphism classes of irreducible $D$-sectors.
\end{theorem}

\begin{proof}
Let $S_l$, $S_r$, $S_b$, $S_m$ and $I_1,I_2,\ldots,I_8$ be as follows:
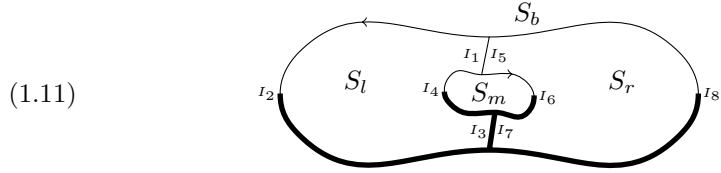
\begin{equation}\label{eq(1.15)}
\tikzmath[scale=.5]{\coordinate (a) at (0,0);\coordinate (b) at (.15,1);\coordinate (c) at (-.2,2);\coordinate (d) at (0,3);\coordinate (e) at (-5,.4);\coordinate (f) at (-5,2.6);\coordinate (g) at (-1,1.2);\coordinate (h) at (-1,2);\coordinate (a') at (d);\coordinate (b') at (c);\coordinate (c') at (b);\coordinate (d') at (a);\begin{scope}[yshift = 85, rotate= 180]\coordinate (e') at (-5,.4);\coordinate (f') at (-5,2.6);\coordinate (g') at (-1,1.2);\coordinate (h') at (-1,2);\end{scope}
\draw (b) to node (I3) [left, xshift = 2] {$\scriptscriptstyle I_3$} node (I7) [right, xshift = -2] {$\scriptscriptstyle I_7$} (a) to [out = 180, in = -45, looseness=1.1] (e) 
	to [out = -45 + 180, in = 225, looseness=1.1] node[left, xshift = 2] {$\scriptscriptstyle I_2$} (f) to [out = 225 + 180, in = 180, looseness=1.1] node (a1) [pos = .37] {}
	(d) to [looseness=0] node (I1) [left, xshift = 2] {$\scriptscriptstyle I_1$}node (I5) [right, xshift = -2] {$\scriptscriptstyle I_5$} (c) to [out = 180, in = 45, looseness=1.1] 
	(h) to [out = 45 + 180, in = -225, looseness=1.1] node[left, xshift = 2] {$\scriptscriptstyle I_4$} (g) to [out = -225 + 180, in = 180, looseness=1.1] (b);
\draw (a') to [out = 0, in = -45 + 180, looseness=1.1]node (a1') [pos = .6] {} (e') to [out = -45, in = 225 + 180, looseness=1.1] node[right, xshift = -2] {$\scriptscriptstyle I_8$} 
	(f') to [out = 225, in = 0, looseness=1.1] (d') (c') to [out = 0, in = 45 + 180, looseness=1.1] (h') to [out = 45, in = -225 + 180, looseness=1.1] node[right, xshift = -2] {$\scriptscriptstyle I_6$}
	(g') to [out = -225, in = 0, looseness=1.1] node (b1) [pos = .37] {} (b');
\draw[line width=2] (5.5,1.5) to [out = 90+180, in = 225+180, looseness=1] (f') to [out = 225, in = 0, looseness=1.1] (d') (a) to [out = 180, in = -45, looseness=1.1] (e) to [out = -45 + 180, in = -90, looseness=1] (-5.5,1.5) (-1.18,1.55) to [out = -90, in = -225, looseness=1] (g)  to [out = -225 + 180, in = 180, looseness=1.1] (c') to [out = 0, in = 45+180, looseness=1.1]  (h') to [out = 45, in = -90, looseness=1] (1.18,1.45) (a) -- (b);
\node at (-3.5,1.8) {$S_l$};
\node at (3.5,1.8) {$S_r$};
\node at (0,1.5) {$S_m$};
\node at (1,3.6) {$S_b$};
\draw[->] (a1.center) -- ++ (180:0.01); \draw[->] (b1.center) -- ++ (0:0.01);
} % tikzmath
\end{equation}
and let $H_l=H_0(S_l,D)$, $H_r=H_0(S_r,D)$, $H_b=H_0(S_b,D)$, $H_m=H_0(S_m,D)$, and
$H_{\mathit{ann}}:=\,\,\tikzmath{
\node (a) at (0,0) {$
H_l\,\boxtimes_{\cala(I_5)} H_r\,\boxtimes_{\calb(I_3)}$};
\def\dd{.45}\def\ll{.25}\def\rr{.15}
\draw[dashed, rounded corners = 5] (a.east)++(0,.1) -- ++(\rr,0) -- ++(0,-\dd) -- ($(a.west) + (-\ll,-\dd) + (0,.1)$) -- ++(0,\dd) -- ++(\rr,0);} % tikzmath
\smallskip$\,.
Let also
\begin{gather*}
A:=D(I_2\cup I_4),\quad B:=(\cala(I_1)\,\bar\otimes\,\calb(I_3))^\op\cong\cala(I_5)\,\bar\otimes\,\calb(I_7),\quad C:=D(I_6\cup I_8)^\op,
\\
A_l:=D(I_2),\quad A_m:=D(I_4)^\op,\quad C_m:=D(I_6)^\op,\quad C_r:=D(I_8).
\end{gather*}
Since ${}_AH_l{}_B$ and ${}_BH_r{}_C$ are dualizable bimodules,
$H_{\mathit{ann}}=H_l\boxtimes_B H_r$ is dualizable as an $A$-$C$-bimodule.
It therefore splits into finitely many irreducible summands~\cite[Lemma 4.10]{BDH(Dualizability+Index-of-subfactors)}.

Let us now consider $H_{\mathit{ann}}$ with its actions of $D(I)$ for $I\in \INT_{\circ\bullet}^{S_b}$. 
The von Neumann algebra generated by those algebras on $H_{\mathit{ann}}$ has a finite dimensional center,
since otherwise would contradict the fact that ${}_A H_{\mathit{ann}}{}_C$ splits into finitely many irreducible summands.
We can thus write $H_{\mathit{ann}}$ as a direct sum of finitely many 
factorial $S_b$-sectors of $D$:
\begin{equation}\label{eq: decomposition of Hann}
H_{\mathit{ann}} = K_1(S_b)\oplus\ldots\oplus K_n(S_b).
\end{equation}
Here $K_1, \ldots, K_n$ are $D$-sectors, which we transfer to $S_b$ by means of an arbitrary diffeomorphism $S^1\cong S_b$.  (As in the situation without defects~\cite[Sect 3.2]{BDH(nets)}, a sector is called \emph{factorial} if its endomorphism algebra is a factor.)

Given an arbitrary factorial sector $K$, we now show that there exists a $K_i$ in the above list to which $K$ is stably isomorphic, i.e., such that
$K\otimes \ell^2\cong K_i\otimes \ell^2$.
Let us introduce the bicolored circles $S_2:=I_2\cup_{\partial I_2} -I_2$ and $S_4:=I_4\cup_{\partial I_4} -I_4$. 
We have isomorphisms
\(
K(S_2)\boxtimes_{A_l} H_l \cong K(S_l) \cong H_l \boxtimes_{A_m} K(S_4)
\)
of $S_l$-sectors (constructed as in~\cite[\lemvacvacvacdefects]{BDH(1*1)}).
Fusing with $H_r$ over $B$, we get an isomorphism
\[
K(S_2)\boxtimes_{A_l}H_{\mathit{ann}}\cong H_{\mathit{ann}} \boxtimes_{A_m} K(S_4).
\]
By Lemma \ref{lem: H_0 otimes H_0   <  H_Ann  --  with defects}, it follows that
\begin{align*}
K(S_b)\otimes H_m &\cong 
(K(S_2)\boxtimes_{A_l}H_b)\otimes H_m \\
&\cong K(S_2)\boxtimes_{A_l}(H_b\otimes H_m) \\
&\subset K(S_2) \boxtimes_{A_l} H_{\mathit{ann}}\\
&\cong (K(S_2) \boxtimes_{A_l} H_l) \boxtimes_{D(I_5) \otimes D(I_3)^\op} H_r \\
& \cong (H_l \boxtimes_{A_l} K(S_4)) \boxtimes_{D(I_5) \otimes D(I_3)^\op} H_r
\cong
H_{\mathit{ann}}\boxtimes_{A_m} K(S_4).
\end{align*}
%Let us now forget the $\{\cala(I)\}_{I\subset S_m}$ actions, and only remember the structure of $\{\cala(I)\}_{I\subset S_b}$-representation.
Since $D$ is irreducible, $A_m$ is a factor, so $K(S_4)$ and $L^2A_m$ are stably isomorphic as $A_m$-modules, and we get the following (non-canonical) inclusion of $S_b$-sectors of $D$:
\[
\begin{split}
K(S_b) \otimes \ell^2 \cong
K(S_b) \otimes H_m \otimes \ell^2 &\subset
H_{\mathit{ann}} \boxtimes_{A_m} K(S_4) \otimes \ell^2 \\ & \cong 
H_{\mathit{ann}} \boxtimes_{A_m} L^2A_m \otimes \ell^2 \cong
H_{\mathit{ann}} \otimes \ell^2,
\end{split}
\]
where the first equality is induced by an arbitrary Hibert space isomorphism 
$\ell^2 \cong H_m \otimes \ell^2$.
The sector $K(S_b)$ is factorial. 
It therefore maps to a single summand $K_i\otimes \ell^2$ of $H_{\mathit{ann}} \otimes \ell^2$.
It follows that $K$ and $K_i$ are stably isomorphic.
In particular, this shows that there are at most finitely many stable isomorphism classes of factorial $D$-sectors on $S_b$.

By Lemma~\ref{lem:disintegratingdefects}, %~\cite[\lemmaDsectorsdisintegrate]{BDH(1*1)}, 
any $D$-sector can be disintegrated into irreducible ones.
As a consequence, if there existed a factorial 
sector of type~{\it II} or~{\it III}, 
then (as in~\cite[Cor. 58]{\KLM}) there would be uncountably many non-isomorphic irreducible sectors.
This is impossible, and so all factorial sectors must be of type~$I$.

We now show all irreducible $D$-sectors are finite.  Let us go back to $H_{\mathit{ann}}$ and analyze it as a $\{D(I)\}_{I\in \INT^{S_b}_{\circ\bullet}}$ - $\{D(I)\}_{I\in \INT^{S_m}_{\circ\bullet}}$~-representation.
Since each summand $K_i(S_b)$ in the decomposition \eqref{eq: decomposition of Hann} is a type~$I$ factorial $D$-sector,
we can write it as $L_i\otimes M_i$, where $L_i$ is an irreducible representation of $\{D(I)\}_{I\in \INT^{S_b}_{\circ\bullet}}$,
and the multiplicity space $M_i$ carries a residual action of
$\{D(I)\}_{I\in \INT^{S_m}_{\circ\bullet}}$.
The decomposition \eqref{eq: decomposition of Hann} then becomes
\[
{}_{A_l\,\bar\otimes\,A_m} (H_{\mathit{ann}})\, {}_{C_r\,\bar\otimes\,C_m} 
\,\,\cong\,\,\,
\bigoplus_i
{}_{A_l} L_i\, {}_{C_r} \otimes {}_{A_m} M_i\, {}_{C_m} .
\]
Since $H_{\mathit{ann}}$ is a dualizable $A$-$C$-bimodule, the bimodules ${}_{A_l} L_i {}_{C_r}$ must also be dualizable.
This finishes the argument, as any irreducible $D$-sector on $S_b$ is 
isomorphic to one of the $L_i$.
\end{proof}

Given a finite irreducible defect $D$, let $\Delta_D$ be the finite set of isomorphism classes of irreducible $D$-sectors.
For every $\lambda\in \Delta_D$, we pick a representative $H_\lambda$ of the isomorphism class,
which we draw as follows:
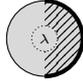
\[
\tikzmath[scale=.5]{
\filldraw[fill = gray!30] (0,0) circle (1);
\filldraw[pattern = north east lines, ultra thick] (0,1) arc (90:-90:1); 
\node[circle, fill = gray!30, inner sep=1.5, draw, densely dotted] at (0,0)  {$\scriptscriptstyle\lambda$};} % tikzmath
\]
The set $\Delta_D$ has an involution $\lambda\mapsto \bar\lambda$ given by sending a Hilbert space $H_\lambda$ to its complex conjugate $H_{\bar\lambda}\cong \overline H_\lambda$,
with actions of $D(I)$ given by 
\[
a\overline \xi := \overline{\cala(j)(a^*)\xi},
\]
where $j:S^1\to S^1$ is the reflection in the horizontal axis (which is color preserving).
Note that the isomorphism $H_{\bar\lambda}\cong \overline H_\lambda$ is by no means canonical---see the discussion in \cite[Rem.\,2.4.2]{Bakalov-Kirillov(Lect-tens-cat+mod-func)}.

The following Peter--Weyl theorem for defects is analogous to a corresponding annular-sector decomposition theorem for conformal nets by Kawahigashi--Longo--M\"uger~\cite[Thm 9]{\KLM}, cf also~\cite[Thm 3.23]{BDH(nets)}:

\begin{theorem}
  \label{thm:KLM}
  Let $D$ be a finite irreducible defect, let $S_l$, $S_r$, $S_m$, $S_b$ be as in \eqref{eq(1.15)}, and let
  \[
  H_{\mathit{ann}}\,:=\,\,\,
  \tikzmath{
  \node (a) at (0,0) {$
  H_0(S_l,D)\underset{\cala(I_5)}\boxtimes H_0(S_r,D)\underset{\calb(I_3)}\boxtimes$};
  \def\dd{.5}\def\ll{.35}\def\rr{.25}
  \draw[dashed, rounded corners = 6] (a.east)++(0,.1) -- ++(\rr,0) -- ++(0,-\dd) -- ($(a.west) + (-\ll,-\dd) + (0,.1)$) -- ++(0,\dd) -- ++(\rr,0);} % tikzmath
  \smallskip\]
  We then have a non-canonical isomorphism
  \begin{equation} \label{eq:KLM -- defects}
    H_{\mathit{ann}} \,\cong\,\, \bigoplus_{\lambda\in\Delta_D} H_\lambda(S_m) \otimes H_{\bar \lambda}(S_b)
  \end{equation}
  of $(S_m\sqcup S_b)$-sectors. 
  %$\{D(I)\}_{I\in\INT_{\circ\bullet}^{S_m}\cup\INT_{\circ\bullet}^{S_b}}$-representations.
  We draw this isomorphism as
\[
\def\coords{
  \coordinate (a) at (0,0);
  \coordinate (b) at (.15,1);
  \coordinate (c) at (-.2,2);
  \coordinate (d) at (0,3);
  \coordinate (e) at (-5,.4);
  \coordinate (f) at (-5,2.6);
  \coordinate (g) at (-1,1.2);
  \coordinate (h) at (-1,2);
  \coordinate (a') at (d);
  \coordinate (b') at (c);
  \coordinate (c') at (b);
  \coordinate (d') at (a);
\begin{scope}[yshift = 85, rotate= 180]
  \coordinate (e') at (-5,.4);
  \coordinate (f') at (-5,2.6);
  \coordinate (g') at (-1,1.2);
  \coordinate (h') at (-1,2);
\end{scope}
}
\tikzmath[scale=.35]{\coordinate (a) at (0,0);\coordinate (b) at (.15,1);\coordinate (c) at (-.2,2);\coordinate (d) at (0,3);\coordinate (e) at (-5,.4);\coordinate (f) at (-5,2.6);\coordinate (g) at (-1,1.2);\coordinate (h) at (-1,2);\coordinate (a') at (d);\coordinate (b') at (c);\coordinate (c') at (b);\coordinate (d') at (a);\begin{scope}[yshift = 85, rotate= 180]\coordinate (e') at (-5,.4);\coordinate (f') at (-5,2.6);\coordinate (g') at (-1,1.2);\coordinate (h') at (-1,2);\end{scope}
\fill[fill = gray!30] (a) to [out = 180, in = -45, looseness=1.1] (e) to [out = -45 + 180, in = 225, looseness=1.1] (f) to [out = 225 + 180, in = 180, looseness=1.1] node (a1) [pos = .37] {} (d) to [out = 0, in = -45+180, looseness=1.1] (e') to [out = -45, in = 225+180, looseness=1.1] (f') to [out = 225, in = 0, looseness=1.1] (d') (c) to [out = 180, in = 45, looseness=1.1]  (h) to [out = 45 + 180, in = -225, looseness=1.1] (g) to [out = -225 + 180, in = 180, looseness=1.1] (b) to [out = 0, in = 45+180, looseness=1.1]  (h') to [out = 45, in = -225+180, looseness=1.1] (g') to [out = -225, in = 0, looseness=1.1] node (b1) [pos = .37] {} (b');
\draw[ultra thick] (5.5,1.5) to [out = 90+180, in = 225+180, looseness=1] (f') to [out = 225, in = 0, looseness=1.1] (d') (a) to [out = 180, in = -45, looseness=1.1] (e) to [out = -45 + 180, in = -90, looseness=1] (-5.5,1.5) (-1.18,1.55) to [out = -90, in = -225, looseness=1] (g)  to [out = -225 + 180, in = 180, looseness=1.1] (c') to [out = 0, in = 45+180, looseness=1.1]  (h') to [out = 45, in = -90, looseness=1] (1.18,1.45);
\fill[pattern = north east lines](5.5,1.5) to [out = 90+180, in = 225+180, looseness=1] (f') to [out = 225, in = 0, looseness=1.1] (d') -- (a) to [out = 180, in = -45, looseness=1.1] (e) to [out = -45 + 180, in = -90, looseness=1] (-5.5,1.5) -- (-1.18,1.55) to [out = -90, in = -225, looseness=1] (g)  to [out = -225 + 180, in = 180, looseness=1.1] (c') to [out = 0, in = 45+180, looseness=1.1]  (h') to [out = 45, in = -90, looseness=1] (1.18,1.45) -- cycle;\draw[->] (a1.center) -- ++ (180:0.01);
\draw (a) to [out = 180, in = -45, looseness=1.1] (e) to [out = -45 + 180, in = 225, looseness=1.1] (f) to [out = 225 + 180, in = 180, looseness=1.1] (d)
to [out = 0, in = -45+180, looseness=1.1] (e') to [out = -45, in = 225+180, looseness=1.1] (f') to [out = 225, in = 0, looseness=1.1] (d');
\draw(c) to [out = 180, in = 45, looseness=1.1]  (h) to [out = 45 + 180, in = -225, looseness=1.1] (g) to [out = -225 + 180, in = 180, looseness=1.1] (b)
to [out = 0, in = 45+180, looseness=1.1]  (h') to [out = 45, in = -225+180, looseness=1.1] (g') to [out = -225, in = 0, looseness=1.1] (b');
\draw[->] (b1.center) -- ++ (0:0.01);} % tikzmath
  \; \xrightarrow{\cong} \;
  \bigoplus_{\lambda\in\Delta_D}\, \;
\tikzmath[scale=.35]{\coords
\useasboundingbox (-2,.4) rectangle (2,2.6);
\filldraw[fill = gray!30] 
	(c) to [out = 180, in = 45, looseness=1.1] 
	(h) to [out = 45 + 180, in = -225, looseness=1.1]
	(g) to [out = -225 + 180, in = 180, looseness=1.1]
	(c') to [out = 0, in = 45 + 180, looseness=1.1] 
	(h') to [out = 45, in = -225 + 180, looseness=1.1]
	(g') to [out = -225, in = 0, looseness=1.1]                                     node (b1) [pos = .7] {}
	(b');
\filldraw[ultra thick, pattern = north east lines] (-1.18,1.55) to [out = -90, in = -225, looseness=1] (g)  to [out = -225 + 180, in = 180, looseness=1.1] (c') to [out = 0, in = 45+180, looseness=1.1]  (h') to [out = 45, in = -90, looseness=1] (1.18,1.45); \draw[->] (b1.center) -- ++ (0:0.01);
\node[circle, fill = gray!30, inner sep=.5, draw, densely dotted] at (0,1.5)  {$\scriptscriptstyle\lambda$};
} % tikzmath
\otimes
\tikzmath[scale=.35]{
\coords
\filldraw[fill = gray!30]
	(a) to [out = 180, in = -45, looseness=1.1] 
	(e) to [out = -45 + 180, in = 225, looseness=1.1]
	(f) to [out = 225 + 180, in = 180, looseness=1.1] node (a1) [pos = .5] {}
	(a') to [out = 0, in = -45 + 180, looseness=1.1] 
	(e') to [out = -45, in = 225 + 180, looseness=1.1]
	(f') to [out = 225, in = 0, looseness=1.1] (d');
\draw[->] (a1.center) -- ++ (175:0.01);
\filldraw[ultra thick, pattern = north east lines] (5.5,1.5) to [out = 90+180, in = 225+180, looseness=1] (f') to [out = 225, in = 0, looseness=1.1] (a) to [out = 180, in = -45, looseness=1.1] (e) to [out = -45 + 180, in = -90, looseness=1] (-5.5,1.5);
\node[circle, fill = gray!30, inner sep=1.5, draw, densely dotted] at (0,1.5)  {$\scriptstyle\bar \lambda$};
} % tikzmath
\]
\end{theorem}

\begin{proof}
Let $H_l$, $H_r$, $A$, $A_l$, $A_m$, $B$, $C$, $C_m$, $C_r$
be as in the proof %of Lemma \ref{lem: H_0 otimes H_0   <  H_Ann  --  with defects} and 
of Theorem \ref{thm: KLM -- all irreducible sectors are finite : with defects},
and let $\check H_l:=H_0(-S_l)$ be the dual bimodule to ${}_A H_{l\,B}$ (see Lemma \ref{lem: H_0(-S) -- for defects}).

The Hilbert space $H_{\mathit{ann}} = H_l\boxtimes_B H_r$ is a finite $A$-$C$-bimodule and therefore splits into finitely many irreducible summands.
By the argument in the proof of Theorem \ref{thm: KLM -- all irreducible sectors are finite : with defects},
each irreducible summand is the tensor product of an
irreducible $D$-sector on $S_m$ and an irreducible $D$-sector on $S_b$.
So we can write $H_{\mathit{ann}}$ as a direct sum
\begin{equation*}
H_{\mathit{ann}} \cong \bigoplus_{\lambda,\mu\in \Delta_D} N_{\lambda\mu}\, H_{\lambda}(S_m) \otimes H_{\mu}(S_b)
\end{equation*}
with finite multiplicities $N_{\lambda\mu}\in\IN$.

Given $\lambda,\mu\in \Delta_D$, we now compute $N_{\lambda\mu}$. 
Let $K$ be the vertical fusion of $H_\lambda$ and $H_\mu$.
By slight abuse of notation, we abbreviate $H_\lambda:=H_\lambda(S_m)$, $H_{\mu}:=H_{\mu}(S_b)$, and $K:=K(S_r)$.
We then have
\[\begin{split}
\hom_{A,C}\big(H_\lambda \otimes H_{\mu},H_{\mathit{ann}}\big)
=\,\,&\hom_{A,C}\big(H_\lambda \otimes H_{\mu},H_l\boxtimes_B H_r\big)\\
=\,\,&\hom_{B,C}\big(\check H_l\boxtimes_A(H_\lambda \otimes H_{\mu}),H_r\big)\\
=\,\,&\hom_{B,C}\big(H_\lambda\boxtimes_{A_m^\op}\check H_l\boxtimes_{A_l^{\phantom{\op}}}\!\! H_{\mu},H_r\big)\\
=\,\,&\hom_{B,C}\big(K,H_r\big)\\
=\,\,&\begin{cases}
\IC&\text{if $\mu=\bar\lambda$}\\
0&\text{otherwise.}\\
\end{cases}
\end{split}\]
If follows that $N_{\lambda\mu}=\delta_{\bar\lambda\mu}$.
\end{proof}

\begin{remark}
The isomorphism \eqref{eq:KLM -- defects} is non-canonical.
Actually, it doesn't even make sense to ask whether or not it is canonical since the right hand side of the equation is only well defined up to non-canonical isomorphism.
\end{remark}

\begin{corollary}\label{cor: KLM for defects + fill in the hole}
Let $S_l$, $S_r$, $S_b$, $S_m$ and $I_1,I_2,\ldots,I_8$ be as in \eqref{eq(1.15)}.
Then the algebra generated by $D(I_4)$ and $D(I_6)$ on $H_{\mathit{ann}}$
%\[
%H_{\mathit{ann}}:=\,\,\tikzmath{
%\node (a) at (0,0) {$
%H_0(S_l,D)\underset{\cala(I_5)}\boxtimes H_0(S_r,D)\underset{\calb(I_3)}\boxtimes$};
%  \def\dd{.5}\def\ll{.35}\def\rr{.25}
%\draw[dashed, rounded corners = 5] (a.east)++(0,.1) -- ++(\rr,0) -- ++(0,-\dd) -- ($(a.west) + (-\ll,-\dd) + (0,.1)$) -- ++(0,\dd) -- ++(\rr,0);} % tikzmath
%\]
is canonically isomorphic to $\bigoplus_{\lambda\in\Delta_D} \bfB(H_\lambda(S_m,D))$.
Moreover, there is a non-canonical isomorphism
\begin{equation}\label{eq: fill the whole -- defect}
H_{\mathit{ann}}\,\boxtimes_{(D(I_4)\vee D(I_6))^\op} H_0(S_m,D)\,\,\cong\,\, H_0(S_b,D)
\end{equation}
which we represent as follows:
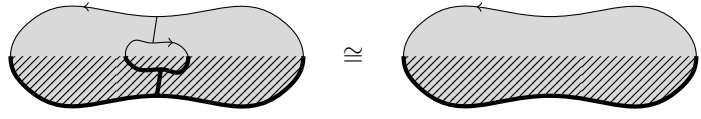
\[
\def\coords{\coordinate (a) at (0,0);\coordinate (b) at (.15,1);\coordinate (c) at (-.2,2);\coordinate (d) at (0,3);\coordinate (e) at (-5,.4);\coordinate (f) at (-5,2.6);\coordinate (g) at (-1,1.2);\coordinate (h) at (-1,2);\coordinate (a') at (d);\coordinate (b') at (c);\coordinate (c') at (b);\coordinate (d') at (a);
\begin{scope}[yshift = 85, rotate= 180]\coordinate (e') at (-5,.4);\coordinate (f') at (-5,2.6);\coordinate (g') at (-1,1.2);\coordinate (h') at (-1,2);
\end{scope}
}
\tikzmath[scale=.35]{\coords
\fill[fill = gray!30]
	(a) to [out = 180, in = -45, looseness=1.1] 
	(e) to [out = -45 + 180, in = 225, looseness=1.1]
	(f) to [out = 225 + 180, in = 180, looseness=1.1] node (a1) [pos = .5] {}
	(a') to [out = 0, in = -45 + 180, looseness=1.1] 
	(e') to [out = -45, in = 225 + 180, looseness=1.1]
	(f') to [out = 225, in = 0, looseness=1.1] (d');
\draw[ultra thick] (5.5,1.5) to [out = 90+180, in = 225+180, looseness=1] (f') to [out = 225, in = 0, looseness=1.1] (d') -- (c')(a) to [out = 180, in = -45, looseness=1.1] (e) to [out = -45 + 180, in = -90, looseness=1] (-5.5,1.5) (-1.18,1.5) to [out = -90, in = -225, looseness=1] (g)  to [out = -225 + 180, in = 180, looseness=1.1] (c') to [out = 0, in = 45+180, looseness=1.1]  (h') to [out = 45, in = -90, looseness=1] (1.18,1.5);
\fill[pattern = north east lines](5.5,1.5) to [out = 90+180, in = 225+180, looseness=1] (f') to [out = 225, in = 0, looseness=1.1] (d') -- (a) to [out = 180, in = -45, looseness=1.1] (e) to [out = -45 + 180, in = -90, looseness=1] (-5.5,1.5) -- cycle;
\draw (a) to [out = 180, in = -45, looseness=1.1] (e) to [out = -45 + 180, in = 225, looseness=1.1] (f) to [out = 225 + 180, in = 180, looseness=1.1] (d)
to [out = 0, in = -45+180, looseness=1.1] (e') to [out = -45, in = 225+180, looseness=1.1] (f') to [out = 225, in = 0, looseness=1.1] (d') (a') -- ($(b')+(.05,0)$);
\draw(c) to [out = 180, in = 45, looseness=1.1]  (h) to [out = 45 + 180, in = -225, looseness=1.1] (g) to [out = -225 + 180, in = 180, looseness=1.1] (b)
to [out = 0, in = 45+180, looseness=1.1]  (h') to [out = 45, in = -225+180, looseness=1.1] (g') to [out = -225, in = 0, looseness=1.1] node (b1) [pos = .37] {} (b');
\draw[->] (a1.center) -- ++ (180:0.01); \draw[->] (b1.center) -- ++ (-3:0.01);
} % tikzmath
\quad\cong\quad
\tikzmath[scale=.35]{
\coords
\filldraw[fill = gray!30]
	(a) to [out = 180, in = -45, looseness=1.1] 
	(e) to [out = -45 + 180, in = 225, looseness=1.1]
	(f) to [out = 225 + 180, in = 180, looseness=1.1] node (a1) [pos = .5] {}
	(a') to [out = 0, in = -45 + 180, looseness=1.1] 
	(e') to [out = -45, in = 225 + 180, looseness=1.1]
	(f') to [out = 225, in = 0, looseness=1.1] (d');
\draw[->] (a1.center) -- ++ (175:0.01);
\filldraw[ultra thick, pattern = north east lines] (5.5,1.5) to [out = 90+180, in = 225+180, looseness=1] (f') to [out = 225, in = 0, looseness=1.1] (a) to [out = 180, in = -45, looseness=1.1] (e) to [out = -45 + 180, in = -90, looseness=1] (-5.5,1.5);
} % tikzmath
\]
\end{corollary}

\subsection{Extending defects to bicolored $1$-manifolds}

In~\cite[\thmnetsarbitrarymanifolds]{BDH(modularity)}, we extended the domain of definition of a conformal net from the category of intervals to the category of all compact $1$-manifolds (where the morphisms are embeddings that are either orientation preserving or orientation reversing).  In~\cite[\eqdefectsdisconnected]{BDH(1*1)}, we extended a defect to take values on disjoint unions of intervals.  We now further extend a defect to all compact bicolored $1$-manifolds, with an arbitrary number of color-change points.  This extension will be useful when we construct the unit and counit sectors for adjunctions of defects, because the composite of a defect and its adjoint can be naturally reexpressed as the value of the defect on an interval with two color-change points.

\begin{definition} A \emph{bicolored $1$-manifold} is a compact $1$-manifold $M$ (always oriented), possibly with boundary,
equipped with two compact submanifolds $M_\circ, M_\bullet\subset M$ such that $M_\circ\cap M_\bullet$ consists of finitely many points.
Moreover, each point of $M_\circ\cap M_\bullet$ should be equipped with a local coordinate 
$(-\e,\e) \hookrightarrow M$ that sends $(-\e,0]$ to $M_{\circ}$ and $[0,\e)$ to $M_{\bullet}$.
\end{definition}

Given a bicolored $1$-manifold $M$, we pick a decomposition $M=M_0\cup M_1$ such that 
$P:=M_0\cap M_1$ has finitely many points, none of which is a color-change point.
Every connected component of $M_0$ and $M_1$ should be an interval, and should contain at most one color-change point.
Pick local coordinates around $P$, and define $N_i:=(M_i\times\{1\})\cup Q\subset M\times [0,1]$, where $Q:=P\times [0,1]$ inherits its bicoloring from $P$.
The manifolds $N_i$ and $Q$ are oriented so as to make the inclusions $M_i\rightarrow N_i$ and $Q\rightarrow N_1$ orientation preserving;
the inclusion $Q\rightarrow N_0$ is then orientation reversing.
The local coordinates around $P$ induce a smooth structure on $N_i$.
As in~\cite[\eqdefectsdisconnected]{BDH(1*1)}, we define the defect on a disjoint union of bicolored intervals by $D(I_1 \cup \ldots \cup I_n) := D(I_1) \bar\otimes \ldots \bar\otimes D(I_n)$.  We then define the defect on any bicolored $1$-manifold as follows.

\begin{definition}\label{def: D on all 1-manifolds}
Given a defect $D$ and a bicolored $1$-manifold $M$, we define the value of $D$ on $M$ to be
\begin{equation}\label{eq: def of D(M)}
D(M)\,:=\, D(N_0)\circledast_{D(Q)} D(N_1).
\end{equation}
\end{definition}
\nid (See~\cite[Sec 1.E \& App B.IV]{BDH(1*1)} for discussion and the definition of the relative fusion product $\circledast$ of von Neumann algebras.)

In~\cite[Cor. 1.13]{BDH(modularity)}, we showed that the value of a conformal net on a $1$-manifold was independent of the choice of decomposition used in the definition; the same argument generalizes to the situation here, showing that the algebra \eqref{eq: def of D(M)} is independent (up to canonical isomorphism) of the choice of decomposition~${M=M_1\cup M_2}$.

Here is an example of the above definition:
\[
\def\coords{  
  \coordinate (a) at (.4,0.1);\coordinate (ab) at (1.3,0.03);
  \coordinate (b) at (2,0);\coordinate (bc) at (3.05,.8);
  \coordinate (c) at (2.5,1.5);\coordinate (cd) at (1,1.45);
  \coordinate (d) at (-0.1,1.2);\coordinate (da) at (-.5,.5);
  
  \coordinate (e) at (4.3,0);\coordinate (ef) at (5.3,.6);
  \coordinate (f) at (4.8,1.3);\coordinate (fe) at (3.7,.5);

  \coordinate (g) at (5.6,0);\coordinate (gh) at (5.9,0.01);
  \coordinate (h) at (6.2,.05);\coordinate (hi) at (6.6,0.3);
  \coordinate (i) at (7,.5);
  }
\def\AB{  \draw (a) to[out = 10, in = 170, looseness=1.1] (ab);\draw[ultra thick] (ab) to[out = -10, in = 185, looseness=1.1] (b); }
\def\ABpegs{  \draw[rounded corners=1] ($(a)+(0,-.7)$) -- (a) to[out = 10, in = 170, looseness=1.1] (ab);\draw[ultra thick, rounded corners=1] (ab) to[out = -10, in = 185, looseness=1.1] (b) -- ($(b)+(0,-.7)$);}
\def\BC{  \draw[ultra thick] (b) to[out = 5, in = -105, looseness=1] (bc);\draw (bc) to[out = 75, in = -10, looseness=1.2] (c);}
\def\BCpegs{  \draw[ultra thick, rounded corners=1] ($(b)+(0,-.7)$) -- (b) to[out = 5, in = -105, looseness=1] (bc);\draw[rounded corners=1] (bc) to[out = 75, in = -10, looseness=1.2] (c) -- ($(c)+(0,-.7)$);}
\def\CD{  \draw (c) to[out = 170, in = 15, looseness=1] (cd); \draw[ultra thick] (cd) to[out = 195, in = 10, looseness=1] (d);}
\def\CDpegs{  \draw[rounded corners=1] ($(c)+(0,-.7)$) -- ($(c)+(0,-.7)$) --  (c) to[out = 170, in = 15, looseness=1] (cd); \draw[ultra thick, rounded corners=1] (cd) to[out = 195, in = 10, looseness=1] (d) -- ($(d)+(0,-.7)$);}
\def\DA{  \draw[ultra thick] (d) to[out = 190, in = 110, looseness=1.2] (da);  \draw (da) to[out = -70, in = 190, looseness=1.1] (a);}
\def\DApegs{  \draw[ultra thick, rounded corners=.5] ($(d)+(0,-.7)$) -- (d) to[out = 190, in = 110, looseness=1.2] (da);  \draw[rounded corners=1] (da) to[out = -70, in = 190, looseness=1.1] (a) -- ($(a)+(0,-.7)$);}
\def\EF{  \draw (e) to[out = 0, in = -110, looseness=1] (ef); \draw[ultra thick] (ef) to[out = 70, in = 10, looseness=1.5] (f);}
\def\EFpegs{  \draw[rounded corners=1] ($(e)+(0,-.7)$) -- (e) to[out = 0, in = -110, looseness=1] (ef); \draw[ultra thick, rounded corners=.5] (ef) to[out = 70, in = 10, looseness=1.5] (f) -- ($(f)+(0,-.7)$);}
\def\FE{  \draw[ultra thick] (f) to[out = 190, in = 75, looseness=.9] (fe); \draw (fe) to[out = -105, in = 180, looseness=1.1] (e);}
\def\FEpegs{  \draw[ultra thick, rounded corners=1] ($(f)+(0,-.7)$) -- (f) to[out = 190, in = 75, looseness=.9] (fe); \draw[rounded corners=1] (fe) to[out = -105, in = 180, looseness=1.1] (e) -- ($(e)+(0,-.7)$);}
\def\GH{  \draw[ultra thick] (g) to[out = 0, in = 185, looseness=1] (gh); \draw (gh) to[out = 5, in = 190, looseness=1] (h);}
\def\GHpegs{  \draw[ultra thick] (g) to[out = 0, in = 185, looseness=1] (gh); \draw[rounded corners=1] (gh) to[out = 5, in = 190, looseness=1] (h) -- ($(h)+(0,-.7)$);}
\def\HI {  \draw (h) to[out = 10, in = -130, looseness=.7] (hi); \draw[ultra thick] (hi) to[out = 50, in = 180, looseness=1] (i);}
\def\HIpegs {  \draw[rounded corners=1] ($(h)+(0,-.7)$) -- (h) to[out = 10, in = -130, looseness=.7] (hi); \draw[ultra thick] (hi) to[out = 50, in = 180, looseness=1] (i);}
  D\bigg(
  \tikzmath[scale=.35]
  {\useasboundingbox (-.65,0.1) rectangle (7.05,1.2);
  \coords
  \AB\BC\CD\DA\EF\FE\GH\HI  } % tikzmath
  \bigg):=
  D\bigg(
  \tikzmath[scale=.35]
  {\useasboundingbox (-.7,0) rectangle (6.7,1.2);
  \coords
  \BCpegs\DApegs\FEpegs\GHpegs} % tikzmath
  \bigg)
  \underset{
  D\big(
  \tikzmath[scale=.175]
  {\useasboundingbox (-.7,-.5) rectangle (6.7,1.5);
  \coords
  \foreach \x in {(a),(c), (e), (h)} \draw \x -- +(0,-.7);
  \foreach \x in {(b),(d),(f)} \draw[line width = 1] \x -- +(0,-.7);} % tikzmath
  \big)}{\circledast}
  D\bigg(
  \tikzmath[scale=.35]
  {\useasboundingbox (-.35,0) rectangle (7.1,1.2);
  \coords
  \ABpegs\CDpegs\EFpegs\HIpegs} % tikzmath
  \bigg)
\]

\noindent 
In Section~\ref{sec:finitenetsdualizable}, this extension of a defect to take values on all bicolored $1$-manifolds will allow a computationally convenient expression for the composite of a defect and its dual.

\begin{proposition}\label{prop: tricolored interval acts on H_0(D)}
Let ${}_\cala D_\calb$ be a finite defect.
Let $S^1$ be the standard bicolored circle, let $I\subset S^1$ be the following bicolored manifold
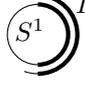
\[
\tikzmath[scale=.4]{
%\fill[fill = gray!30] (0,0) circle (1);
%\fill[pattern = north east lines] (0,1) arc (90:-90:1) -- cycle; 
\draw[ultra thick](0,1) arc (90:-90:1);
\draw[ultra thick](0,1.3) arc (90:-90:1.3);
\draw(0,1) arc (90:270:1);
\draw (110:1.3) arc (110:-110:1.3);
\draw(35:1.8) node {$I$};
\draw(150:.3) node {$S^1$};
}
\]
and let $D(I)$ be as in \ref{def: D on all 1-manifolds}.
Then the natural action of $D(I\cap S^1_\top)\otimes_\alg D(I\cap S^1_\bot)$ on $H_0(D)$
extends to a normal (that is, ultraweakly continuous) action of $D(I)$.
\end{proposition}

\begin{proof}
We first address the case when $D$ is irreducible.
By definition, the algebra $D(I)$ acts (normally) on
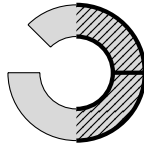
\[
\tikzmath[scale=.06]
{
\fill[fill = gray!30] (-180:8) arc (-180:135:8) -- (135:15) arc (135:-180:15) -- cycle;
\draw[ultra thick] (90:8) arc (90:-90:8) (90:15) arc (90:-90:15) (8,0)--(15,0);
\draw (90:8) arc (90:135:8) (90:15) arc (90:135:15) (135:8)--(135:15);
\draw (-90:8) arc (-90:-180:8) (-90:15) arc (-90:-180:15) (-180:8)--(-180:15);
\fill[pattern = north east lines] (-90:8) arc (-90:90:8) -- (90:15) arc (90:-90:15) -- cycle;
}
\]
Fusing in \,$
\tikzmath[scale=.05]
{
\useasboundingbox (180:15) rectangle ++(9,6);
\filldraw[fill = gray!30] (180:8) arc (180:135:8) -- (135:15) arc (135:180:15) -- cycle;
}
$\,, we can use the fact that a vacuum sector of a conformal net fuses with a vacuum sector of a defect to a vacuum sector of the defect~\cite[\lemvacvacvacdefects]{BDH(1*1)} and the fact that cyclic fusion is cyclically invariant~\cite[App. A]{BDH(modularity)}
to see that $D(I)$ also acts on 
\[
H_{\mathit{ann}}\,\,:=\,\,\,\tikzmath[scale=.06]
{
\fill[fill = gray!30] (-180:8) arc (-180:180:8) -- (180:15) arc (180:-180:15) -- cycle;
\draw[ultra thick] (90:8) arc (90:-90:8) (90:15) arc (90:-90:15) (8,0)--(15,0);
\draw (90:8) arc (90:180:8) (90:15) arc (90:180:15) (180:8)--(180:15);
\draw (-90:8) arc (-90:-180:8) (-90:15) arc (-90:-180:15);
\fill[pattern = north east lines] (-90:8) arc (-90:90:8) -- (90:15) arc (90:-90:15) -- cycle;
}\,\,.
\]
By Corollary \ref{cor: KLM for defects + fill in the hole}, the algebra generated by $D\big(\,\tikzmath[scale=.04]{
\useasboundingbox (-8,-4) rectangle (8,8);
\draw[ultra thick](0:8) arc (0:90:8);\draw(90:8) arc (90:180:8);}\,\big)$
and $D\big(\,\tikzmath[scale=.04]{\useasboundingbox (-8,-8) rectangle (8,4);\draw[ultra thick](0:8) arc (0:-90:8);\draw(-90:8) arc (-90:-180:8);}\,\big)$
in $\bfB(H_{\mathit{ann}})$ admits a natural right action on 
$\tikzmath[scale=.04]{\fill[fill = gray!30] circle (8); \filldraw[ultra thick, pattern = north east lines](90:8) arc (90:-90:8);\draw(90:8) arc (90:-270:8);}$\,.
Since the action of $D(I)$ on $H_{\mathit{ann}}$ commutes with that of 
$D\big(\,\tikzmath[scale=.04]{
\useasboundingbox (-8,-4) rectangle (8,8);
\draw[ultra thick](0:8) arc (0:90:8);\draw(90:8) arc (90:180:8);}\,\big)\vee
D\big(\,\tikzmath[scale=.04]{\useasboundingbox (-8,-8) rectangle (8,4);\draw[ultra thick](0:8) arc (0:-90:8);\draw(-90:8) arc (-90:-180:8);}\,\big)$,
the algebra $D(I)$ also acts on
\[
\tikzmath[scale=.06]
{
\fill[fill = gray!30] (-180:8) arc (-180:180:8) -- (180:15) arc (180:-180:15) -- cycle;
\draw[ultra thick] (90:8) arc (90:-90:8) (90:15) arc (90:-90:15) (8,0)--(15,0);
\draw (90:8) arc (90:180:8) (90:15) arc (90:180:15) (180:8)--(180:15);
\draw (-90:8) arc (-90:-180:8) (-90:15) arc (-90:-180:15);
\fill[pattern = north east lines] (-90:8) arc (-90:90:8) -- (90:15) arc (90:-90:15) -- cycle;
}\,\,\,\underset{\big(D\big(\,\tikzmath[scale=.03]{
\useasboundingbox (-8,-4) rectangle (8,8);
\draw[ultra thick](0:8) arc (0:90:8);\draw(90:8) arc (90:180:8);}\,\big)\vee
D\big(\,\tikzmath[scale=.03]{\useasboundingbox (-8,-8) rectangle (8,4);\draw[ultra thick](0:8) arc (0:-90:8);\draw(-90:8) arc (-90:-180:8);}\,\big)\big)^\op}
\boxtimes\,\,\,
\tikzmath[scale=.04]{\fill[fill = gray!30] circle (8); \filldraw[ultra thick, pattern = north east lines](90:8) arc (90:-90:8);\draw(90:8) arc (90:-270:8);}
\]
By \eqref{eq: fill the whole -- defect}, the latter is isomorphic to $H_0(D)$.

When $D$ is not irreducible, write it as a sum $D_1\oplus\ldots\oplus D_n$ of irreducible defects.
We then have $H_0(D)=\bigoplus_i H_0(D_i)$, and
\[
\begin{split}
D(I) =
D\left(\,\,\,
\tikzmath[scale=.3]{
\draw[ultra thick](0,1) arc (90:-90:1);
\draw (110:1) arc (110:-110:1);
}\right)
&=\,
D\left(\,\,\,
\tikzmath[scale=.3, baseline=-8.8]{
\draw[ultra thick](0,1) arc (90:0:1) -- (.5,0);
\draw (110:1) arc (110:0:1);
}\right)
\underset{\calb(\tikz[baseline=-1.5]{\draw[ultra thick](0,0) -- (.18,0);})}
\circledast
D\left(\,\,\,
\tikzmath[scale=.3, baseline=10]{
\draw[ultra thick](0,-1) arc (-90:0:1) -- (.5,0);
\draw (-110:1) arc (-110:0:1);
}\right)
\\
&=\,
\bigoplus_{i,j}\,\,
D_i\left(\,\,\,
\tikzmath[scale=.3, baseline=-8.8]{
\draw[ultra thick](0,1) arc (90:0:1) -- (.5,0);
\draw (110:1) arc (110:0:1);
}\right)
\underset{\calb(\tikz[baseline=-1.5]{\draw[ultra thick](0,0) -- (.18,0);})}
\circledast
D_j\left(\,\,\,
\tikzmath[scale=.3, baseline=10]{
\draw[ultra thick](0,-1) arc (-90:0:1) -- (.5,0);
\draw (-110:1) arc (-110:0:1);
}\right)
\end{split}
\]
The subalgebra $D_i(I\cap S^1_\top)\otimes_{\alg}D_j(I\cap S^1_\bot)\subset D(I)$
acts as zero on $H_0(D_k)$ unless $i=j=k$, in which case the first part of the proof applies and it extends to a normal action of $D_i(I)$ on $H_0(D_i)$.
Thus the action of $D(I\cap S^1_\top)\otimes_{\alg}D(I\cap S^1_\bot)=\bigoplus_{i,j}D_i(I\cap S^1_\top)\otimes_{\alg}D_j(I\cap S^1_\bot)$
on $\bigoplus H_0(D_i)$ extends to a normal action of $D(I)$.
\end{proof}

\section{A characterization of dualizable conformal nets}\label{sec:dualizablenets}

\subsection{Involutions on nets, defects, sectors, and intertwiners}\label{sec: star structures}

The $3$-category $\CN$ is equipped with four antilinear involutions 
\begin{tikzpicture}{\useasboundingbox (-.12,-.12)rectangle(.12,.18);\node {${}^*$};}\end{tikzpicture},
\begin{tikzpicture}{\useasboundingbox (-.12,-.12)rectangle(.12,.18);\node {$\bar{}$};}\end{tikzpicture}, 
\begin{tikzpicture}{\useasboundingbox (-.12,-.12)rectangle(.12,.18);\node {${}^\dagger$};}\end{tikzpicture}, 
\begin{tikzpicture}{\useasboundingbox (-.15,-.12)rectangle(.15,.18);\node {${}^\op$};}\end{tikzpicture},
where the $i$th involution is contravariant at the level of $(4-i)$-morphisms, and covariant at all other levels.  The second and third involutions will provide adjoints for finite sectors and defects respectively, and the fourth involution will provide the dual of a conformal net---that the involutions do indeed give adjoints, respectively duals, is proven in Section~\ref{sec:finitenetsdualizable}.

The first involution 
\begin{tikzpicture}{\useasboundingbox (-.12,-.12)rectangle(.12,.18);\node {${}^*$};}\end{tikzpicture}
acts trivially on the $0$, $1$, and $2$-morphisms, and sends a $3$-morphism $f:H\to K$ to its adjoint $f^*:K\to H$ (in the sense of maps between Hilbert spaces).

The second one
\begin{tikzpicture}{\useasboundingbox (-.12,-.12)rectangle(.12,.18);\node {$\bar{}$};}\end{tikzpicture}
acts trivially on $0$ and on $1$-morphisms.
It sends a $D$-$E$-sector $(H,\{\rho_I\})$, where the homomorphisms $\rho_I$ are given by
\begin{alignat*}{7}
&\rho_I \colon \cala(I) \to \bfB(H)\,\,\,
&\text{for}\quad I \in \INT_{\!S^1\!,\circ}\qquad\quad
&\rho_I \colon D(I)\to \bfB(H)\,\,\, 
&&\text{for}\quad I \in \INT_{\!S^1\!,\top}\\
&\rho_I \colon \calb(I)\to \bfB(H)\,\,\, 
&\text{for}\quad I \in \INT_{\!S^1\!,\bullet}\qquad\quad
&\rho_I \colon E(I)\to \bfB(H)\,\,\,
&&\text{for}\quad I \in \INT_{\!S^1\!,\bot}
\end{alignat*}  
to the complex conjugate Hilbert space $\bar H$
and $E$-$D$-sector structure given by
\begin{equation}\label{eq:  pencil star}
\begin{matrix}
\bar\rho_I \colon \cala(I) \to \bfB(\bar H)\,\,\,
\text{for}\,\,\, I \in \INT_{\!S^1\!,\circ}\qquad
\bar\rho_I \colon E(I)\to \bfB(\bar H)\,\,\, 
\text{for}\,\,\, I \in \INT_{\!S^1\!,\top}\\[1.4mm]
\bar\rho_I \colon \calb(I)\to \bfB(\bar H)\,\,\, 
\text{for}\,\,\, I \in \INT_{\!S^1\!,\bullet}\qquad
\bar\rho_I \colon D(I)\to \bfB(\bar H)\,\,\,
\text{for}\,\,\, I \in \INT_{\!S^1\!,\bot}
\end{matrix}
\end{equation}
where $\bar \rho_I (a):=\rho_{j(I)}(j_*(a^*))$,
and $j:z\mapsto \bar z$ is the reflection in the horizontal axis.
Here, $j_*$ stands for either $\cala(j)$, $E(j)$, $\calb(j)$, or $D(j)$.
The involution \begin{tikzpicture}{\useasboundingbox (-.12,-.12)rectangle(.12,.18);\node {$\bar{}$};}\end{tikzpicture} sends a $3$-morphism $f:H\to K$ to its complex conjugate $\bar f:\bar H\to \bar K$.

The third involution
\begin{tikzpicture}{\useasboundingbox (-.12,-.12)rectangle(.12,.18);\node {${}^\dagger$};}\end{tikzpicture}
acts trivially on objects.
Given a bicolored interval $I$, let $I^{\rev}$ denote the same interval with reversed bicoloring, that is, $(I^{\rev})_\circ = I_\bullet$ and $(I^{\rev})_\bullet = I_\circ$.
The orientation of $I^{\rev}$ is the same as that of $I$, but the local coordinate is negated.
The reversed defect of ${}_\cala D_\calb$ is the defect ${}_\calb D^{\dagger}\!\!{}_\cala$ defined by $D^{\dagger}(I) = D(I^{\rev})$.
For a $D$-$E$-sector $H$, the corresponding $D^\dagger$-$E^\dagger$-sector $H^\dagger$ is the complex conjugate of $H$, with structure maps
\begin{alignat*}{7}
&\rho_I^\dagger \colon \calb(I) &\,\to \bfB(H^\dagger)\,\,\,
&\text{for}\,\,\, I \in \INT_{\!S^1\!,\circ}\qquad\quad
&\rho_I^\dagger \colon D^\dagger(I)&\to \bfB(H^\dagger)\,\,\, 
&&\text{for}\,\,\, I \in \INT_{\!S^1\!,\top}\\
&\rho_I^\dagger \colon \cala(I)&\,\to \bfB(H^\dagger)\,\,\, 
&\text{for}\,\,\, I \in \INT_{\!S^1\!,\bullet}\qquad\quad
&\rho_I^\dagger \colon E^\dagger(I)&\to \bfB(H^\dagger)\,\,\,
&&\text{for}\,\,\, I \in \INT_{\!S^1\!,\bot}
\end{alignat*}
given by $\rho^\dagger_I (a)=\rho_{r(I)}(r_*(a^*))$, where $r:z\mapsto -\bar z$ is now the vertical reflection.
$3$-morphisms are sent to their complex conjugates.

The fourth involution
\begin{tikzpicture}{\useasboundingbox (-.15,-.12)rectangle(.15,.18);\node {${}^\op$};}\end{tikzpicture}
sends $\cala\in\CN$ to the a conformal net $\cala^{\op}(I):=\cala(I)^\op$.
Similarly, it sends a morphism ${}_\cala D_\calb$ to the $\cala^\op$-$\calb^\op$-defect $D^\op(I):=D(I)^\op$.
A $D$-$E$-sector $(H,\{\rho_I\})$ is sent to the complex conjugate Hilbert space, with actions $\rho_I^\op(a^\op):=\rho_I(a^*)$.
Finally, $3$-morphisms go to their complex conjugates.

\begin{remark}\label{remark-ambi}
The existence of these four involutions ensures that any duality or adjunction in $\CN$ is automatically ambidextrous, that is, it is both a left and a right duality or adjunction.  (When we say `$X$ has ambidextrous adjoint (or dual) $Y$', we mean that $Y$ admits both the structure of a left and the structure of a right adjoint (or dual) to $X$.)
\end{remark}

\subsection{The snake interchange isomorphism for defects}

To establish, in the next section, that the reversed defect ${}_\calb D^{\dagger}\!\!{}_\cala$ is an (ambidextrous) adjoint of the defect ${}_\cala D_\calb$, we will need 
%to verify a snake equation involving a diagonal composite of two sectors, and for that, we need 
the following variant of the sector interchange isomorphism~\cite[\interchangeiso]{BDH(1*1)}.  

To simplify the maneuvers involved in this interchange isomorphism, here and for the remainder of the paper, we use a model for the vertical composition of sectors that fuses sectors along one-quarter of their boundary:
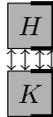
\[\tikzmath[scale=\displscale]{ 
\pgftransformxshift{15}\draw[<->] (20.1-.7,2.8) -- (20.1-.7,-2.8); 
\pgftransformxshift{15}\draw[<->] (23.2-.7,2.8) -- (23.2-.7,-2.8); 
\pgftransformxshift{510} 
\fill[spacecolor] (0,3) -- (0,15) -- (12,15) -- (12,3);
\draw (6,3) -- (0,3) -- (0,15) -- (6,15);
\draw[ultra thick](6,15) -- (12,15) -- (12,3) -- (6,3); 
\fill[spacecolor] (0,-15) -- (0,-3) -- (12,-3) -- (12,-15);
\draw (6,-15) -- (0,-15) -- (0,-3) -- (6,-3);
\draw[ultra thick](6,-3) -- (12,-3) -- (12,-15) -- (6,-15); 
\draw (6,9)node {$H$} (6,-9)node {$K$}; 
\draw[<->] (12.2-.7,2.6) -- (12.2-.7,-2.6); 
\draw[<->] (9.1-.7,2.6) -- (9.1-.7,-2.6);}
\]
This is by contrast with the model we used previously, in \cite{BDH(1*1)}, which involved fusing along half of the boundary of each sector.  The equivalence between these two fusions is discussed in Appendix~\ref{app-fusion}.

Let $\cala$, $\calb$, $\calc$ be conformal nets, let ${}_\cala D_\calb$, ${}_\calb E_\calc$, ${}_\calb F_\calc$, ${}_\cala G_\calc$ be defects, let
$H$ be an $F$-$E$-sector, and let $K$ be a $D\!\circledast_\calb \!E$\,-\,$G$\,-sector.
We are interested in two ways of evaluating the diagram
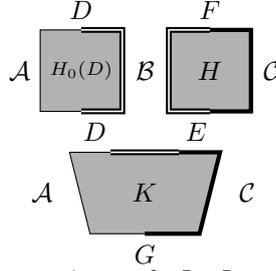
\begin{equation}\label{eq: skew glueing of sectors}
\tikzmath[scale=1.7]{\node[inner sep=5] (a) at (0,0) {$\cala$};\node[inner sep=5] (b) at (2,0) {$\calb$};\node[inner sep=5] (c) at (4,0) {$\calc$};
\draw[->] (a) .. controls (1,-1) and (3,-1) .. node[below]{$\scriptstyle G$} (c);
\draw[->] (a) -- node[below]{$\scriptstyle D$} (b);\draw[->] (b) -- node[below]
{$\scriptstyle E$} (c);
\draw[->] (b) to[out=55, in=125]node[above]{$\scriptstyle F$} (c);
\node at (2.95,.3) {$\Downarrow$};\node at (3.15,.3) {$H$};
\node at (1.95,-.4) {$\Downarrow$};\node at (2.15,-.4) {$K$};
}%tikzmath 
\end{equation}
i.e., of fusing the three sectors
\(
\def\h{10}
\tikzmath[scale=.09]{
\useasboundingbox (-5.5,-23) rectangle (35,17);
\fill[spacecolor](0,0) rectangle(12,12);\draw (6,0) -- (0,0) -- (0,12) -- (6,12);\draw[thick, double](6,12) -- (12,12) -- (12,0) -- (6,0); 
\draw (-3,6) node {$\cala$}(6,15) node {$D$}(8,-3) node {$D$}(6,6)node {$\scriptstyle H_0(D)$};
\pgftransformxshift{\h}\draw(15,6) node {$\calb$};\pgftransformxshift{\h}
\pgftransformxshift{515}
\fill[spacecolor](0,0) rectangle(12,12);\draw[thick, double] (6,0) -- (0,0) -- (0,12) -- (6,12);\draw[ultra thick](6,12) -- (12,12) -- (12,0) -- (6,0); 
\draw (15,6) node {$\calc$}(6,15) node {$F$}(6,6)node {$H$} (4,-3) node {$E$};
\pgftransformyshift{-513}
\pgftransformxshift{-100}   
\fill[spacecolor](-8,0) -- (-11,12) -- (11,12) -- (8,0) -- cycle;\draw (0,0) -- (-8,0) -- (-11,12) -- (-5,12);\draw[thick, double] (-5,12) -- (5,12);\draw[ultra thick] (5,12) -- (11,12) -- (8,0) -- (0,0); 
\draw (-3,6) (15,6) node {$\calc$}(0,-3) node {$G$}(0,6)node {$K$};
\draw (-15,6) node {$\cala$}(6,-3);
}\,\,.
\)
\\
Let us name and orient the relevant intervals $I_1$, $I_2$, $\ldots$, $I_{10}$ as indicated here:
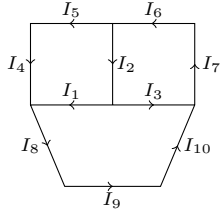
\[
\tikzmath[scale=.09]{\draw (-12,0) -- (12,0) (0,0) -- (0,12) -- (12,12) -- (12,0) -- (7,-12) -- (-7,-12) -- (-12,0) -- (-12,12) -- (0,12);
\draw [->] (-6.3,0) -- (-6.31,0);
\draw [->] (6,0) -- (6.01,0);
\draw [->] (0,6.01) -- (0,6);
\draw [->] (-12,6.01) -- (-12,6);
\draw [->] (12,6) -- (12,6.01);
\draw [->] (0,-12) -- (0.01,-12);
\draw [->] (-6.3,12) -- (-6.31,12);
\draw [->] (5.71,12) -- (5.7,12);
\draw [->] (-9.505,-6) -- (-9.5,-6.015);
\draw [->] (9.5,-6.015) -- (9.505,-6);
\draw (-6,2) node {$\scriptstyle I_1$}(2,6) node {$\scriptstyle I_2$}(6,2) node {$\scriptstyle I_3$}(-14,6) node {$\scriptstyle I_4$}
(-6,14) node {$\scriptstyle I_5$}(6,14) node {$\scriptstyle I_6$}(14.5,6) node {$\scriptstyle I_7$}(-12,-6) node {$\scriptstyle I_8$}(0,-14) node {$\scriptstyle I_9$}(12.5,-6) node {$\scriptstyle I_{10}$};}\,\,.
\]
All of them are copies of the standard interval $[0,1]$.
Let also $S_l:=\bar I_1\cup \bar I_2\cup I_5\cup I_4$, $S_r:= I_2\cup I_3\cup I_7\cup I_6$,
$S_b:= I_8\cup I_9\cup I_{10}\cup \bar I_3\cup I_1$,
$S_{lr}:= \bar I_1\cup I_3\cup I_7\cup I_6\cup I_5\cup I_4$,
$S_{lb}:=I_8\cup I_9\cup I_{10}\cup \bar I_3\cup \bar I_2\cup I_5\cup I_4$, and
$S_{lrb}:=I_8\cup I_9\cup I_{10}\cup I_7\cup I_6\cup I_5\cup I_4$,
where we have used bars to indicate reverse orientation.
\begin{lemma}\label{lem: HKK = HKH}
There is a non-canonical unitary isomorphism
\begin{equation}\label{eq: HKK = HKH}
\big(H_0(S_l,D)\boxtimes_{\calb(I_1)} K\big)\underset{\calb(I_2)\vee E(I_3)}\boxtimes H
\,\,\cong\,\,
\big(H_0(S_l,D)\boxtimes_{\calb(I_2)} H\big)\underset{D(I_1)\vee E(\bar I_3)}\boxtimes K\,,
\end{equation}
equivariant with respect to $\cala(I_4)$, $D(I_5)$, $E(I_6)$, $\calc(I_7)$, $\cala(I_8)$, $G(I_9)$, and $\calc(I_{10})$.
\end{lemma}

\begin{proof}
For fixed $\cala$, $\calb$, $\calc$, $D$, $E$, $G$, $K$,
the desired isomorphism \eqref{eq: HKK = HKH} can be thought of as a natural transformation 
\begin{equation}\label{eq: Emod 2arr DGmod}
E\big( I_2\cup I_3
\big)\text{-modules}\,\,\,\tworarrow\,\,
D\big(I_4\cup I_5\big)\otimes_{\alg}
G\big(I_8\cup I_9\cup I_{10}\big)
\text{ -modules}
\end{equation}
between functors of the variable $H$.
The fact that \eqref{eq: HKK = HKH} commutes with the action of $F(I_6\cup I_7)$ is then encoded in the naturality of \eqref{eq: Emod 2arr DGmod}.

Since $H_0(E)$ is a faithful $E(I_2 \cup I_3)$-module, 
it is enough, by~\cite[\lemNTbetweenmodulecategories]{BDH(1*1)}, to construct the isomorphism \eqref{eq: HKK = HKH} for $H=H_0(E)$ and check that it commutes with the action of $F(I_6\cup I_7)$.
Pick involutions $\varphi\in\Diff_-(S_l)$, $\psi\in \Diff_-(S_r)$, $\chi\in \Diff_-(S_{lr})$ such that
\[
\varphi(I_4\cup I_5\cup I_2) = I_1,\quad
\psi(I_6\cup I_7) = I_2\cup I_3,\quad \chi(I_4\cup I_5\cup I_6\cup I_7)= I_1\cup I_3,
\]
and corresponding (non-canonical) unitaries $u:H_0(D)\xrightarrow{\scriptscriptstyle \cong} L^2(D(I_1))$, $v:H_0(E)\xrightarrow{\scriptscriptstyle \cong} L^2(E(I_2\cup I_3))$, and $w:H_0(D\circledast_\calb E)\xrightarrow{\scriptscriptstyle \cong} L^2(D(I_1)\vee E(I_3))$, as in~\cite[\lemnoncanonvacuumdefect]{BDH(1*1)}.
Let also
\[
\begin{split}
\alpha\,&:=  \varphi|_{I_4\cup I_5\cup I_2}\cup\, \mathrm{Id}_{I_8\cup I_9 \cup I_{10}\cup I_3}:S_{lb}\to S_{b},\\
\beta\,&:=  \psi|_{I_6\cup I_7}\cup\, \mathrm{Id}_{I_5\cup I_4\cup I_8\cup I_9 \cup I_{10}}:S_{lrb}\to S_{lb},\\
\gamma\,&:=\chi|_{I_4\cup I_5\cup I_6\cup I_7}\cup\, \mathrm{Id}_{I_8\cup I_9 \cup I_{10}}: S_{lrb}\to S_{b}.\\
\end{split}
\]
We may assume that $\varphi$, $\psi$, and $\chi$ are chosen so that $\alpha\circ \beta=\gamma$.
The isomorphism \eqref{eq: HKK = HKH} for $H=H_0(E)$ can then be written explicitly:
\[
\begin{split}
&\hspace{.1cm}\big(H_0(D) \boxtimes_{D(I_1)} K\big)\boxtimes_{E(I_2\cup I_3)} H_0(E)\xrightarrow{u\otimes 1}
\big(L^2(D(I_1)) \boxtimes_{D(I_1)} K\big)\boxtimes_{E(I_2\cup I_3)} H_0(E)\\
&\hspace{1.3cm}\xrightarrow{\scriptscriptstyle \cong}\,\,
\alpha^*K\boxtimes_{E(I_2\cup I_3)} H_0(E)
\,\,\xrightarrow{1\otimes v}\,\,
\alpha^*K\boxtimes_{E(I_2\cup I_3)} L^2(E(I_2\cup I_3))\\
&\hspace{1.7cm}\,\xrightarrow{\scriptscriptstyle \cong}\,
\beta^*\alpha^* K
=\gamma^*K\,\xrightarrow{\scriptscriptstyle \cong}\,
L^2\big(D(I_1)\vee E(\bar I_3)\big) \boxtimes_{D(I_1)\vee E(\bar I_3)} K\\
&\xrightarrow{w^{-1}\otimes 1}\,H_0(D\circledast_\calb E) \boxtimes_{D(I_1)\vee E(\bar I_3)} K\,\xrightarrow{\Omega\otimes 1}\,
\big(H_0(D)\boxtimes_{\calb(I_2)} H_0(E)\big)\boxtimes_{D(I_1)\vee E(\bar I_3)}\, K,
\end{split}
\]
where $\Omega$ denotes the ``$1 \boxtimes 1$-isomorphism'' constructed in~\cite[\thminterchange]{BDH(1*1)}.
\end{proof}

Generalizing \eqref{eq: skew glueing of sectors}, we now consider this situation:
\begin{equation}\label{eq: past diag for zigzag check}
\tikzmath[scale=1.4]{\node[inner sep=5] (a) at (0,0) {$\cala$};\node[inner sep=5] (b) at (2,0) {$\calb$};\node[inner sep=5] (c) at (4,0) {$\calc$};\node[inner sep=5] (d) at (6,0) {$\cald$};
\draw[->] (a) .. controls (1,-1) and (3,-1) .. node[below]{$\scriptstyle Q$} (c);
\draw[->] (b) .. controls (3,1) and (5,1) .. node[above]{$\scriptstyle P$} (d);
\draw[->] (a) -- node[above]{$\scriptstyle D$} (b);\draw[->] (b) -- node[above]{$\scriptstyle E$} (c);\draw[->] (c) -- node[above]{$\scriptstyle F$} (d);
\node at (1.95,-.4) {$\Downarrow$};\node at (2.15,-.4) {$K$};
\node at (3.95,.4) {$\Downarrow$};\node at (4.15,.4) {$H$};
}%tikzmath 
\end{equation}
which corresponds (using Appendix~\ref{app-fusion}) to the the following configuration of sectors:
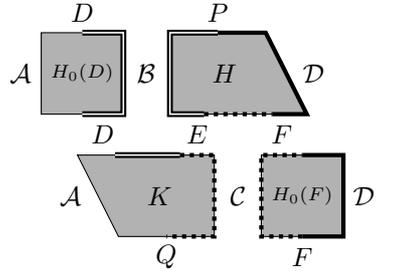
\[
\def\h{10}
\tikzmath[scale=.09]{
\fill[spacecolor](0,0) rectangle(12,12);\draw (6,0) -- (0,0) -- (0,12) -- (6,12);\draw[thick, double](6,12) -- (12,12) -- (12,0) -- (6,0); 
\draw (-3,6) node {$\cala$}(6,15) node {$D$}(9,-3) node {$D$}(6,6)node {$\scriptstyle H_0(D)$};
\pgftransformxshift{\h}\draw(15,6) node {$\calb$};\pgftransformxshift{\h}
\pgftransformxshift{515}
\fill[spacecolor](0,0) -- (0,12) -- (14,12) -- (20,0) -- cycle;\draw[thick, double] (5,0) -- (0,0) -- (0,12) -- (7,12);\draw[ultra thick](7,12) -- (14,12) -- (20,0) -- (15,0); 
\draw[ultra thick, dotted] (5,0) -- (15,0);\draw(5,0) -- (15,0);
\draw (21,6) node {$\cald$}(7,15) node {$P$}(8,6)node {$H$} (4,-3) node {$E$};
\pgftransformyshift{-513}
\pgftransformxshift{-100}   
\fill[spacecolor](-4,0) -- (-10,12) -- (10,12) -- (10,0) -- cycle;\draw (3,0) -- (-4,0) -- (-10,12) -- (-4.5,12);\draw[thick, double] (-4.5,12) -- (5,12);
\draw[ultra thick, dotted] (5,12) -- (10,12) -- (10,0) -- (3,0); 
\draw (5,12) -- (10,12) -- (10,0) -- (3,0); 
\draw (3,-3) node {$Q$}(2,6)node {$K$};
\draw (-11,6) node {$\cala$}(6,-3);
\draw(13.5,6) node {$\calc$};
\pgftransformxshift{480}
\fill[spacecolor](0,0) rectangle(12,12);\draw[ultra thick, dotted] (6,0) -- (0,0) -- (0,12) -- (6,12);\draw (6,0) -- (0,0) -- (0,12) -- (6,12);\draw[ultra thick](6,12) -- (12,12) -- (12,0) -- (6,0); 
\draw (15,6) node {$\cald$}(3,15) node {$F$}(6,-3) node {$F$}(6,6)node {$\scriptstyle H_0(F)$};
}\,\,.
\]
We name the relevant intervals $I_1$, $I_2$, $\ldots$, $I_{13}$:
\begin{equation}\label{eq: I_1 . . . . . . . I_13}
\tikzmath[scale=.09]{\draw (-12,0) -- (24,0) (0,0) -- (0,12) -- (18,12) -- (24,0) -- (24,-12) -- (12,-12) -- (12,0)(12,-12) -- (-6,-12) -- (-12,0) -- (-12,12) -- (0,12);
\draw (-5.7,2) node {$\scriptstyle I_1$}
(2,6) node {$\scriptstyle I_2$}
(6,2) node {$\scriptstyle I_3$}
(14.3,-6) node {$\scriptstyle I_4$}
(18,2) node {$\scriptstyle I_5$}
(-14,6) node {$\scriptstyle I_6$}
(-5.7,14) node {$\scriptstyle I_7$}
(10,14) node {$\scriptstyle I_8$}
(24,6) node {$\scriptstyle I_9$}
(-12,-6) node {$\scriptstyle I_{10}$}
(3,-14) node {$\scriptstyle I_{11}$}
(18,-14) node {$\scriptstyle I_{12}$}
(27,-6) node {$\scriptstyle I_{13}$};
\draw [->] (-6.3,0) -- (-6.31,0);
\draw [->] (17.8,0) -- (17.79,0);
\draw [->] (18.3,-12) -- (18.31,-12);
\draw [->] (6,0) -- (6.01,0);
\draw [->] (12,-6) -- (12,-6.01);
\draw [->] (24,-5.51) -- (24,-5.5);
\draw [->] (0,6.01) -- (0,6);
\draw [->] (-12,6.01) -- (-12,6);
\draw [->] (3.3,-12) -- (3.31,-12);
\draw [->] (-6.3,12) -- (-6.31,12);
\draw [->] (8.01,12) -- (8,12);
\draw [->] (-9.005,-6) -- (-9,-6.01);
\draw [->] (21.005,6) -- (21,6.01);
}\,\,
\end{equation}
Once again, all these intervals are copies of the standard interval $[0,1]$.
\begin{lemma}\label{lem: zig-zag for sectors}
Let  $\cala$, $\calb$, $\calc$, $\ldots$, $K$ be as in \eqref{eq: past diag for zigzag check}.
Then there is a non-canonical unitary isomorphism
\begin{equation}\label{eq: HKK = HKH -- now bigger}
\begin{split}
&H_0(D)\underset{D(I_1)\vee \calb(I_2)}\boxtimes\Big(K\,\boxtimes_{E(I_3)} H\Big)\underset{\calc(I_4)\vee F(I_5)}\boxtimes H_0(F)
\\&\qquad\qquad\cong\quad
\big(H_0(D)\,\boxtimes_{\calb(I_2)} H\big)\underset{D(I_1)\vee E(\bar I_3)\vee F(I_5)}\boxtimes \big(K\,\boxtimes_{\calc(I_4)} H_0(F)\big)
\end{split}
\end{equation}
that is equivariant with respect to the actions of the algebras $\cala(I_6)$, $D(I_7)$, $P(I_8)$, $\cald(I_9)$, $\cala(I_{10})$, $Q(I_{11})$, $F(I_{12})$, and $\cald(I_{13})$.
\end{lemma}

\begin{proof}
Fix $\cala$, $\calb$, $\calc$, $\cald$, $D$, $E$, $F$.
We shall construct a natural transformation
\[
\begin{split}
%\begin{matrix}
(E\circledast_\calc F)\big(\,
\tikzmath[scale=\textscale]{\draw (6,0) -- (0,0) -- (0,12);\draw[ultra thick]  (12,0) -- (6,0); \draw[->] (4.5,0) -- (4.51,0);}\,\big)\text{-modules}%\\
\,\,\,\times&\,\,\, (D\circledast_\calb E)\big(\,
\tikzmath[scale=\textscale]{\draw(6,12) -- (0,12);\draw[ultra thick] (6,12) -- (12,12) -- (12,0); \draw[thick, ->] (7.01,12) -- (7,12);}\,\big)\text{-modules}%\end{matrix}
\\\,\,\,\tworarrow\,\,&
D\big(I_7\cup I_6\big)\otimes_{\alg}
F\big(I_{12}\cup I_{13}\big)
\text{ -modules},
\end{split}
\]
where 
$(E\circledast_\calc F)\big(\,
\tikzmath[scale=\textscale]{\draw (6,0) -- (0,0) -- (0,12);\draw[ultra thick]  (12,0) -- (6,0); \draw[->] (4.5,0) -- (4.51,0);}\,\big) = E(I_2\cup I_3)\vee F(\bar I_5)$
and 
$(D\circledast_\calb E)\big(\,
\tikzmath[scale=\textscale]{\draw(6,12) -- (0,12);\draw[ultra thick] (6,12) -- (12,12) -- (12,0); \draw[thick, ->] (7.01,12) -- (7,12);}\,\big) = D( I_1)\vee E(\bar I_3\cup \bar I_4)$, as in~\cite[Def.\,1.43]{BDH(1*1)}.
The isomorphism \eqref{eq: HKK = HKH -- now bigger} is the value of that natural transformation on the object $(H,K)$.

By~\cite[\lemNTbetweenmodulecategories]{BDH(1*1)}, it is enough to construct the above natural transformation for the pair $(H=H_0(E\circledast_\calc F), K=H_0(D\circledast_\calb E))$. In that case, it is given by
\begin{gather*}
H_0(D)\underset{D(I_1)\vee \calb(I_2)}\boxtimes\Big(H_0(D\circledast_\calb E)\,\underset{E(I_3)}\boxtimes H_0(E\circledast_\calc F)\Big)\underset{\calc(I_4)\vee F(I_5)}\boxtimes H_0(F)\\
\xrightarrow{\scriptstyle 1\otimes 1\otimes \Omega\otimes 1}\\
H_0(D)\underset{D(I_1)\vee \calb(I_2)}\boxtimes\Big(H_0(D\circledast_\calb E)\,\underset{E(I_3)}\boxtimes H_0(E)\,\underset\calc\boxtimes\, H_0(F)\Big)\underset{\calc(I_4)\vee F(I_5)}\boxtimes H_0(F)\\
\cong \bigg(H_0(D)\underset{D(I_1)\vee \calb(I_2)}\boxtimes\Big(H_0(D\circledast_\calb E)\,\underset{E(I_3)}\boxtimes H_0(E)\Big)\bigg)\underset\calc\boxtimes\, H_0(F)\underset{\calc(I_4)\vee F(I_5)}\boxtimes H_0(F)\phantom{\cong}\\
\xrightarrow{\text{Lemma \ref{lem: HKK = HKH} }\otimes 1\otimes 1}\\
\bigg(\Big(H_0(D)\underset{\calb(I_2)}\boxtimes H_0(E)\Big)\underset{D(I_1)\vee E(\bar I_3)}\boxtimes H_0(D\circledast_\calb E) \bigg)\underset\calc\boxtimes\, H_0(F)\underset{\calc(I_4)\vee F(I_5)}\boxtimes H_0(F)\\
\cong \bigg(\Big(H_0(D)\underset{\calb(I_2)}\boxtimes H_0(E)\,\underset\calc\boxtimes\, H_0(F)\Big)\underset{D(I_1)\vee E(\bar I_3)}\boxtimes H_0(D\circledast_\calb E) \bigg)\underset{\calc(I_4)\vee F(I_5)}\boxtimes H_0(F)\phantom{\cong}\\
\xrightarrow{\scriptstyle 1\otimes  \Omega^{-1}\otimes 1\otimes 1}\\
\cong \bigg(\Big(H_0(D)\underset{\calb(I_2)}\boxtimes H_0(E\circledast_\calc F)\Big)\underset{D(I_1)\vee E(\bar I_3)}\boxtimes H_0(D\circledast_\calb E) \bigg)\underset{\calc(I_4)\vee F(I_5)}\boxtimes H_0(F)\phantom{\cong}\\
\xrightarrow{\text{Lemma \ref{lem: HKK = HKH}}}\\
\Big(H_0(D)\underset{\calb(I_2)}\boxtimes H_0(E\circledast_\calc F)\Big)
\underset{D(I_1)\vee E(\bar I_3)\vee F(I_5)}\boxtimes 
\Big(H_0(D\circledast_\calb E) \underset{\calc(I_4)}\boxtimes H_0(F)\Big).
\end{gather*}
\end{proof}

\subsection{Finite nets are dualizable} \label{sec:finitenetsdualizable}

We investigate the relationship of finiteness and dualizability for, in turn, sectors, defects, and nets.

\subsubsection*{Dualizability for sectors}

Recall that all defects are assumed to be semisimple.
\begin{proposition} \label{prop:sectoradjoint}
A sector ${}_D H_E$ has an adjoint (necessarily ambidextrous) if and only if it is finite.  In this case, the adjoint is canonically isomorphic to ${}_E \bar H_D$.
\end{proposition}

\begin{proof}
If the sector ${}_D H_E$ has an adjoint ${}_E K_D$, that adjoint sector provides the (ambidextrous) adjoint ${}_{E(S^1_\top)} K_{D(S^1_\bot)^\op}$ to the bimodule ${}_{D(S^1_\top)} H_{E(S^1_\bot)^\op}$, ensuring that $H$ is finite. 

Conversely, if $H$ is a dualizable $D(S^1_\top)$-$E(S^1_\bot)^\op$-bimodule then, by~\cite[Cor.\,6.12]{BDH(Dualizability+Index-of-subfactors)} and the fact that $D$ and $E$ are semisimple,
its dual is canonically isomorphic to $\bar H$, with the $E(S^1_\bot)^\op$-$D(S^1_\top)$-bimodule structure given by $a\bar\xi b=\overline{b^*\xi a^*}$. Identify the left action of $E(S^1_\bot)^\op$ with a left action of $E(S^1_\top)$, and the right action of $D(S^1_\top)$ with a left action of $D(S^1_\bot)$ via the isomorphisms $j_*:E(S^1_\bot)^\op\to E(S^1_\top)$ and $j_*:D(S^1_\top)^\op\to D(S^1_\bot)$; then extend these actions to the structure of an $E$-$D$-sector on $\bar H$ according to~\eqref{eq:  pencil star}.  The unit and counit bimodule intertwiners for the bimodule duality serve, in fact, as sector intertwiners, providing $_E\bar H_D$ with the structure of an adjoint sector to ${}_DH_E$.
\end{proof}
\nid By Remark~\ref{remark-ambi}, we have the following:
\begin{corollary} \label{cor:sectordualizable}
A sector is dualizable if and only if it is finite.
\end{corollary}

\subsubsection*{Dualizability for defects}

\def \Itopw {\tikz{\useasboundingbox (-.11,-.12) rectangle (.12,.19); \draw node {$I$} (-.08,.18)--(.12,.18);}}
\def \Itopb {\tikz{\useasboundingbox (-.11,-.12) rectangle (.12,.19); \draw[very thick] node {$I$} (-.08,.18)--(.12,.18);}}

Given a bicolored interval $I$, we define the following two bicolored manifolds $\Itopw$ and $\Itopb$.
The underlying manifold of $\Itopw$ and of $\Itopb$ are both given by $I_\circ\cup[0,1]\cup I_\bullet$,
and their bicolorings are
\[
\Itopw_\circ = [0,1],\qquad \Itopw_\bullet = I_\circ \cup I_\bullet,\qquad
\Itopb_\circ = I_\circ \cup I_\bullet,\qquad \Itopb_\bullet = [0,1].
\]
Here is an example illustrating the above concepts:
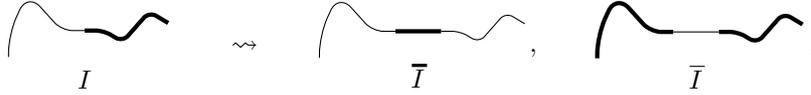
\[
\tikzmath[scale=\displscale]{\draw (-20,-7) [rounded corners=7pt]-- (-20,-1) -- (-15,10) [rounded corners=4pt]-- (-6,0) -- (0,0);
\draw[ultra thick] (0,0) [rounded corners=3pt]-- (5,0) -- (10,-3) -- (17,5) -- (22,2)(0,-13) node {$I$};}%tikzmath
\qquad\leadsto\qquad
\tikzmath[scale=\displscale]{\draw (-14,5) [rounded corners=7pt]-- (-14,11) -- (-9,22) [rounded corners=4pt]-- (0,12) -- (6,12)
(18,12) [rounded corners=3pt]-- (23,12) -- (28,9) -- (35,17) -- (40,14)(12,0) node {$\Itopb$};
\draw[ultra thick] (6,12) -- (18,12);}\,,%tikzmath
\qquad
\tikzmath[scale=\displscale]{\draw[ultra thick] (-14,5) [rounded corners=7pt]-- (-14,11) -- (-9,22) [rounded corners=4pt]-- (0,12) -- (6,12)
(18,12) [rounded corners=3pt]-- (23,12) -- (28,9) -- (35,17) -- (40,14)(12,0) node {$\Itopw$};
\draw (6,12) -- (18,12);}\,.%tikzmath
\end{equation*}
Let $D$ be an $\cala$-$\calb$-defect.
Definition \ref{def: D on all 1-manifolds} is made so as to provide an easy description of $D\circledast D^\dagger$ and $D^\dagger \circledast D$.
They are given by
\begin{equation}\label{eq: DDdag and DdagD}
\big(D\circledast_\calb D^\dagger\big)(I) = D(\Itopb)\qquad\text{and}\qquad
\big(D^\dagger \circledast_\cala D\big)(I) = D(\Itopw),
\end{equation}
essentially by definition.

\begin{proposition} \label{prop:defectadjoint}
Let $\cala$ and $\calb$ be finite conformal nets.  Every finite defect ${}_\cala D_\calb$ has ambidextrous adjoint ${}_\calb D^\dagger\!\!{}_\cala$, and the unit and counit sectors of both the left and right adjunctions are finite.
\end{proposition}

\begin{proof}
By Remark~\ref{remark-ambi}, it suffices to consider just one of the two adjunctions.

\def\Soneb{\tikz{\useasboundingbox (-.11,-.13) rectangle (.2,.18); \draw[very thick] (.07,0) node {$S^1$} (-.08,.17)--(.12,.17);}}
\def\SoneW{\tikz{\useasboundingbox (-.11,-.13) rectangle (.2,.18); \draw (.07,0) node {$S^1$} (-.11,-.21)--(.09,-.21);}}
Let $\Soneb$ be the bicolored manifold obtained by taking the standard circle $S^1$, cutting it open at $i\in S^1$, and then glueing in a copy of $[0,1]$.
The black part of $\Soneb$ is the interval $[0,1]$ that is added on the top, and all the rest is white.
Similarly, 
let $\SoneW$ be the bicolored manifold that is obtained by inserting a white interval at the location of $-i\in S^1$, and coloring all the rest black.
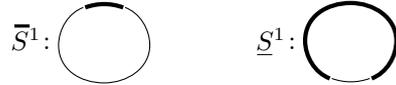
\[
\Soneb:\tikzmath[scale=\displscale]{
\draw (-5.5,10) to[out=200, in=90, looseness=.9] (-12,0) to[out=-90, in=180, looseness=1.05] (0,-10) to[out=0, in=-90] (12,0) to[out=90, in=-20, looseness=.9] (5.5,10);
\draw[ultra thick] (-5,10) to[out=20, in=160, looseness=.9] (5,10);
}\qquad\qquad %tikzmath
\SoneW:\tikzmath[scale=\displscale]{
\draw[ultra thick] (-5.5,-10) to[out=-200, in=-90, looseness=.9] (-12,0) to[out=90, in=180, looseness=1.05] (0,10) to[out=0, in=90] (12,0) to[out=-90, in=20, looseness=.9] (5.5,-10);
\draw (-5,-10) to[out=-20, in=-160, looseness=.9] (5,-10);
} %tikzmath
\]
By \eqref{eq: DDdag and DdagD}, a $D\circledast_\calb D^\dagger$ - $1_\cala$ -sector is the same thing as a $\{D(I)\}_{I\in\INT_{\tikz[scale=.65]{\useasboundingbox (-.11,-.13) rectangle (.22,.18); \draw[very thick] (.12,0.04) node[scale=1] {$\scriptscriptstyle S^1$} (-.08,.17)--(.12,.17);}}}
$-representation, where 
$\INT_{\tikz[scale=.65]{\useasboundingbox (-.11,-.13) rectangle (.22,.18); \draw[very thick] (.1,0.04) node[scale=1] {$\scriptstyle S^1$} (-.08,.22)--(.12,.22);}}
$
denotes the poset of subintervals $I\subset \Soneb$, $\partial I \cap \Soneb\!{}_\bullet=\emptyset$, that are allowed to contain $\Soneb\!{}_\bullet$ in their interior, but that are not allowed to contain $\Soneb\!{}_\circ$.
Pick a color preserving diffeomorphism $\varphi$ from $\Soneb$ to the standard bicolored circle.
By Proposition \ref{prop: tricolored interval acts on H_0(D)}, 
we can use $\varphi$ to induce the structure of a $\{D(I)\}_{I\in\INT_{\tikz[scale=.65]{\useasboundingbox (-.11,-.13) rectangle (.22,.18); \draw[very thick] (.12,0.04) node[scale=1] {$\scriptscriptstyle S^1$} (-.08,.17)--(.12,.17);}}}
$-representation on $H_0(D)$.
That is the counit sector $r$ of our adjunction.
Similarly, restricting $H_0(D)$ along a color preserving diffeomorphism from $\SoneW$ to the standard bicolored circle provides a
$1_\calb$\,-\,$D^\dagger\circledast_\cala D$\,-sector $s$,
which is the unit of our adjunction.
The sectors $r$ and $s$ are finite by the finiteness of any defect vacuum sector with respect to these boundary decompositions~\cite[\lemDtildeHBfinite]{BDH(1*1)}.\footnote{This vacuum sector finiteness result~\cite[\lemDtildeHBfinite]{BDH(1*1)} was stated for irreducible defects, but it also holds for semisimple defects and in fact for arbitrary defects, using the direct integral decomposition~\cite[Lem.\,1.32]{BDH(1*1)}.}

We now have to show that $r$ and $s$ satisfy the duality equations
\[
\left(\tikzmath[scale=.9]{
\useasboundingbox (-.3,-1) rectangle (6.3,1);
\node (a) at (0,0) {$\cala$};\node (b) at (2,0) {$\calb$};\node (c) at (4,0) {$\cala$};\node (d) at (6,0) {$\calb$};
\draw[->] (a) -- node[above, yshift=-.1]{$\scriptstyle D$} (b);\draw[->] (a) to[out=-45, in=-135]node[below]{$\scriptstyle 1_\cala$} (c);\draw[->] (c) -- node[above, yshift=-.1]{$\scriptstyle D$} (d);
\draw[->] (b) to[out=45, in=135]node[above]{$\scriptstyle 1_\calb$} (d);\draw[->] (b) -- node[above, yshift=-.1, xshift=.3] {$\scriptstyle D^\dagger$} (c);
\node at (2,-.52) {$\Downarrow$};\node at (2.3,-.52) {$r$};\node at (4,.52) {$\Downarrow$};\node at (4.3,.52) {$s$};
}%tikzmath
\right)
\tikzmath{\useasboundingbox (-.6,-1.2) rectangle (.6,1.2);\node{$\cong$};}
\left(
\tikzmath[scale=.9]{ \useasboundingbox (-.3,-1) rectangle (4.1,1);
\node (a') at (0,0) {$\cala$};\node (c') at (3.8,0) {$\calb$};
\draw[->] (a') .. controls (1,.7) and (2.8,.7) .. node[above]{$\scriptstyle D$} (c');
\draw[->] (a') .. controls (1,-.7) and (2.8,-.7) .. node[below]{$\scriptstyle D$} (c');
\node at (1.9,0) {$\Downarrow$};\node at (2.4,0) {$\scriptstyle 1_{D}$};}%tikzmath
\right)
\]
and
\[
\;\left(\tikzmath[scale=.9]{
\useasboundingbox (-.3,-1) rectangle (6.3,1);
\node (a) at (0,0) {$\calb$};\node (b) at (2,0) {$\cala$};\node (c) at (4,0) {$\calb$};\node (d) at (6,0) {$\cala$};
\draw[->] (a) -- node[above, yshift=-.1, xshift=.3]{$\scriptstyle D^\dagger$} (b);\draw[->] (a) to[out=45, in=135]node[above]{$\scriptstyle 1_\calb$} (c);\draw[->] (c) -- node[above, yshift=-.1, xshift=.3]{$\scriptstyle D^\dagger$} (d);
\draw[->] (b) to[out=-45, in=-135]node[below]{$\scriptstyle 1_\cala$} (d);\draw[->] (b) -- node[above, yshift=-.1] {$\scriptstyle D$} (c);
\node at (2,.52) {$\Downarrow$};\node at (2.3,.52) {$s$};\node at (4,-.52) {$\Downarrow$};\node at (4.3,-.52) {$r$};
}%tikzmath
\right)
\tikzmath{\useasboundingbox (-.6,-1.2) rectangle (.6,1.2);\node{$\cong$};}
\left(
\tikzmath[scale=.9]{ \useasboundingbox (-.3,-1) rectangle (4.1,1);
\node (a') at (0,0) {$\calb$};\node (c') at (3.8,0) {$\cala$};
\draw[->] (a') .. controls (1,.7) and (2.8,.7) .. node[above]{$\scriptstyle D^\dagger$} (c');
\draw[->] (a') .. controls (1,-.7) and (2.8,-.7) .. node[below]{$\scriptstyle D^\dagger$} (c');
\node at (1.9,0) {$\Downarrow$};\node at (2.4,0) {$\scriptstyle 1_{D^\dagger}$};}%tikzmath
\right).
\]
We only check the first equation, the second one being completely analogous. 
Let $I_1, \ldots, I_{13}$ be as in \eqref{eq: I_1 . . . . . . . I_13}.
By Lemma \ref{lem: zig-zag for sectors} and Appendix~\ref{app-fusion},
the left hand side
$\tikzmath[scale=.5]{
\useasboundingbox (-.3,-1) rectangle (6.3,1);
\node (a) at (0,0) {$\scriptstyle \cala$};\node (b) at (2,0) {$\scriptstyle \calb$};\node (c) at (4,0) {$\scriptstyle \cala$};\node (d) at (6,0) {$\scriptstyle \calb$};
\draw[->] (a) -- (b);\draw[->] (a) to[out=-40, in=-140] (c);\draw[->] (c) --  (d);
\draw[->] (b) to[out=40, in=140] (d);\draw[->] (b) --  (c);
\node at (2,-.55) {$\scriptstyle \Downarrow$};\node at (4,.55) {$\scriptstyle \Downarrow$};
}%tikzmath
$
is isomorphic to
\begin{equation}\label{eq: HrsH}
H_0(D)\underset{D(I_1\cup I_2)}\boxtimes\Big(r\underset{D^\dagger(I_3)}\boxtimes s\Big)\underset{D(I_4\cup I_5)}\boxtimes H_0(D).
\end{equation}
Because the fusion of two vacuum sectors for a defect is again a vacuum sector for that same defect~\cite[\lemvacvacvacdefects]{BDH(1*1)},
the middle term $\,r\boxtimes_{D^\dagger(I_3)} s\,$ is
the vacuum sector of $D$ associated to $I_1 \cup I_{10} \cup I_{11} \cup \bar I_4 \cup \bar I_5 \cup I_9 \cup I_8 \cup I_2$.
By two more applications of that same lemma, we identify \eqref{eq: HrsH} with the identity sector on $D$.
\end{proof}

\nid Recall that all conformal nets and defects are assumed to be semisimple.  Combining the above proposition with Corollary~\ref{cor:sectordualizable}, we have the following.

\begin{corollary} \label{cor:defectdualizable}
Every finite defect between finite conformal nets is fully dualizable.
\end{corollary}

\subsubsection*{Dualizability for conformal nets}

In~\cite{BDH(3-category)}, we constructed a $3$-category whose objects are finite conformal nets, whose morphisms are defects, whose $2$-morphisms are sectors, and whose $3$-morphisms are intertwiners.\footnote{Insisting that the conformal nets be finite allowed us to prove that the composition of two defects is again a defect; we do not know if the composition of defects between arbitrary conformal nets is a defect, in particular whether the composite satisfies the vacuum sector axiom~\cite[Def.\,1.7, axiom (iv)]{BDH(1*1)}.}  If $\cala$ and $\calb$ are conformal nets that are not-necessarily finite, then, even though we do not know that they live in a $3$-category, we can still make sense of $\cala$ and $\calb$ being dual: specifically, $\calb$ is the right dual of $\cala$ if there exist unit and counit defects 
${}_{\cala \otimes \calb}\raisebox{.5ex}{$r$}_{\underline\IC}$ and 
${}_{\underline\IC}\,\raisebox{.5ex}{$s$}_{\calb \otimes \cala}$
such that
$(1_\cala \otimes s) \circledast_{\cala\otimes\calb \otimes\cala} (r \otimes 1_\cala)$
and
$(s \otimes 1_{\calb}) \circledast_{\calb\otimes\cala \otimes\calb} ( 1_{\calb} \otimes r)$
are defects, and are equivalent (in the $2$-category of $\cala$-$\cala$-defects or $\calb$-$\calb$-defects)
to the identity defects on $\cala$ and $\calb$, respectively.

\begin{theorem}\label{thm:finitenetdualizable}
An arbitrary conformal net $\cala$ has ambidextrous dual $\cala^\op$.  If $\cala$ is finite, then the unit and counit defects of both the left and right dualities are themselves finite.
\end{theorem}
\begin{proof}
By Remark~\ref{remark-ambi}, it is enough to discuss just one of the two dualities.
We show that $\cala\dashv\cala^\op$.

Given a bicolored interval, let $I_{\circ\!\bullet}$ stands for $I_\circ\cap\, I_\bullet$.
It consists of one point if $I$ is genuinely bicolored, and it is empty otherwise. 
The counit defect ${}_{\cala \otimes \cala^\op}\raisebox{.5ex}{$r$}_{\underline\IC}$ and the unit defect 
${}_{\underline\IC}\,\raisebox{.5ex}{$s$}_{\cala^\op \otimes \cala}$
are defined by
\begin{equation}\label{eq: def of r and s defects}
r:I\mapsto \cala\big(I_\circ\cup_{I_{\circ\!\bullet}} \bar I_\circ\big)\qquad\text{and}\qquad
s:I\mapsto \cala\big(\bar I_\bullet\cup_{I_{\circ\!\bullet}} I_\bullet\big),
\end{equation}
where the bar stands for orientation reversal.
In pictures, this is:
\[
r \Big( \tikzmath[scale=.4] {\useasboundingbox (-2.2,-1.1) rectangle (2.2,1.1);\draw (-2,-.7) [rounded corners=7pt]-- (-2,-.1) -- 
(-1.5,1) [rounded corners=4pt]-- (-.6,0) -- (0,0); \draw[ultra thick] (0,0) [rounded corners=3pt]-- (.5,0) -- (1,-.3) -- (1.7,.5) -- (2.2,.2);\draw [->] (-.31,.002) -- (-.3,0);}\Big) %tikzmath
:=
\cala\Big(
\tikzmath[scale=.4] {\useasboundingbox (-2.3,-1.1) rectangle (.2,1.1);
\draw (-2,-.7) [rounded corners=7pt]-- (-2,-.1) -- (-1.5,1) [rounded corners=4pt]-- (-.6,0)[rounded corners=1] -- (0,0) -- (0,.1);
\draw (-2.15,-.55) [rounded corners=7pt]-- (-2.17,.1) -- (-1.45,1.19) [rounded corners=4pt]-- (-.6,.2)[rounded corners=1] -- (0,.2) -- (0,.1);
\draw [->] (-.31,.202) -- (-.3,0.2);}
\Big)
\qquad\text{and}\qquad
s \Big( \tikzmath[scale=.4] {\useasboundingbox (-2.2,-1.1) rectangle (2.2,1.1);\draw (-2,-.7) [rounded corners=7pt]-- (-2,-.1) -- 
(-1.5,1) [rounded corners=4pt]-- (-.6,0) -- (0,0); \draw[ultra thick] (0,0) [rounded corners=3pt]-- (.5,0) -- (1,-.3) -- (1.7,.5) -- (2.2,.2);\draw [->] (-.31,.002) -- (-.3,0);}\Big) %tikzmath
:=
\cala\Big(
\tikzmath[scale=.4] {\useasboundingbox (-.2,-1.1) rectangle (2.3,1.1);
\draw[rounded corners=1]  (0,.1) -- (0,0) [rounded corners=3pt]-- (.5,0) -- (1,-.3) -- (1.7,.5) -- (2.2,.2);
\draw[rounded corners=1] (0,.1) -- (0,0.2) [rounded corners=3pt]-- (.55,0.2) -- (.95,-.1) -- (1.65,.7) -- (2.2,.4);\draw [->] (.5,-.04) -- (.51,-.0425);}
\Big)
\]
We now verify the two duality equations for $r$ and $s$.
We need to show that the fusions 
$(1_\cala \otimes s) \circledast_{\cala\otimes\cala^\op \otimes\cala} (r \otimes 1_\cala)$
and
$(s \otimes 1_{\cala^\op}) \circledast_{\cala^\op\otimes\cala \otimes\cala^\op} ( 1_{\cala^\op} \otimes r)$ are indeed defects, and
are equivalent to identity defects on $\cala$ and $\cala^\op$, respectively.
Let $I$ be  genuinely bicolored interval.
By~\cite[Lem.\,1.12]{BDH(modularity)}, the definition of the above fusions reduces to
\[
\begin{split}
\big((1 \otimes s) \circledast (r \otimes 1)\big)&(I)= \cala\big(I_\circ \cup_{\{0\}} [0,1] \cup_{\{1\}} \overline{[0,1]} \cup_{\{0\}} [0,1] \cup_{\{1\}} I_\bullet\big)\\
\big((s \otimes 1) \circledast ( 1 \otimes r)\big)&(I)= \cala\big(\bar I_\circ \cup_{\{0\}} \overline{[0,1]} \cup_{\{1\}} [0,1] \cup_{\{0\}} \overline{[0,1]} \cup_{\{1\}} \bar I_\bullet\big),
\end{split}
\]
or perhaps more clearly
\[
\tikzmath[scale=.4] {\useasboundingbox (-2.2,-1.1) rectangle (2.2,1.1);\draw (-2,-.7) [rounded corners=7pt]-- (-2,-.1) -- 
(-1.5,1) [rounded corners=4pt]-- (-.6,0) -- (0,0); \draw[ultra thick] (0,0) [rounded corners=3pt]-- (.5,0) -- (1,-.3) -- (1.7,.5) -- (2.2,.2);\draw [->] (-.31,.002) -- (-.3,0);}
\,\,\,\mapsto \,
\cala\Big(\tikzmath[scale=.4] {\useasboundingbox (-2.2,-.9) rectangle (3.2,1.3);
\draw (-2,-.4) [rounded corners=7pt]-- (-2,.2) -- (-1.5,1.3) [rounded corners=4pt]-- (-.6,.3) [rounded corners=1]-- (1,.3) -- (1,.1) -- (0,.1) -- (0,-.1) [rounded corners=3pt]-- (1.5,-.1) -- (2,-.4) -- (2.7,.4) -- (3.2,.1);
\draw[->] (.5,.3) -- (.51,.3);}\Big)
\quad\,\,\,\text{and}\quad\,\,\,
\tikzmath[scale=.4] {\useasboundingbox (-2.2,-1.1) rectangle (2.2,1.1);\draw (-2,-.7) [rounded corners=7pt]-- (-2,-.1) -- 
(-1.5,1) [rounded corners=4pt]-- (-.6,0) -- (0,0); \draw[ultra thick] (0,0) [rounded corners=3pt]-- (.5,0) -- (1,-.3) -- (1.7,.5) -- (2.2,.2);\draw [->] (-.31,.002) -- (-.3,0);}
\,\,\,\mapsto \,
\cala\Big(\tikzmath[scale=.4] {\useasboundingbox (-2.2,-1.5) rectangle (3.2,.7);
\draw (-2,-1.2) [rounded corners=7pt]-- (-2,-.6) -- (-1.5,.5) [rounded corners=4pt]-- (-.6,-.5) [rounded corners=1]-- (1,-.5) -- (1,-.3) -- (0,.-.3) -- (0,-.1) [rounded corners=3pt]-- (1.5,-.1) -- (2,-.4) -- (2.7,.4) -- (3.2,.1);
\draw[->] (.46,-.1) -- (.45,-.1);}\Big).
\]
These are isomorphic to the weak units on $\cala$ and $\cala^\op$, and therefore are defects; they are equivalent to identity defects by~\cite[\remarkweakidentity\ \& \exampleonenotoneone]{BDH(1*1)}.

Assuming $\cala$ is finite, we now proceed to show that the unit and counit defects are finite.
Let $I_1,\ldots,I_4$ be as in \eqref{eq: picture of I_1 ... I_4}; the intervals $I_1$ and $I_3$ are genuinely bicolored, $I_2$ is white, and $I_4$ is black.
The actions of $r(I_1)$, $r(I_2)$, $r(I_3)$, $r(I_4)$ on the vacuum sector $H_0(r)$ are conjugate to the actions of 
$\cala(I_1)$, $\cala(I_2)\otimes \cala(I_4)$, $\cala(I_3)$, and $\IC$ on $H_0(\cala)$.
The condition of Definition~\ref{def: finiteness for defects} then holds by the split property of $\cala$.

By the same argument, one also shows that $s$ is a finite defect.
\end{proof}

\noindent
From this theorem and Corollary~\ref{cor:defectdualizable}, we have the following:

\begin{corollary} \label{cor:netdualizable}
A finite conformal net is fully dualizable.
\end{corollary}

In any $n$-category, a composition of fully dualizable $1$-morphisms is again fully dualizable; similarly a composition (either vertical or horizontal) of fully dualizable $2$-morphisms is again fully dualizable.  Thus, by Corollary~\ref{cor:sectordualizable}, the collection of finite sectors is closed under composition, and by Corollary~\ref{cor:defectdualizable} and Proposition~\ref{prop:defectfinite} below, the collection of finite defects is closed under composition.  By direct inspection, the collection of finite conformal nets is closed under tensor product.  Altogether we see that the collection of finite conformal nets, finite defects, finite sectors, and intertwiners forms a sub-symmetric-monoidal-$3$-category of the symmetric-monoidal-$3$-category of conformal nets.

Together, Corollaries~\ref{cor:sectordualizable},~\ref{cor:defectdualizable}, and~\ref{cor:netdualizable} establish the following:
\begin{theorem} \label{thm-duals}
The $3$-category of finite conformal nets, finite defects, finite sectors, and intertwiners has all duals.
\end{theorem}
\nid Applying the cobordism hypothesis (as before under the assumption that it applies to the symmetric monoidal $3$-category of conformal nets---see Footnote~\ref{foot-ch}), we obtain the corresponding topological field theories:
\begin{corollary} 
Associated to any finite conformal net $\cala$, there is a $3$-dimensional local framed topological field theory with target the $3$-category of conformal nets, whose value on the positively framed point is the conformal net $\cala$.
\end{corollary}

\subsection{Dualizable nets are finite}

In the preceding section we saw that the subcategory of finite conformal nets, finite defects, finite sectors, and intertwiners has all duals.  In this section, we prove that that this subcategory is in fact the maximal subcategory of the $3$-category of conformal nets that has all duals.

We already saw in Corollary~\ref{cor:sectordualizable} that a dualizable sector is necessarily finite.  We now show that a fully dualizable defect must be finite:
\begin{proposition} \label{prop:defectfinite}
Let $\cala$ and $\calb$ be finite conformal nets, and let ${}_\cala D_\calb$ be a defect.
If $D$ has an adjoint, then $D$ is finite.
\end{proposition}

\begin{proof}
Let $D^\vee$ be the dual of $D$, and let $r$ and $s$ be the counit and unit sectors, so that 
\[
\left(\tikzmath[scale=.9]{
\useasboundingbox (-.3,-1) rectangle (6.3,1);
\node (a) at (0,0) {$\cala$};\node (b) at (2,0) {$\calb$};\node (c) at (4,0) {$\cala$};\node (d) at (6,0) {$\calb$};
\draw[->] (a) -- node[above, yshift=-.1]{$\scriptstyle D$} (b);\draw[->] (a) to[out=-45, in=-135]node[below]{$\scriptstyle 1_\cala$} (c);\draw[->] (c) -- node[above, yshift=-.1]{$\scriptstyle D$} (d);
\draw[->] (b) to[out=45, in=135]node[above]{$\scriptstyle 1_\calb$} (d);\draw[->] (b) -- node[above, yshift=-.1, xshift=.3] {$\scriptstyle D^\vee$} (c);
\node at (2,-.52) {$\Downarrow$};\node at (2.3,-.52) {$r$};\node at (4,.52) {$\Downarrow$};\node at (4.3,.52) {$s$};
}%tikzmath
\right)
\tikzmath{\useasboundingbox (-.6,-1.2) rectangle (.6,1.2);\node{$\cong$};}
\left(
\tikzmath[scale=.9]{ \useasboundingbox (-.3,-1) rectangle (4.1,1);
\node (a') at (0,0) {$\cala$};\node (c') at (3.8,0) {$\calb$};
\draw[->] (a') .. controls (1,.7) and (2.8,.7) .. node[above]{$\scriptstyle D$} (c');
\draw[->] (a') .. controls (1,-.7) and (2.8,-.7) .. node[below]{$\scriptstyle D$} (c');
\node at (1.9,0) {$\Downarrow$};\node at (2.4,0) {$\scriptstyle 1_{D}$};}%tikzmath
\right).
\]
In other words, with $I_1,\ldots,I_{13}$ arranged as before
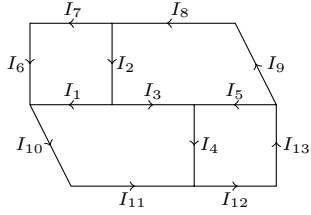
\[\tikzmath[scale=.09]{\draw (-12,0) -- (24,0) (0,0) -- (0,12) -- (18,12) -- (24,0) -- (24,-12) -- (12,-12) -- (12,0)(12,-12) -- (-6,-12) -- (-12,0) -- (-12,12) -- (0,12);
\draw (-5.7,2) node {$\scriptstyle I_1$}
(2,6) node {$\scriptstyle I_2$}
(6,2) node {$\scriptstyle I_3$}
(14.3,-6) node {$\scriptstyle I_4$}
(18,2) node {$\scriptstyle I_5$}
(-14,6) node {$\scriptstyle I_6$}
(-5.7,14) node {$\scriptstyle I_7$}
(10,14) node {$\scriptstyle I_8$}
(24,6) node {$\scriptstyle I_9$}
(-12,-6) node {$\scriptstyle I_{10}$}
(3,-14) node {$\scriptstyle I_{11}$}
(18,-14) node {$\scriptstyle I_{12}$}
(27,-6) node {$\scriptstyle I_{13}$};
\draw [->] (-6.3,0) -- (-6.31,0);
\draw [->] (17.8,0) -- (17.79,0);
\draw [->] (18.3,-12) -- (18.31,-12);
\draw [->] (6,0) -- (6.01,0);
\draw [->] (12,-6) -- (12,-6.01);
\draw [->] (24,-5.51) -- (24,-5.5);
\draw [->] (0,6.01) -- (0,6);
\draw [->] (-12,6.01) -- (-12,6);
\draw [->] (3.3,-12) -- (3.31,-12);
\draw [->] (-6.3,12) -- (-6.31,12);
\draw [->] (8.01,12) -- (8,12);
\draw [->] (-9.005,-6) -- (-9,-6.01);
\draw [->] (21.005,6) -- (21,6.01);
}\,\,
\]
we have
\begin{equation}\label{eq: duality eq: we're showing dualizable ==> finite}
\big(H_0(D)\,\boxtimes_{\calb(I_2)} s\big)\underset{D(I_1)\vee D^\vee(\bar I_3)\vee D(I_5)}\boxtimes \big(r\,\boxtimes_{\cala(I_4)} H_0(D)\big)\,\,\cong\,\, H_0(D).
\end{equation}
We check that $D$ is finite by showing that the action on \eqref{eq: duality eq: we're showing dualizable ==> finite} of the algebra $D(I_7) \otimes_\alg D(I_{12})$ 
extends to $D(I_7) \,\bar\otimes\, D(I_{12})$.

The Hilbert space $r$ is invertible
as a $D(I_1)\vee D^\vee(I_3)^\op\vee \cala(I_4)^\op$ - $(D(I_1)\vee D^\vee(I_3)^\op\vee \cala(I_4)^\op)'$-bimodule.
Similarly, the Hilbert space $s$ is an invertible $\calb(I_2)\vee D^\vee(I_3)\vee D(I_5)^\op$ - $(\calb(I_2)\vee D^\vee(I_3)\vee D(I_5)^\op)'$-bimodule.
Fusing \eqref{eq: duality eq: we're showing dualizable ==> finite} with the inverse bimodules $\bar r$ and $\bar s$,
and using the (non-canonical) isomorphisms
\[\begin{split}&L^2\big(D(I_1)\vee D^\vee(I_3)^\op\vee \cala(I_4)^\op\big)\,\,\cong\,\, H_0(D\circledast_\calb D^\vee)\\
&L^2\big(\calb(I_2)\vee D^\vee(I_3)\vee D(I_5)^\op\big)\,\,\cong\,\, H_0(D^\vee\circledast_\cala D),\end{split}\]
we get the Hilbert space
\begin{equation*}
\begin{split}
&\big(H_0(D)\,\boxtimes_{\calb} H_0(D^\vee\circledast_\cala D)\big)\underset{D(I_1)\vee D^\vee(\bar I_3)\vee D(I_5)}\boxtimes \big(H_0(D\circledast_\calb D^\vee)\,\boxtimes_{\cala} H_0(D)\big)\\
\cong\,&
\big(H_0(D)\boxtimes_{\calb} H_0(D^\vee)\boxtimes_\cala H_0(D)\big)\!\underset{D(I_1)\vee D^\vee(\bar I_3)\vee D(I_5)}\boxtimes \!\big(H_0(D)\boxtimes_\calb H_0(D^\vee)\boxtimes_{\cala} H_0(D)\big).
%\big(H_0(D)\underset{\calb}\boxtimes H_0(D^\vee)\underset\cala\boxtimes H_0(D)\big)\underset{D(I_1)\vee D^\vee(\bar I_3)\vee D(I_5)}\boxtimes \big(H_0(D)\underset\calb\boxtimes H_0(D^\vee)\underset\cala\boxtimes H_0(D)\big)
\end{split}
\end{equation*}
The latter is isomorphic to 
\begin{equation}\label{eq: H_0(D)H_0(D^dag)H_0(D)}
H_0(D)\,\boxtimes_{\calb} H_0(D^\vee)\boxtimes_\cala H_0(D)
\end{equation}
by the interchange isomorphism~\cite[Sec 6.D]{BDH(1*1)}. 
To be precise, letting $J_1,J_2\ldots, J_{10}$ be as in the following figure
\(
\tikzmath[scale=.07]{\draw (0,0) rectangle (36,12) (12,0)--(12,12) (24,0)--(24,12);
\draw
(15,6) node {$\scriptstyle J_1$}(27,6) node {$\scriptstyle J_2$}(39,6) node {$\scriptstyle J_3$}(30,14) node {$\scriptstyle J_4$}(18,14) node {$\scriptstyle J_5$}
(6,14) node {$\scriptstyle J_6$}(-3,6) node {$\scriptstyle J_7$}(6,-2) node {$\scriptstyle J_8$}(18,-2) node {$\scriptstyle J_9$}(30,-2) node {$\scriptstyle J_{10}$};
\draw [->] (6,0) -- (6.01,0);\draw [->] (18,0) -- (18.01,0);\draw [->] (30,0) -- (30.01,0);\draw [->] (6.01,12) -- (6,12);\draw [->] (18.01,12) -- (18,12);
\draw [->] (30.01,12) -- (30,12);\draw[->](36,6)--(36,6.01);\draw[<-](24,6)--(24,6.01);\draw[<-](12,6)--(12,6.01);\draw[<-](0,6)--(0,6.01);
}
\),
the Hilbert space \eqref{eq: H_0(D)H_0(D^dag)H_0(D)} is given by
$H_0(D)\,\boxtimes_{\calb(J_1)} H_0(D^\vee)\boxtimes_{\cala(J_2)} H_0(D)$.\smallskip

The intervals $J_6$ and $J_{10}$ correspond to $I_7$ and $I_{12}$, respectively.
Note that $H_0(D^\vee)$ is split as a $\calb(J_1)$-$\cala(J_2)$-bimodule.
Since the fusion of a split bimodule with any bimodule is always split, it follows that 
\eqref{eq: H_0(D)H_0(D^dag)H_0(D)} is split as a 
$D(J_6)\vee \cala(J_7)\vee D(J_8)$-$(D(J_{10})\vee \calb(J_3)\vee D(J_4))^\op$-bimodule.
In particular, it is split as a $D(J_6)$-$D(J_{10})^\op$-bimodule.
In other words, the completion of $D(J_6) \otimes_\alg D(J_{10})$ is isomorphic to the spatial tensor product $D(J_6) \,\bar\otimes\, D(J_{10})$.
\end{proof}

Finally, we show that fully dualizable conformal nets must be finite.  
Even though we do not have at hand a $3$-category of all (not-necessarily-finite) conformal nets, we do have enough of the structure of that hypothetical $3$-category to make sense of the notion of an arbitrary conformal net being fully dualizable, and therefore to make sense of the statement that a fully dualizable not-necessarily-finite conformal net must in fact be finite.  

Recall from Theorem~\ref{thm:finitenetdualizable} that any (not-necessarily-finite) conformal net $\cala$ has an ambidextrous dual $\cala^\op$ with evaluation defect ${}_{\cala \otimes \cala^\op}\raisebox{.5ex}{$r$}_{\underline\IC}$ and coevaluation defect ${}_{\underline\IC}\,\raisebox{.5ex}{$s$}_{\cala^\op \otimes \cala}$.  We call such a conformal net \emph{dualizable} if these evaluation and coevaluation defects $r$ and $s$ both have ambidextrous adjoints with dualizable unit and counit sectors.  This definition (specifically the notion of an adjunction for the evaluation and coevaluation defects) is well posed because, for any not-necessarily-finite conformal net $\calb$ and any defects ${}_\calb D_{\underline\IC}$ and ${}_{\underline\IC} E_\calb$, the fusion products $D \circledast_{\underline\IC} E$ and $E \circledast_\calb D$ are indeed defects (the first one by~\cite[Thm.\,1.44]{BDH(1*1)}; the second one because a $\underline\CC$-$\underline\CC$-defect is just a von Neumann algebra~\cite[Prop.\,1.22]{BDH(1*1)}).

\begin{theorem} \label{thm:netfinite}
Let $\cala$ be a not-necessarily-finite conformal net, and let $r$ and $s$ be the evaluation and coevaluation defects of the duality of $\cala$ and $\cala^\op$, given by $r:I\mapsto \cala\big(I_\circ\cup_{I_{\circ\!\bullet}} \bar I_\circ\big)$ and $s:I\mapsto \cala\big(\bar I_\bullet\cup_{I_{\circ\!\bullet}} I_\bullet\big)$.
If the defect $r$ has an adjoint, and its counit sector ${}_{r\circledast_{\underline \IC} r^\vee}R_{\,1_{\cala\otimes\cala^\op}}$ is dualizable, then the conformal net $\cala$ is finite.
\end{theorem}

\noindent Note that the proof of this proposition requires particular care:
$\cala$ is not assumed to have finite index, and so most of our previous results cannot be used here.

\begin{proof}
Recall that 
\[
r \Big( \tikzmath[scale=.4] {\useasboundingbox (-2.2,-1.1) rectangle (2.2,1.1);\draw (-2,-.7) [rounded corners=7pt]-- (-2,-.1) -- 
(-1.5,1) [rounded corners=4pt]-- (-.6,0) -- (0,0); \draw[ultra thick] (0,0) [rounded corners=3pt]-- (.5,0) -- (1,-.3) -- (1.7,.5) -- (2.2,.2);\draw [->] (-.31,.002) -- (-.3,0);}\Big) %tikzmath
=\cala\Big(\tikzmath[scale=.4] {\useasboundingbox (-2.3,-1.1) rectangle (.2,1.1);\draw (-2,-.7) [rounded corners=7pt]-- (-2,-.1) -- (-1.5,1) [rounded corners=4pt]-- (-.6,0)[rounded corners=1] -- (0,0) -- (0,.1); \draw (-2.15,-.55) [rounded corners=7pt]-- (-2.17,.1) -- (-1.45,1.19) [rounded corners=4pt]-- (-.6,.2)[rounded corners=1] -- (0,.2) -- (0,.1);\draw [->] (-.31,.202) -- (-.3,0.2);}\Big). %tikzmath
\]
By assumption, $r$ has an adjoint.
Let $r^\vee$ be its adjoint $\underline\IC$ -$(\cala\otimes \cala^\op)$-defect.
Let also ${}_{r\circledast_{\underline \IC} r^\vee}R_{\,1_{\cala\otimes\cala^\op}}$ and ${}_{1_{\underline \IC}}S_{r^\vee \circledast_{\cala\otimes\cala^\op} r}$
be the corresponding counit and unit sectors.
We now describe the algebras that act on the Hilbert spaces $R$ and $S$.

Take the ``standard circle'' $\partial[0,1]^2$ and cut it open at the point $(\frac 12,1)$.
Call the two resulting boundary points $p$ and $q$.
The resulting manifold, call it $M$, looks roughly  like this: 
\[
\,\,\,\,\tikzmath[scale = 0.35]{\useasboundingbox (-1,0) rectangle (3,3.2);
\draw (1,.1)to[out = 180,in = 10] (0,0)to[out = 100,in = -45](-1,2)to[out = 36,in = 198](0,2.5) circle (.05) circle (.025);
\draw (1,.1)to[out = 0,in = 170] (2,0) to[out = 80,in = 225](3,2)to[out = 144,in = -18](2,2.5) circle (.05) circle (.025);
\draw[->] (1.1,.1) -- (1.11,.1);\node[scale = .8] at (0,3) {$p$};\node[scale = .8] at (2,3) {$q$};}\,\,. %tikzmath
\]
Now consider its doubling $N:=M\cup_{\{p,q\}} \bar M$:
\begin{equation}\label{eq: consider its doubling}
N=\,\,\,\tikzmath[scale = 0.35]{\useasboundingbox (-1.2,-.2) rectangle (3.2,3.3);
\draw (1,-.1)to[out = 180,in = 10] (-.15,-.22)to[out = 100,in = -45](-1.27,2.03)to[out = 36,in = 198](0,2.67) arc (108:18:.1) +(18:.06) -- +(18:-.06) +(0,0) arc (18:-72:.1)
to[in = 36,out = 198] (-.96,2) to[in = 100,out = -45] (0.01,0.02) to[in = 180,out = 10] (1,.1);
\draw (1,-.1)to[out = 0,in = 170] (2.15,-.22) to[out = 80,in = 225] (3.27,2.03) to[out = 144,in = -18](2,2.67) arc (72: 162:.1) +(-18:.06) -- +(-18:-.06) +(0,0) arc (162: 252:.1)
to[in = 144,out = -18] (2.96,2) to[in = 80,out = 225] (1.99,0.02) to[in = 0,out = 170] (1,.1);
\draw[->] (1.15,.1) -- (1.16,.1);\node[scale = .8] at (0.1,3.2) {$p$};\node[scale = .8] at (1.9,3.2) {$q$};}\,, % tikzmath
\end{equation}
and let $\kappa:N\to N$ be the orientation reversing involution that exchanges $M$ and $\bar M$ and fixes $p$ and $q$. Given a 
$\kappa$-invariant neighborhood $J$ of $q$, let $J_\kappa:= [0,1]\cup J/\kappa$ be the bicolored interval with bicoloring given by $(J_\kappa)_\circ = [0,1]$ and $(J_\kappa)_\bullet = J/\kappa$
\begin{equation}\label{eq: J n Jk}
J=\tikzmath[scale = 0.35]{\useasboundingbox (1,-.2) rectangle (3.2,3.3);
\clip (3.8,2) -- (3,.8) -- (1,1.8) -- (1,3.5) -- cycle;
\draw (1,-.1)to[out = 0,in = 170] (2.15,-.22) to[out = 80,in = 225] (3.27,2.03) to[out = 144,in = -18](2,2.67) arc (72: 162:.1) arc (162: 252:.1)
to[in = 144,out = -18] (2.96,2) to[in = 80,out = 225] (1.99,0.02) to[in = 0,out = 170] (1,.1);
\draw[<-] (2.502,2.49) -- (2.5,2.491);
}%tikzmath
\qquad\quad\rightsquigarrow\quad\qquad J_\kappa =\tikzmath[scale = 0.35]{\useasboundingbox (1,-.2) rectangle (3.2,3.3);
\draw[ultra thick](2.43,1.1) to[in = -135,out = 65] (3.1,2.02) to[in = -16,out = 145] (1.88,2.61);\draw (1.88,2.61) -- +(164:1);
\draw[thick,->] (2.44,2.391) -- (2.442,2.39);
}\,.%tikzmath
\end{equation}
By definition of $(r\circledast_{\underline \IC} r^\vee)$-$(1_{\cala\otimes\cala^\op})$-sector,
the Hilbert space $R$ has actions of $\cala(J)$ for every subinterval $J\subset N$ that avoids $q$, and actions of $r^\vee (J_\kappa)$ for every $\kappa$-invariant interval $J$ that contains $q$.

The algebras acting on $S$ are somewhat easier to describe.
Consider the double $D:=[0,1]\cup_{\{0,1\}}[0,1]$ of the standard interval $[0,1]$, and let $\kappa:D\to D$ be the involution that exchanges the two copies of $[0,1]$.
The Hilbert space $S$ has an action of $\cala(J)$ for every subinterval $J\subset D$ that avoids the point $0$,
and an action of $r^\vee (J_\kappa)$ for every $\kappa$-invariant interval that contains $0$.

We find it convenient to think of $R$ as being associated to a saddle, and of $S$ as being associated to a cap:
\[
\tikzmath[scale = 0.4]{
\fill [gray!30] (-1,.2) -- (-1,2) [rounded corners=1]-- (0,2) -- (0,2.2) [sharp corners]-- (-1.2,2.2) -- (-1.2,.2) (2,2) -- (2,2.2) [rounded corners=1]-- (1,2.2) -- (1,2) -- (2,2);
\fill [spacecolor] (-1,0) -- (-1,2) [rounded corners=1]-- (0,2) -- (0,2.1) [sharp corners]-- (0,1.8) arc (-180:0:.5) -- (1,2.1) [rounded corners=1]-- (1,2) [sharp corners]-- (2.2,2) -- (2.2,0);(2,2) -- (2,2.2) [rounded corners=1]-- (1,2.2) -- (1,2) -- (2,2); \draw[densely dotted] (2,1.9) -- (2,.2) -- (-.9,.2); \draw (-1.1,.2) -- (-1.2,.2) -- (-1.2,2.2) [rounded corners=1]-- (0,2.2) -- (0,2) [sharp corners] -- (-1,2) -- 
(-1,0) -- (2.2,0) -- (2.2,2) [rounded corners=1] -- (1,2) -- (1,2.2) [sharp corners] -- (2,2.2) -- (2,2.1);\node at (.48,1.15) {$R$};\draw (.95,2.1) --node[scale=.8, above, xshift=1, yshift=-1]{$\scriptstyle r^{\!\vee}$} (1.05,2.1);} %tikzmath
\qquad\qquad\quad
\tikzmath[scale = 0.4]{\useasboundingbox (0,0) rectangle (2,2.5);\fill [spacecolor] (-.1,.1) arc (180:0:1.2) arc (0:-90:.1) -- (0,0) arc (270:180:.1);
\draw (2.3,.1) arc (0:-90:.1) -- (0,0) arc (270:180:.1);\draw[densely dotted] (2.3,.1) arc (0:90:.1) -- (0,.2) arc (90:180:.1);
\node at (1.05,1.15) {$S$};\draw (-.15,.1) --node[scale=.8, below, xshift=1]{$\scriptstyle r^{\!\vee}$} (-.05,.1) (-.1,.1) -- (-.1,.14);\node at (3,0) {.};
} %tikzmath
\]
The duality equation
\[
\left(\tikzmath[scale=.9]{\useasboundingbox (-.7,-1) rectangle (6.3,1);\node (a) at (0,0) {$\cala\otimes \cala^\op$};\node (b) at (2,0) {$\underline \IC$};\node (c) at (4,0) {$\cala\otimes \cala^\op$};
\node (d) at (6,0) {$\underline \IC$};\draw[->] (a) -- node[above, yshift=-.1]{$\scriptstyle r$} (b);\draw[->] (a) to[out=-45, in=-135]node[below]{$\scriptstyle 1_{\cala\otimes \cala^\op}$} (c);
\draw[->] (c) -- node[above, yshift=-.1, xshift=-1]{$\scriptstyle r$} (d);\draw[->] (b) to[out=45, in=135]node[above]{$\scriptstyle 1_{\underline \IC}$} (d);\draw[->] (b) -- node[above, yshift=-.1, xshift=4]
{$\scriptstyle r^\vee$} (c);\node at (2,-.52) {$\Downarrow$};\node at (2.3,-.52) {$\scriptstyle R$};\node at (4,.52) {$\Downarrow$};\node at (4.3,.52) {$\scriptstyle S$};}\right)%tikzmath
\tikzmath{\useasboundingbox (-.5,-1.2) rectangle (.5,1.2);\node{$\cong$};}%tikzmath
\left(\tikzmath[scale=.9]{ \useasboundingbox (-.3,-1) rectangle (4.1,1);\node (a') at (.4,0) {$\cala\otimes \cala^\op$};\node (c') at (3.8,0) {$\underline \IC$};
\draw[->] (a') to[out=35, in=145] node[above]{$\scriptstyle r$} (c');\draw[->] (a') to[out=-35, in=-145] node[below]{$\scriptstyle r$} (c');
\node at (2.1,0) {$\Downarrow$};\node at (2.5,0) {$\scriptstyle 1_r$};}\right)%tikzmath
\]
then translates into the statement
\begin{equation}\label{eq: morse cancellation}
\tikzmath[scale = 0.4]{\useasboundingbox (-1.2,0) rectangle (3.5,5);
\filldraw [gray!30] (-1,.2) -- (-1,2) [rounded corners=1]-- (0,2) -- (0,2.2) [sharp corners]-- (-1.2,2.2) -- (-1.2,.2)(2,2) -- (2,2.2) [rounded corners=1]-- (1,2.2) -- (1,2) -- (2,2);
\filldraw [spacecolor] (-1,0) -- (-1,2) [rounded corners=1]-- (0,2) -- (0,2.1) [sharp corners]-- (0,1.8) arc (-180:0:.5) -- (1,2.1) [rounded corners=1]-- (1,2) [sharp corners]-- (2.2,2) -- (2.2,0);(2,2) -- (2,2.2) [rounded corners=1]-- (1,2.2) -- (1,2) -- (2,2); \draw[densely dotted] (2,1.9) -- (2,.2) -- (-.9,.2); \draw (-1.1,.2) -- (-1.2,.2) -- (-1.2,2.2) [rounded corners=1]-- (0,2.2) -- (0,2) [sharp corners] -- (-1,2) -- (-1,0) -- (2.2,0) -- (2.2,2) [rounded corners=1] -- (1,2) -- (1,2.2) [sharp corners] -- (2,2.2) -- (2,2.1); \node at (.48,1.15) {$R$};
\draw (.95,2.1) --node[scale=.8, left, xshift=2, yshift=2]{$\scriptstyle r^{\!\vee}$} (1.05,2.1);
\filldraw [spacecolor] (2.6,0) [rounded corners=1]-- (3.6,0) [sharp corners]-- (3.6,.1) -- (3.6,2.1) [rounded corners=1]-- (3.6,2) [sharp corners]-- (2.6,2);
\filldraw [gray!30] (2.6,.2) -- (2.6,2) [rounded corners=1]-- (3.6,2) -- (3.6,2.2) [sharp corners]-- (2.4,2.2) -- (2.4,.2);
\draw (2.5,.2) -- (2.4,.2) -- (2.4,2.2) [rounded corners=1]-- (3.6,2.2) -- (3.6,2) [sharp corners]-- (2.6,2) -- (2.6,0) [rounded corners=1]-- (3.6,0) -- (3.6,.1);
\draw[densely dotted] (2.7,.2) [rounded corners=1]-- (3.6,.2) -- (3.6,.1);
\pgftransformxshift{31.3} %shift
\pgftransformyshift{70} %shift
\fill [spacecolor] (-.1,.1) arc (180:0:1.3) arc (0:-90:.1) -- (0,0) arc (270:180:.1);\draw (2.5,.1) arc (0:-90:.1) -- (0,0) arc (270:180:.1);
\draw[densely dotted] (2.5,.1) arc (0:90:.1) -- (0,.2) arc (90:180:.1); \node at (1.16,1.25) {$S$};\draw (-.15,.1) -- (-.05,.1) (-.1,.1) -- (-.1,.14);
\pgftransformxshift{-133.7} %shift
\filldraw [spacecolor] (2.6,0) [rounded corners=1]-- (3.6,0) [sharp corners]-- (3.6,.1) -- (3.6,2.1) [rounded corners=1]-- (3.6,2) [sharp corners]-- (2.6,2);
\filldraw [gray!30] (2.6,.2) -- (2.6,2) [rounded corners=1]-- (3.6,2) -- (3.6,2.2) [sharp corners]-- (2.4,2.2) -- (2.4,.2);
\draw (2.5,.2) -- (2.4,.2) -- (2.4,2.2) [rounded corners=1]-- (3.6,2.2) -- (3.6,2) [sharp corners]-- (2.6,2) -- (2.6,0) [rounded corners=1]-- (3.6,0) -- (3.6,.1);
\draw[densely dotted] (2.7,.2) [rounded corners=1]-- (3.6,.2) -- (3.6,.1);
} %tikzmath
\tikzmath{\useasboundingbox (-1,-1.2) rectangle (1,1.2);\node at (0,-.2) {$\cong$};}%tikzmath
\def\rc{1.5}\tikzmath[scale = 0.6]{\useasboundingbox (2.5,0) rectangle (3.6,3.2);
\filldraw [spacecolor] (2.6,0) [rounded corners=\rc]-- (3.6,0) [sharp corners]-- (3.6,.1) -- (3.6,2.1) [rounded corners=\rc]-- (3.6,2) [sharp corners]-- (2.6,2);
\filldraw [gray!30] (2.6,.2) -- (2.6,2) [rounded corners=\rc]-- (3.6,2) -- (3.6,2.2) [sharp corners]-- (2.4,2.2) -- (2.4,.2);
\draw (2.5,.2) -- (2.4,.2) -- (2.4,2.2) [rounded corners=\rc]-- (3.6,2.2) -- (3.6,2) [sharp corners]-- (2.6,2) -- (2.6,0) [rounded corners=\rc]-- (3.6,0) -- (3.6,.1);
\draw[densely dotted] (2.7,.2) [rounded corners=\rc]-- (3.6,.2) -- (3.6,.1);
} %tikzmath
\tikzmath{\useasboundingbox (-.3,-1.2) rectangle (.1,1.2);\node at (0,-.28) {$,$};}%tikzmath
\end{equation}
where the left hand side stands for the fusion of the Hilbert spaces\, 
$
\tikzmath[scale = 0.4]{\useasboundingbox (-2.4,-.2) rectangle (2.5,2);
\fill [spacecolor] (-.1,.1) arc (180:0:1.3) arc (0:-90:.1) -- (0,0) arc (270:180:.1);\draw (2.5,.1) arc (0:-90:.1) -- (0,0) arc (270:180:.1);
\draw[densely dotted] (2.5,.1) arc (0:90:.1) -- (0,.2) arc (90:180:.1);\node at (1.16,1.25) {$S$};\draw (-.15,.1) -- (-.05,.1) (-.1,.1) -- (-.1,.14);\node at (-.5,.9) {$\scriptstyle \otimes$};
\pgftransformxshift{-133.7} %shift
\filldraw [spacecolor] (2.6,0) [rounded corners=1]-- (3.6,0) [sharp corners]-- (3.6,.1) -- (3.6,2.1) [rounded corners=1]-- (3.6,2) [sharp corners]-- (2.6,2);
\filldraw [gray!30] (2.6,.2) -- (2.6,2) [rounded corners=1]-- (3.6,2) -- (3.6,2.2) [sharp corners]-- (2.4,2.2) -- (2.4,.2);\draw (2.5,.2) -- (2.4,.2) -- (2.4,2.2) [rounded corners=1]-- 
(3.6,2.2) -- (3.6,2) [sharp corners]-- (2.6,2) -- (2.6,0) [rounded corners=1]-- (3.6,0) -- (3.6,.1);\draw[densely dotted] (2.7,.2) [rounded corners=1]-- (3.6,.2) -- (3.6,.1);}
$
and
$
\tikzmath[scale = 0.4]{\useasboundingbox (-1.4,-.2) rectangle (6.5,2.2);
\filldraw [gray!30] (-1,.2) -- (-1,2) [rounded corners=1]-- (0,2) -- (0,2.2) [sharp corners]-- (-1.2,2.2) -- (-1.2,.2)(2,2) -- (2,2.2) [rounded corners=1]-- (1,2.2) -- (1,2) -- (2,2);
\filldraw [spacecolor] (-1,0) -- (-1,2) [rounded corners=1]-- (0,2) -- (0,2.1) [sharp corners]-- (0,1.8) arc (-180:0:.5) -- (1,2.1) [rounded corners=1]-- (1,2) [sharp corners]-- 
(2.2,2) -- (2.2,0);(2,2) -- (2,2.2) [rounded corners=1]-- (1,2.2) -- (1,2) -- (2,2);\draw[densely dotted] (2,1.9) -- (2,.2) -- (-.9,.2); \draw (-1.1,.2) -- (-1.2,.2) -- (-1.2,2.2) [rounded corners=1]-- (0,2.2) -- (0,2) [sharp corners] -- (-1,2) -- (-1,0) -- (2.2,0) -- (2.2,2) [rounded corners=1] -- (1,2) -- (1,2.2) [sharp corners] -- (2,2.2) -- (2,2.1);
\node at (.48,1.15) {$R$};\draw (.95,2.1) -- (1.05,2.1);\node at (3.6,.8) {$\scriptstyle \boxtimes_{\scriptstyle\cala(\hspace{.2cm})}$};\draw (4.06,.07) -- (4.06,1.4) (4.24,-.1) -- (4.24,1.23);
\pgftransformxshift{75} %shift
\filldraw [spacecolor] (2.6,0) [rounded corners=1]-- (3.6,0) [sharp corners]-- (3.6,.1) -- (3.6,2.1) [rounded corners=1]-- (3.6,2) [sharp corners]-- (2.6,2);
\filldraw [gray!30] (2.6,.2) -- (2.6,2) [rounded corners=1]-- (3.6,2) -- (3.6,2.2) [sharp corners]-- (2.4,2.2) -- (2.4,.2);
\draw (2.5,.2) -- (2.4,.2) -- (2.4,2.2) [rounded corners=1]-- (3.6,2.2) -- (3.6,2) [sharp corners]-- (2.6,2) -- (2.6,0) [rounded corners=1]-- (3.6,0) -- (3.6,.1);
\draw[densely dotted] (2.7,.2) [rounded corners=1]-- (3.6,.2) -- (3.6,.1);} %tikzmath
$
along the algebra 
\[
\cala\big(\tikzmath[scale = 0.4]{\draw (-1.1,0) [rounded corners=1]-- (0,0) -- (0,.2) -- (-1.1,.2);\draw[->] (-.51,.2) -- (-.5,.2);} %tikzmath
\big)\,\,\bar\otimes\,\, \Big(r^\vee\big(\tikzmath[scale = 0.4]{\useasboundingbox (-2,-.1) rectangle (.2,.1);\draw (-2,0) -- (-1.3,0);
\draw[ultra thick, line join=round] (-1.3,0) -- (0,0) -- (0,-.5);\draw[thick,->] (-.51,0) -- (-.5,0);} %tikzmath
\big) \,\circledast_{(\cala\otimes \cala^\op) (\,\tikzmath[scale=.3]{\draw (0,0) -- (0,-1);\draw[->] (0,-.46) -- (0,-.45);}\,)} %tikzmath
\cala\big(\tikzmath[scale = 0.4]{\useasboundingbox (-1.4,-.1) rectangle (0,.3);\draw (-1.1,-.5) -- (-1.1,0) [rounded corners=1]-- (0,0) -- (0,.2) [sharp corners]-- (-1.3,.2) -- (-1.3,-.5);
\draw[->] (-.51,.2) -- (-.5,.2);}\big)\Big) %tikzmath
\]
associated to the manifold
$ \tikzmath[scale = 0.4]{\useasboundingbox (-1.3,2) rectangle (3.6,2.3); \draw (-1.1,2) [rounded corners=1]-- (0,2) -- (0,2.2) -- (-1.1,2.2) (3.6,2.2) -- (1,2.2) -- (1,2) -- (3.6,2) -- cycle 
(.95,2.1) --node[scale=.8, above, xshift=1, yshift=-1]{$\scriptstyle r^{\!\vee}$} (1.05,2.1);} %tikzmath
$\,, and 
$ \,\def\rc{1}\tikzmath[scale = 0.4]{\useasboundingbox (2.4,-.2) rectangle (3.6,2.4);
\filldraw [spacecolor] (2.6,0) [rounded corners=\rc]-- (3.6,0) [sharp corners]-- (3.6,.1) -- (3.6,2.1) [rounded corners=\rc]-- (3.6,2) [sharp corners]-- (2.6,2);
\filldraw [gray!30] (2.6,.2) -- (2.6,2) [rounded corners=\rc]-- (3.6,2) -- (3.6,2.2) [sharp corners]-- (2.4,2.2) -- (2.4,.2);
\draw (2.5,.2) -- (2.4,.2) -- (2.4,2.2) [rounded corners=\rc]-- (3.6,2.2) -- (3.6,2) [sharp corners]-- (2.6,2) -- (2.6,0) [rounded corners=\rc]-- (3.6,0) -- (3.6,.1);
\draw[densely dotted] (2.7,.2) [rounded corners=\rc]-- (3.6,.2) -- (3.6,.1);}\, %tikzmath
$ stands for $1_r$.
The upper left $ \,\def\rc{1}\tikzmath[scale = 0.4]{\useasboundingbox (2.4,-.2) rectangle (3.6,2.4);
\filldraw [spacecolor] (2.6,0) [rounded corners=\rc]-- (3.6,0) [sharp corners]-- (3.6,.1) -- (3.6,2.1) [rounded corners=\rc]-- (3.6,2) [sharp corners]-- (2.6,2);
\filldraw [gray!30] (2.6,.2) -- (2.6,2) [rounded corners=\rc]-- (3.6,2) -- (3.6,2.2) [sharp corners]-- (2.4,2.2) -- (2.4,.2);
\draw (2.5,.2) -- (2.4,.2) -- (2.4,2.2) [rounded corners=\rc]-- (3.6,2.2) -- (3.6,2) [sharp corners]-- (2.6,2) -- (2.6,0) [rounded corners=\rc]-- (3.6,0) -- (3.6,.1);
\draw[densely dotted] (2.7,.2) [rounded corners=\rc]-- (3.6,.2) -- (3.6,.1); }\, %tikzmath
$ in \eqref{eq: morse cancellation} does not change anything, and so it can be safely ignored~\cite[\lemHKK]{BDH(modularity)}.
Equation \eqref{eq: morse cancellation} then becomes
\[
\tikzmath[scale = 0.4]{\useasboundingbox (-1.2,0) rectangle (3.5,5);
\filldraw [gray!30] (-1,.2) -- (-1,2) [rounded corners=1]-- (0,2) -- (0,2.2) [sharp corners]-- (-1.2,2.2) -- (-1.2,.2)(2,2) -- (2,2.2) [rounded corners=1]-- (1,2.2) -- (1,2) -- (2,2);
\filldraw [spacecolor] (-1,0) -- (-1,2) [rounded corners=1]-- (0,2) -- (0,2.1) [sharp corners]-- (0,1.8) arc (-180:0:.5) -- (1,2.1) [rounded corners=1]-- (1,2) [sharp corners]-- (2.2,2) -- (2.2,0);(2,2) -- (2,2.2) [rounded corners=1]-- (1,2.2) -- (1,2) -- (2,2); \draw[densely dotted] (2,1.9) -- (2,.2) -- (-.9,.2); \draw (-1.1,.2) -- (-1.2,.2) -- (-1.2,2.2) [rounded corners=1]-- (0,2.2) -- (0,2) [sharp corners] -- (-1,2) -- (-1,0) -- (2.2,0) -- (2.2,2) [rounded corners=1] -- (1,2) -- (1,2.2) [sharp corners] -- (2,2.2) -- (2,2.1); \node at (.48,1.15) {$R$};
\draw (.95,2.1) --node[scale=.8, left, xshift=2, yshift=2]{$\scriptstyle r^{\!\vee}$} (1.05,2.1);
\filldraw [spacecolor] (2.6,0) [rounded corners=1]-- (3.6,0) [sharp corners]-- (3.6,.1) -- (3.6,2.1) [rounded corners=1]-- (3.6,2) [sharp corners]-- (2.6,2);
\filldraw [gray!30] (2.6,.2) -- (2.6,2) [rounded corners=1]-- (3.6,2) -- (3.6,2.2) [sharp corners]-- (2.4,2.2) -- (2.4,.2);
\draw (2.5,.2) -- (2.4,.2) -- (2.4,2.2) [rounded corners=1]-- (3.6,2.2) -- (3.6,2) [sharp corners]-- (2.6,2) -- (2.6,0) [rounded corners=1]-- (3.6,0) -- (3.6,.1);
\draw[densely dotted] (2.7,.2) [rounded corners=1]-- (3.6,.2) -- (3.6,.1);
\pgftransformxshift{31.3} %shift
\pgftransformyshift{70} %shift
\fill [spacecolor] (-.1,.1) arc (180:0:1.3) arc (0:-90:.1) -- (0,0) arc (270:180:.1);\draw (2.5,.1) arc (0:-90:.1) -- (0,0) arc (270:180:.1);
\draw[densely dotted] (2.5,.1) arc (0:90:.1) -- (0,.2) arc (90:180:.1); \node at (1.16,1.25) {$S$};\draw (-.15,.1) -- (-.05,.1) (-.1,.1) -- (-.1,.14);} %tikzmath
\tikzmath{\useasboundingbox (-1,-1.2) rectangle (1,1.2);\node at (0,-.2) {$\cong$};}%tikzmath
\def\rc{1.5}\tikzmath[scale = 0.6]{\useasboundingbox (2.5,0) rectangle (3.6,3.2);
\filldraw [spacecolor] (2.6,0) [rounded corners=\rc]-- (3.6,0) [sharp corners]-- (3.6,.1) -- (3.6,2.1) [rounded corners=\rc]-- (3.6,2) [sharp corners]-- (2.6,2);
\filldraw [gray!30] (2.6,.2) -- (2.6,2) [rounded corners=\rc]-- (3.6,2) -- (3.6,2.2) [sharp corners]-- (2.4,2.2) -- (2.4,.2);
\draw (2.5,.2) -- (2.4,.2) -- (2.4,2.2) [rounded corners=\rc]-- (3.6,2.2) -- (3.6,2) [sharp corners]-- (2.6,2) -- (2.6,0) [rounded corners=\rc]-- (3.6,0) -- (3.6,.1);
\draw[densely dotted] (2.7,.2) [rounded corners=\rc]-- (3.6,.2) -- (3.6,.1);} %tikzmath
\tikzmath{\useasboundingbox (-.25,-1.2) rectangle (.1,1.2);\node at (0,-.28) {$,$};}%tikzmath
\]
or, equivalently, after flattening the above $2$-manifolds:
\begin{equation}\label{eq: flattened out}
\,\,\,\,\quad\tikzmath[scale=.06]{\pgftransformxshift{-10} %shift
\filldraw[fill = gray!30] (90:15) arc (90:270:15) -- (270:8) arc (267.5:92.5:8) -- cycle;\node[scale=.9] at (140:11.6) {$R$}; \pgftransformxshift{20} %shift
\filldraw[fill = gray!30] (90:15) arc (90:-90:15) -- (-90:8) arc (-87.5:87.5:8) -- cycle; \pgftransformxshift{-10} %shift
\fill (-8,0) circle (.4) node[scale=.9, left, yshift=1, xshift=3] {$\scriptscriptstyle r^{\!\vee}$};\filldraw [fill = gray!30] circle (7);\fill (-7,0) circle (.4);\node [scale=.9] {$S$};} %tikzmath
\,\,\,\,\quad\cong\,\,\,\,\quad
\tikzmath[scale=.06]{\filldraw[fill = gray!30] circle (14);}\,\,\,. %tikzmath
\smallskip\end{equation}
Let us name $I_1,\ldots,I_6$ the intervals that appear in \eqref{eq: flattened out}
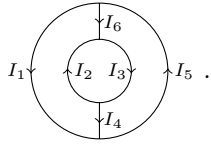
\[
\tikzmath[scale=.06]{\draw circle (7) circle (15) (0,7) -- (0,15) (0,-7) -- (0,-15);
\draw[->] (0,10.51) --node[right, xshift=-2]{$\scriptstyle I_6$} (0,10.5);\draw[->] (0,-11.5) --node[right, xshift=-2, yshift=1]{$\scriptstyle I_4$} (0,-11.51);
\draw[->] (-15,-.5) --node[left, xshift=2, yshift=.7]{$\scriptstyle I_1$} (-15,-.51);\draw[->] (-7,0.5) --node[right, xshift=-1, yshift=-1]{$\scriptstyle I_2$} (-7,0.51);
\draw[->] (7,-.5) --node[left, xshift=2, yshift=.7]{$\scriptstyle I_3$} (7,-.51);\draw[->] (15,0.5) --node[right, xshift=-1, yshift=-1]{$\scriptstyle I_5$} (15,0.51);
}.
\]
Let $\kappa$ be the reflection in the horizontal axis, 
and let $K:=(I_2)_\kappa=[0,1]\cup I_2/\kappa$ be as in \eqref{eq: J n Jk}, bicolored by $K_\circ=[0,1]$ and $K_\bullet=I_2/\kappa$.
We also abbreviate $H_0(I_3\cup I_4\cup I_5\cup I_6,\cala)$ by $H_0(\cala)$.
The left hand side of \eqref{eq: flattened out} stands for the fusion of
$S$ with $R\boxtimes_{\cala(I_6\cup I_4)} H_0(\cala)$ along the algebra
\[
r^\vee(K)\vee \cala(I_3)=r^\vee(K\cup \bar I_6)\circledast_{(\cala\otimes \cala^\op)(I_6)}\cala(I_6\cup I_3\cup I_4),
\] 
where we identify $(\cala\otimes \cala^\op)(I_6)$ with $\cala(I_6\cup I_4)$ using the reflection $\kappa: \bar I_6\stackrel{\scriptscriptstyle \cong}\to I_4$.  

Recall~\cite[Lec 21]{Lurie-vnalgcourse} that a dagger functor $F$ is called `completely additive' if whenever the collection $\iota_\alpha: M_\alpha \ra M$ exhibit $M$ as the direct sum $\bigoplus M_\alpha$, then also $F(\iota_\alpha) : F(M_\alpha) \ra F(M)$ exhibit $F(M)$ as $\bigoplus F(M_\alpha)$.  (We called such a functor `normal' in~\cite[App. B.VIII]{BDH(1*1)}.)
The functor
\[
S \!\underset{r^\vee(K)\vee \cala(I_3)}\boxtimes \! \Big(\!-\,\boxtimes_{\cala(I_6\cup I_4)} H_0(\cala)\Big)\,\,\,:\,\,\,
r^\vee(K\cup \bar I_6)\text{-modules}\,\to\,
\cala(I_5)\text{-modules}
\]
is completely additive.
It is therefore given by Connes fusion with a certain $r^\vee(K\cup \bar I_6)^\op$-$\cala(I_5)$-bimodule~\cite[Lec 21]{Lurie-vnalgcourse}.
It then follows from \eqref{eq: flattened out} that the Hilbert space $R$ is invertible as $r^\vee(K\cup \bar I_6)$ - $\cala(I_1)^\op$-bimodule.

Recall that $R$ is finite as $(r\circledast_{\underline \IC} r^\vee)$-$(1_{\cala\otimes\cala^\op})$-sector.
In other words, it is finite as an
\[
\cala\Big(\tikzmath[scale = 0.35]{\useasboundingbox (-1.2,-.2) rectangle (3.2,3.3);\clip (-1.8,2) -- (-1,.8) -- (1,1.8) -- (1,3.5) -- cycle;
\draw (1,-.1)to[out = 180,in = 10] (-.15,-.22)to[out = 100,in = -45](-1.27,2.03)to[out = 36,in = 198](0,2.67) arc (108:18:.1) arc (18:-72:.1)
to[in = 36,out = 198] (-.96,2) to[in = 100,out = -45] (0.01,0.02) to[in = 180,out = 10] (1,.1);\draw[->] (-.442,2.52) -- (-.44,2.521);} %tikzmath
\Big)\,\bar\otimes\,r^\vee\Big(\tikzmath[scale = 0.35]{\useasboundingbox (-1.2,-.2) rectangle (3.2,3.3);
\draw[ultra thick](2.43,1.1) to[in = -135,out = 65] (3.1,2.02) to[in = -16,out = 145] (1.88,2.61);\draw (1.88,2.61) -- +(164:1);\draw[thick,->] (2.44,2.391) -- (2.442,2.39);} %tikzmath
\Big)\,\,\,\text{-}\,\,\,\cala\Big(\tikzmath[scale = 0.35]{\useasboundingbox (-1.2,-.2) rectangle (3.2,3.3);\clip (-1,-.3) -- (-1,.8) -- (1,1.8) -- (3,.8) -- (3,-.3) -- cycle;
\draw (1,-.1)to[out = 180,in = 10] (-.15,-.22)to[out = 100,in = -45](-1.27,2.03)to[out = 36,in = 198](0,2.67) arc (108:18:.1) arc (18:-72:.1)to[in = 36,out = 198] (-.96,2) to[in = 100,out = -45] (0.01,0.02) to[in = 180,out = 10] (1,.1);\draw (1,-.1)to[out = 0,in = 170] (2.15,-.22) to[out = 80,in = 225] (3.27,2.03) to[out = 144,in = -18](2,2.67) arc (72: 162:.1) arc (162: 252:.1)
to[in = 144,out = -18] (2.96,2) to[in = 80,out = 225] (1.99,0.02) to[in = 0,out = 170] (1,.1);\draw[->] (.66,.1005) -- (.65,.1);\draw[->] (1.35,-.12) -- (1.36,-.1205);}\Big)\,\,\,\text{-bimodule,} %tikzmath
\]
where we again draw our intervals as in \eqref{eq: consider its doubling}.
We know from our previous discussion that $R$ is invertible as an
\[
r^\vee\Big(\tikzmath[scale = 0.35]{\useasboundingbox (-1.2,-.2) rectangle (3.2,3.3);
\draw[ultra thick](2.43,1.1) to[in = -135,out = 65] (3.1,2.02) to[in = -16,out = 145] (1.88,2.61);\draw (1.88,2.61) -- +(164:1);\draw[thick,->] (2.44,2.391) -- (2.442,2.39);} %tikzmath
\Big)\,\,\,\text{-}\,\,\,\cala\Big(\tikzmath[scale = 0.35]{\useasboundingbox (-1.2,-.2) rectangle (3.2,3.3);\clip (-1,-.3) -- (-1,.8) -- (-1.8,2) -- (1,3.5)-- (1,1.8) -- (3,.8) -- (3,-.3) -- cycle;
\draw (1,-.1)to[out = 180,in = 10] (-.15,-.22)to[out = 100,in = -45](-1.27,2.03)to[out = 36,in = 198](0,2.67) arc (108:18:.1) arc (18:-72:.1)to[in = 36,out = 198] (-.96,2) to[in = 100,out = -45] (0.01,0.02) to[in = 180,out = 10] (1,.1);\draw (1,-.1)to[out = 0,in = 170] (2.15,-.22) to[out = 80,in = 225] (3.27,2.03) to[out = 144,in = -18](2,2.67) arc (72: 162:.1) arc (162: 252:.1)
to[in = 144,out = -18] (2.96,2) to[in = 80,out = 225] (1.99,0.02) to[in = 0,out = 170] (1,.1);\draw[->] (.86,.1) -- (.85,.1);}\Big)\,\,\,\text{-bimodule.} %tikzmath
\]
Let $Q$ be the inverse bimodule.
Twisting it by a diffeomorphism
$\tikzmath[scale = 0.3]{\useasboundingbox (1.4,0) rectangle (3.6,3.1);\clip (3.8,2) -- (3,.8) -- (1,1.8) -- (1,3.5) -- cycle;
\draw (1,-.1)to[out = 0,in = 170] (2.15,-.22) to[out = 80,in = 225] (3.27,2.03) to[out = 144,in = -18](2,2.67) arc (72: 162:.1) arc (162: 252:.1)
to[in = 144,out = -18] (2.96,2) to[in = 80,out = 225] (1.99,0.02) to[in = 0,out = 170] (1,.1);\draw[<-] (2.502,2.49) -- (2.5,2.491);} %tikzmath
\cong\tikzmath[scale = 0.3]{\useasboundingbox (-1.4,0) rectangle (3.2,2.3);\clip (-1,-.3) -- (-1,.8) -- (-1.8,2) -- (1,3.5)-- (1,1.8) -- (3,.8) -- (3,-.3) -- cycle;
\draw (1,-.1)to[out = 180,in = 10] (-.15,-.22)to[out = 100,in = -45](-1.27,2.03)to[out = 36,in = 198](0,2.67) arc (108:18:.1) arc (18:-72:.1)to[in = 36,out = 198] (-.96,2) to[in = 100,out = -45] (0.01,0.02) to[in = 180,out = 10] (1,.1);\draw (1,-.1)to[out = 0,in = 170] (2.15,-.22) to[out = 80,in = 225] (3.27,2.03) to[out = 144,in = -18](2,2.67) arc (72: 162:.1) arc (162: 252:.1)
to[in = 144,out = -18] (2.96,2) to[in = 80,out = 225] (1.99,0.02) to[in = 0,out = 170] (1,.1);\draw[->] (.86,.1) -- (.85,.1);} %tikzmath
$,
we may treat $Q$ as an 
\[
\cala\Big(\tikzmath[scale = 0.3]{\useasboundingbox (.6,-.2) rectangle (3.6,3.3);\clip (3.8,2) -- (3,.8) -- (1,1.8) -- (1,3.5) -- cycle;
\draw (1,-.1)to[out = 0,in = 170] (2.15,-.22) to[out = 80,in = 225] (3.27,2.03) to[out = 144,in = -18](2,2.67) arc (72: 162:.1) arc (162: 252:.1)
to[in = 144,out = -18] (2.96,2) to[in = 80,out = 225] (1.99,0.02) to[in = 0,out = 170] (1,.1);\draw[<-] (2.502,2.49) -- (2.5,2.491);} %tikzmath
\Big)\,\,\,\text{-}\,\,\,r^\vee\Big(\tikzmath[scale = 0.35]{\useasboundingbox (.3,-.2) rectangle (3.3,3.3);\draw[ultra thick](2.43,1.1) to[in = -135,out = 65] (3.1,2.02) to[in = -16,out = 145] (1.88,2.61);
\draw (1.88,2.61) -- +(164:1);\draw[thick,->] (2.44,2.391) -- (2.442,2.39);}\Big)\,\,\,\text{-bimodule.} %tikzmath
\]
By definition, it then satisfies
\[
Q\,\,\,\boxtimes_{\!\!r^\vee\textstyle(\tikzmath[scale = 0.25]{\useasboundingbox (.3,.2) rectangle (3.3,3.3);
\draw[ultra thick](2.43,1.1) to[in = -135,out = 65] (3.1,2.02) to[in = -16,out = 145] (1.88,2.61);\draw (1.88,2.61) -- +(164:1);\draw[thick,->] (2.44,2.391) -- (2.442,2.39);} %tikzmath
)} R\,\,\,\,\cong\,\,\,\, H_0\bigg(\tikzmath[scale = 0.3]{\useasboundingbox (-1.4,-.4) rectangle (3.2,2.8);
\draw (1,-.1)to[out = 180,in = 10] (-.15,-.22)to[out = 100,in = -45](-1.27,2.03)to[out = 36,in = 198](0,2.67) arc (108:18:.1) arc (18:-72:.1)
to[in = 36,out = 198] (-.96,2) to[in = 100,out = -45] (0.01,0.02) to[in = 180,out = 10] (1,.1);
\draw (1,-.1)to[out = 0,in = 170] (2.15,-.22) to[out = 80,in = 225] (3.27,2.03) to[out = 144,in = -18](2,2.67) arc (72: 162:.1) arc (162: 252:.1)
to[in = 144,out = -18] (2.96,2) to[in = 80,out = 225] (1.99,0.02) to[in = 0,out = 170] (1,.1);
\draw[->] (1.15,.1) -- (1.16,.1);}\,,\, %tikzmath
\cala\bigg).
\]
We then also have (applying~\cite[\lemHKK]{BDH(modularity)})
\begin{equation}\label{eq: recreating the vacuum from R}
\bigg(L^2\big(\cala\big(\tikzmath[scale = 0.3]{\useasboundingbox (-1.5,3.3) rectangle (1.2,.2);
\clip (-1.8,2) -- (-1,.8) -- (1,1.8) -- (1,3.5) -- cycle;
\draw (1,-.1)to[out = 180,in = 10] (-.15,-.22)to[out = 100,in = -45](-1.27,2.03)to[out = 36,in = 198](0,2.67) arc (108:18:.1) arc (18:-72:.1)
to[in = 36,out = 198] (-.96,2) to[in = 100,out = -45] (0.01,0.02) to[in = 180,out = 10] (1,.1);\draw[->] (-.442,2.52) -- (-.44,2.521);} %tikzmath
\big)\big)\otimes Q\bigg)
\,\boxtimes_{
\cala\textstyle(\tikzmath[scale = 0.25]{\useasboundingbox (-1.5,3.3) rectangle (1.2,.2);\clip (-1.8,2) -- (-1,.8) -- (1,1.8) -- (1,3.5) -- cycle;
\draw (1,-.1)to[out = 180,in = 10] (-.15,-.22)to[out = 100,in = -45](-1.27,2.03)to[out = 36,in = 198](0,2.67) arc (108:18:.1) arc (18:-72:.1)
to[in = 36,out = 198] (-.96,2) to[in = 100,out = -45] (0.01,0.02) to[in = 180,out = 10] (1,.1);\draw[->] (-.442,2.52) -- (-.44,2.521);} %tikzmath
)\,\bar\otimes\,r^\vee(\tikzmath[scale = 0.25]{\useasboundingbox (.3,.2) rectangle (3.3,3.3);
\draw[ultra thick](2.43,1.1) to[in = -135,out = 65] (3.1,2.02) to[in = -16,out = 145] (1.88,2.61);\draw (1.88,2.61) -- +(164:1);\draw[thick,->] (2.44,2.391) -- (2.442,2.39);} %tikzmath
)}
R\,\,\,\,\cong\,\,\,\, H_0\bigg(\tikzmath[scale = 0.3]{\useasboundingbox (-1.4,-.4) rectangle (3.2,2.8);
\draw (1,-.1)to[out = 180,in = 10] (-.15,-.22)to[out = 100,in = -45](-1.27,2.03)to[out = 36,in = 198](0,2.67) arc (108:18:.1) arc (18:-72:.1)
to[in = 36,out = 198] (-.96,2) to[in = 100,out = -45] (0.01,0.02) to[in = 180,out = 10] (1,.1);
\draw (1,-.1)to[out = 0,in = 170] (2.15,-.22) to[out = 80,in = 225] (3.27,2.03) to[out = 144,in = -18](2,2.67) arc (72: 162:.1) arc (162: 252:.1)
to[in = 144,out = -18] (2.96,2) to[in = 80,out = 225] (1.99,0.02) to[in = 0,out = 170] (1,.1);
\draw[->] (1.15,.1) -- (1.16,.1);}\,,\,\cala\bigg). %tikzmath
\end{equation}
Since
$
L^2\big(\cala\big(\tikzmath[scale = 0.3]{\useasboundingbox (-1.5,3.3) rectangle (1.2,.2);\clip (-1.8,2) -- (-1,.8) -- (1,1.8) -- (1,3.5) -- cycle;
\draw (1,-.1)to[out = 180,in = 10] (-.15,-.22)to[out = 100,in = -45](-1.27,2.03)to[out = 36,in = 198](0,2.67) arc (108:18:.1) arc (18:-72:.1)
to[in = 36,out = 198] (-.96,2) to[in = 100,out = -45] (0.01,0.02) to[in = 180,out = 10] (1,.1);\draw[->] (-.442,2.52) -- (-.44,2.521);}\big)\big)\otimes Q %tikzmath
$
is an invertible
\[
\cala\Big(\tikzmath[scale = 0.35]{\useasboundingbox (-1.5,3.3) rectangle (1.2,.2);\clip (-1.8,2) -- (-1,.8) -- (1,1.8) -- (1,3.5) -- cycle;
\draw (1,-.1)to[out = 180,in = 10] (-.15,-.22)to[out = 100,in = -45](-1.27,2.03)to[out = 36,in = 198](0,2.67) arc (108:18:.1) arc (18:-72:.1)
to[in = 36,out = 198] (-.96,2) to[in = 100,out = -45] (0.01,0.02) to[in = 180,out = 10] (1,.1);\draw[->] (-.442,2.52) -- (-.44,2.521);} %tikzmath
\Big)\,\bar\otimes\,\cala\Big(\tikzmath[scale = 0.35]{\useasboundingbox (.8,.2) rectangle (3.5,3.3);\clip (3.8,2) -- (3,.8) -- (1,1.8) -- (1,3.5) -- cycle;
\draw (1,-.1)to[out = 0,in = 170] (2.15,-.22) to[out = 80,in = 225] (3.27,2.03) to[out = 144,in = -18](2,2.67) arc (72: 162:.1) arc (162: 252:.1)
to[in = 144,out = -18] (2.96,2) to[in = 80,out = 225] (1.99,0.02) to[in = 0,out = 170] (1,.1);\draw[<-] (2.502,2.49) -- (2.5,2.491);} %tikzmath
\Big)\,\,\,\text{-}\,\,\,\cala\Big(\tikzmath[scale = 0.35]{\useasboundingbox (-1.5,3.3) rectangle (1.2,.2);\clip (-1.8,2) -- (-1,.8) -- (1,1.8) -- (1,3.5) -- cycle;
\draw (1,-.1)to[out = 180,in = 10] (-.15,-.22)to[out = 100,in = -45](-1.27,2.03)to[out = 36,in = 198](0,2.67) arc (108:18:.1) arc (18:-72:.1)to[in = 36,out = 198] (-.96,2) to[in = 100,out = -45] (0.01,0.02) to[in = 180,out = 10] (1,.1);\draw[->] (-.442,2.52) -- (-.44,2.521);} %tikzmath
\Big)\,\bar\otimes\,r^\vee\Big(\tikzmath[scale = 0.35]{\useasboundingbox (.3,.2) rectangle (3.3,3.3);\draw[ultra thick](2.43,1.1) to[in = -135,out = 65] (3.1,2.02) to[in = -16,out = 145] (1.88,2.61);\draw (1.88,2.61) -- +(164:1);\draw[thick,->] (2.44,2.391) -- (2.442,2.39);}\Big)\,\,\,\text{-bimodule,} %tikzmath
\]
it follows from \eqref{eq: recreating the vacuum from R} and the finiteness of $R$ that
$
H_0\Big(\tikzmath[scale = 0.25]{\useasboundingbox (-1.4,-.4) rectangle (3.2,2.8);
\draw (1,-.1)to[out = 180,in = 10] (-.15,-.22)to[out = 100,in = -45](-1.27,2.03)to[out = 36,in = 198](0,2.67) arc (108:18:.1) arc (18:-72:.1)
to[in = 36,out = 198] (-.96,2) to[in = 100,out = -45] (0.01,0.02) to[in = 180,out = 10] (1,.1);
\draw (1,-.1)to[out = 0,in = 170] (2.15,-.22) to[out = 80,in = 225] (3.27,2.03) to[out = 144,in = -18](2,2.67) arc (72: 162:.1) arc (162: 252:.1)
to[in = 144,out = -18] (2.96,2) to[in = 80,out = 225] (1.99,0.02) to[in = 0,out = 170] (1,.1);
\draw[->] (1.15,.1) -- (1.16,.1);}\,,\,\cala\Big) %tikzmath
$
is finite as an
\[
\cala\Big(\tikzmath[scale = 0.35]{\useasboundingbox (-1.4,-.2) rectangle (2.6,3.3);\clip (-1.8,2) -- (-1,.8) -- (1,1.8) -- (1,3.5) -- cycle;
\draw (1,-.1)to[out = 180,in = 10] (-.15,-.22)to[out = 100,in = -45](-1.27,2.03)to[out = 36,in = 198](0,2.67) arc (108:18:.1) arc (18:-72:.1)
to[in = 36,out = 198] (-.96,2) to[in = 100,out = -45] (0.01,0.02) to[in = 180,out = 10] (1,.1);\draw[->] (-.442,2.52) -- (-.44,2.521);} %tikzmath
\Big)\,\bar\otimes\,\cala\Big(\tikzmath[scale = 0.35]{\useasboundingbox (-.6,-.2) rectangle (3.4,3.3);\clip (3.8,2) -- (3,.8) -- (1,1.8) -- (1,3.5) -- cycle;
\draw (1,-.1)to[out = 0,in = 170] (2.15,-.22) to[out = 80,in = 225] (3.27,2.03) to[out = 144,in = -18](2,2.67) arc (72: 162:.1) arc (162: 252:.1)
to[in = 144,out = -18] (2.96,2) to[in = 80,out = 225] (1.99,0.02) to[in = 0,out = 170] (1,.1);\draw[<-] (2.502,2.49) -- (2.5,2.491);} %tikzmath
\Big)\,\,\,\text{-}\,\,\,\cala\Big(\tikzmath[scale = 0.35]{\useasboundingbox (-1.2,-.2) rectangle (3.2,3.3);\clip (-1,-.3) -- (-1,.8) -- (1,1.8) -- (3,.8) -- (3,-.3) -- cycle;
\draw (1,-.1)to[out = 180,in = 10] (-.15,-.22)to[out = 100,in = -45](-1.27,2.03)to[out = 36,in = 198](0,2.67) arc (108:18:.1) arc (18:-72:.1)to[in = 36,out = 198] (-.96,2) to[in = 100,out = -45] (0.01,0.02) to[in = 180,out = 10] (1,.1);\draw (1,-.1)to[out = 0,in = 170] (2.15,-.22) to[out = 80,in = 225] (3.27,2.03) to[out = 144,in = -18](2,2.67) arc (72: 162:.1) arc (162: 252:.1)
to[in = 144,out = -18] (2.96,2) to[in = 80,out = 225] (1.99,0.02) to[in = 0,out = 170] (1,.1);\draw[->] (.66,.1005) -- (.65,.1);\draw[->] (1.35,-.12) -- (1.36,-.1205);}\Big)\,\,\,\text{-bimodule.}%tikzmath
\]
The latter is the definition of what it means for $\cala$ to be finite.
\end{proof}

\begin{scholium}\label{end scholium}
Recall that strong additivity was assumed as part of our definition of coordinate free conformal nets \cite[Def.\,1.1]{BDH(nets)}.

The above theorem implies that in a hypothetical $3$-category of strongly additive not-necessarily-finite-index conformal nets, a fully dualizable conformal net is necessarily finite-index.  We expect that even more is true, namely, that in a hypothetical $3$-category of not-necessarily-finite-index and not-necessarily-strongly-additive conformal nets, a fully dualizable conformal net is finite-index (and hence strongly additive, by~\cite{Longo-Xu(dichotomy)}).
\end{scholium}

\appendix

\section{Disintegrating sectors between finite defects}

Sectors between conformal nets disintegrate into irreducibles~\cite{\KLM};
in this section we generalize that result to the case of sectors between defects, provided the defects are finite.

\begin{lemma} \label{lem:disintegratingdefects}
Let $\cala$ and $\calb$ be conformal nets.
Let ${}_\cala D_\calb$ and ${}_\cala E_\calb$ be irreducible finite defects. Then any $D$-$E$-sector disintegrates into a direct integral of irreducible $D$-$E$-sectors.
\end{lemma}
\begin{proof}

Pick a countable collection\footnote{In fact this collection can be chosen to be finite.} of pairs of bicolored subintervals $\{I_i^- \subset I_i^+\}_{i \in \mathcal{I}}$ of the standard bicolored circle, with the closure of $I_i^-$ contained in the interior of $I_i^+$, satisfying the following conditions: 
\begin{itemize}
\item[-] $I_i^-$ is genuinely bicolored if and only if $I_i^+$ is genuinely bicolored; 
\item[-] for all $p,q \in S^1$, either
\begin{itemize}
\item[a.] there exists an $i \in \mathcal{I}$ such that $p,q \in I_i^-$, or
\item[b.] there exist $i,j \in \mathcal{I}$ such that $p \in I_i^-$, $q \in I_j^-$, and $I_i^+ \cap I_j^+ = \emptyset$.
\end{itemize}
\end{itemize}
For each $i \in \mathcal{I}$, let $A_i^{\pm}$ denote the algebra $\cA(I_i^\pm)$, $\cB(I_i^\pm)$, $D(I_i^\pm)$, or $E(I_i^\pm)$ depending on whether $I_i^\pm$ is white, black, contains the top defect point, or contains the bottom defect point, respectively.  Because $D$ and $E$ are finite, there exists, for each $i \in \mathcal{I}$, a type $I$ factor $N_i$ such that 
\[
A_i^- \subset N_i \subset A_i^+.
\]
Let $\cK_i \subset N_i$ denote the ideal of compact operators in $N_i$.  For each $i,j \in \mathcal{I}$ such that $I_i^+ \cap I_j^+ = \emptyset$, let $R_{ij} \subset \cK_i \ast \cK_j$ be the kernel of the projection $\cK_i \ast \cK_j \ra \cK_i \otimes \cK_j$ from the free product $C^\ast$-algebra to the tensor product $C^\ast$-algebra.  For each $i,j \in \mathcal{I}$ such that $I_i^+ \subset I_j^-$, let $S_{ij} \subset \cK_i \ast \cK_j$ be the kernel of the map $\cK_i \ast \cK_j \ra \cK_i \vee_{N_j} \cK_j$, where $\cK_i \vee_{N_j} \cK_j$ is the subalgebra of $N_j$ generated by $\cK_i$ and $\cK_j$.  Now define
\[
\mathfrak{A} := (\underset{i}{\ast} \cK_i) / I%\tikz[baseline=0]{\node[scale=2]{$\ast$};}_i \cK_i / I
\]
where $I$ is the norm-closed ideal generated by $R_{ij}$ for $i,j \in \mathcal{I}$ such that $I_i^+ \cap I_j^+ = \emptyset$, and $S_{ij}$ for $i,j \in \mathcal{I}$ such that $I_i^+ \subset I_j^-$.

By Lemma~\ref{lem-repcat}, the category of $D$--$E$-sectors is equivalent to the category of representations of $\mathfrak{A}$ whose restriction to each $\cK_i$ is nondegenerate. 

Because $\mathfrak{A}$ is a separable $C^\ast$-algebra, the category $\mathrm{Rep}(\mathfrak{A})$ admits direct integral decompositions.  We need to show that given a representation $H$ of $\mathfrak{A}$ whose restriction to each $\cK_i$ is nondegenerate, and a direct integral decomposition $(H,\rho) \cong \int_{x \in X} (H^x,\rho^x) dx$, almost all of the integrands $(H^x,\rho^x)$ again have the property that their restriction to each $\cK_i$ is nondegenerate.  Pick an increasing sequence of projections $p_n^i \in \cK_i$, $n \in \mathbb{N}$, that forms an approximate unit.  By Lemma~\ref{lem-a2}, we have that $1 = \sup \rho_i(p_n^i) = \sup \int^\oplus \rho^x_i(p_n^i) =  \int^\oplus \sup p^x_i(p_n^i)$.
This implies that for almost all $x$, we have $\sup \rho^x_i(p_n^i) = 1$.
\end{proof}

\begin{lemma} \label{lem-repcat}
The category of representation of $\mathfrak{A}$ whose restriction to each $\cK_i$ is nondegenerate is equivalent to the category of $D$--$E$-sectors.
\end{lemma}
\begin{proof}
By construction, every $D$--$E$-sector yields an appropriate representation of $\mathfrak{A}$.
Now suppose that we have a representation of $\mathfrak{A}$ on a Hilbert space $H$ whose restriction to each $\cK_i$ is nondegenerate.  By the classification of the representations of compact operators, the action of $\cK_i$ extends uniquely to a normal action $\rho_i : N_i \ra B(H)$.  For every $i,j \in \mathcal{I}$ such that $I_i^+ \cap I_j^+ = \emptyset$, the action of $\cK_i \ast \cK_j$ descends to an action of $\cK_i \otimes \cK_j$; by the ultraweak density of $\cK_i$ in $N_i$, the actions of $N_i$ and $N_j$ commute.  Now, for every $i,j \in \mathcal{I}$ such that $I_i^+ \subset I_j^-$, the action of $\cK_i \ast \cK_j$ descends to an action of $\cK_i \vee_{N_j} \cK_j$.  By~\cite[Cor 53]{\KLM}, that action of $\cK_i \vee_{N_j} \cK_j$ extends uniquely to a normal action $\tilde{\rho_j}:N_j \ra B(H)$, which agrees with $\rho_j$ by the ultraweak density of $\cK_j$ inside $N_j$.  We therefore have a diagram
\[
\xymatrix@C=.4cm{
&& \cK_i \ar[rr] \ar[d] \ar[ddrrr] && N_i \ar[d] \ar[ddr] \\
\cK_j \ar[rr] \ar[drrrrr] && \cK_i \vee_{N_j} \cK_j \ar[rr] \ar[drrr] && N_j\! \ar[dr] \\
&&&&& \!B(H)
}
\]
where all triangles are known to commute except possibly the triangle with edge $N_i \ra N_j$.  The missing triangle commutes because $\cK_i$ is ultraweakly dense in $N_i$.  Therefore, by~\cite[Lem 2.5]{BDH(1*1)}, the actions $\rho_i |_{A_i^-}$ assemble into a $D$--$E$-sector structure on $H$. 
\end{proof}

\begin{lemma} \label{lem-a2}
Let $H_x$ be a measurable family of Hilbert spaces over a probability space $X$.
For each $n \in \IN$, let $p_{n,x} \in B(H_x)$ be a measurable family of projections indexed by the points of $X$.  Assume furthermore that for every $x \in X$, the sequence $\{p_{n,x}\}_{n \in \IN}$ is increasing.  Then
\begin{equation*}
\int^\oplus \sup p_{n,x} = \sup \int^\oplus p_{n,x}\,.
\end{equation*}
\end{lemma}
\begin{proof}
Let $M\subset B(H)$ be the abelian von Neumann algebra on $H:=\int^\oplus H_x$ generated by $\int^\oplus f(x) p_{n,x}$ for all $f \in L^\infty(X)$ and $n \in \IN$.  Note that $M \cong L^\infty(Y)$ for some measure space $Y$.  Since $L^\infty(X) \subset M$, we have a measurable map $\pi: Y \rightarrow X$ and we can write $M = \int^\oplus_X M_x$, where $M_x = L^\infty(\pi^{-1}(x))$.  The projections $p_{n,x} \in M_x$ correspond to measurable subsets $Z_{n,x} \in \pi^{-1}(x)$, and the equation $\int^\oplus \sup p_{n,x} = \sup \int^\oplus p_{n,x}$ follows from the fact that $\bigsqcup_x \bigcup_n Z_{n,x} = \bigcup_n \bigsqcup_x Z_{n,x}$.
\end{proof}

\section{A variant vertical composition}\label{app-fusion}

In \cite[\S2.C]{BDH(1*1)}, we defined the vertical composition of two sectors ${}_DH_E$ and ${}_EK_F$ to be the fusion along half of each `circle', $H\boxtimes_{E(S^1_\top)}K$, with the evident remaining actions of $D$ and $F$:
\begin{equation}\label{eq: vfusion}
{}_DH\boxtimes_EK_F
\,=\,
\mathsf{fusion_v}\,\,
\Big(\tikzmath[scale=\squarescale]{\fill[spacecolor] (0,0) -- (0,12) -- (12,12) -- (12,0);\draw (6,0) -- (0,0) -- (0,12) -- (6,12);\draw[ultra thick](6,12) -- (12,12) -- (12,0) -- (6,0); 
\draw (-3,6) node {$\cala$}(15,6) node {$\calb$}(6,15) node {$D$}(6,-3) node {$E$}(6,6)node {$H$};}%tikzmath
,\tikzmath[scale=\squarescale]{\fill[spacecolor] (0,0) -- (0,12) -- (12,12) -- (12,0);\draw (6,0) -- (0,0) -- (0,12) -- (6,12);\draw[ultra thick](6,12) -- (12,12) -- (12,0) -- (6,0); 
\draw (-3,6) node {$\cala$}(15,6) node {$\calb$}(6,15) node {$E$}(6,-3) node {$F$}(6,6)node {$K$};}%tikzmath
\Big) 
=\tikzmath[scale=\squarescale]
{\fill[spacecolor] (0,3) -- (0,15) -- (12,15) -- (12,3);\draw (6,3) -- (0,3) -- (0,15) -- (6,15);\draw[ultra thick](6,15) -- (12,15) -- (12,3) -- (6,3); 
\fill[spacecolor] (0,-15) -- (0,-3) -- (12,-3) -- (12,-15);\draw (6,-15) -- (0,-15) -- (0,-3) -- (6,-3);\draw[ultra thick](6,-3) -- (12,-3) -- (12,-15) -- (6,-15); 
\draw %(-9,0) node {$\cala$}(21,0) node {$\calb$}
(6,18) node {$D$}(6,9)node {$H$}(6,-18) node {$F$}(6,-9)node {$K$};
\draw[<->] (-.2,9) to[out=180,in=180] (-.2,-9); 
\draw[<->] (-.2,6) to[out=200,in=160] (-.2,-6); 
\draw[<->] (-.2,2.9) to[out=230,in=130] (-.2,-2.9); 
\draw[<->] (2.7,2.7) to[out=255,in=105] (2.7,-2.7); 
\draw[<->] (12.5,9) to[out=0,in=0] (12.5,-9); 
\draw[<->] (12.5,6) to[out=-20,in=20] (12.5,-6); 
\draw[<->] (12.5,2.6) to[out=-50,in=50] (12.5,-2.6); 
\draw[<->] (9.3,2.6) to[out=-75,in=75] (9.3,-2.6); 
\draw[<->] (6,2.6) -- (6,-2.6); 
  }\,.
\end{equation}
An alternative definition would be to fuse along a `quarter-circle':
\begin{equation}\label{eq: altvfusion}
\tikzmath[scale=\squarescale]{ 
\pgftransformxshift{15}\draw[<->] (20.1-.7,2.8) -- (20.1-.7,-2.8); 
\pgftransformxshift{15}\draw[<->] (23.2-.7,2.8) -- (23.2-.7,-2.8); 
\pgftransformxshift{510} 
\fill[spacecolor] (0,3) -- (0,15) -- (12,15) -- (12,3);
\draw (6,3) -- (0,3) -- (0,15) -- (6,15);
\draw[ultra thick](6,15) -- (12,15) -- (12,3) -- (6,3); 
\fill[spacecolor] (0,-15) -- (0,-3) -- (12,-3) -- (12,-15);
\draw (6,-15) -- (0,-15) -- (0,-3) -- (6,-3);
\draw[ultra thick](6,-3) -- (12,-3) -- (12,-15) -- (6,-15); 
\draw (6,9)node {$H$} (6,-9)node {$K$}; 
\draw[<->] (12.2-.7,2.6) -- (12.2-.7,-2.6); 
\draw[<->] (9.1-.7,2.6) -- (9.1-.7,-2.6);}
\end{equation}
and to equip the resulting Hilbert space with the structure of a $D$-$F$-sector by means of a diffeomorphism
\begin{equation*}
\varphi\,:\,\,
\tikzmath[baseline=0, scale=\squarescale]{ 
\draw (6,-6) -- (0,-6) -- (0,6) -- (6,6);
\draw[ultra thick](6,-6) -- (12,-6) -- (12,6) -- (6,6); 
}
\,\,\stackrel{\scriptstyle\cong}\longrightarrow\,\,
\tikzmath[baseline=0, scale=\squarescale]{ 
\draw (6,-12) -- (0,-12) -- (0,12) -- (6,12);
\draw[ultra thick](6,-12) -- (12,-12) -- (12,12) -- (6,12); 
\draw (-.5,0) -- (.5,0);
\draw (12-.7,0) -- (12.7,0);
}\,,
\end{equation*}
compatible with the local coordinates around the color-change points.  Specifically, the resulting sector is $\varphi^*(H\boxtimes_{E(I)}K)$, where $I$ is the top quarter of the circle (associated to the sector $K$), or equivalently the bottom quarter of the circle (associated to the sector $H$).

\begin{lemma}
Let ${}_DH_E$ and ${}_EK_F$ be sectors, and let $\varphi$ be a diffeomorphism from the standard circle to the larger circle, as above.  Then the vertical fusion $H\boxtimes_{E(S^1_\top)}K$ from \eqref{eq: vfusion} is (non-canonically) isomorphic, as a $D$-$F$-sector, to the alternative fusion $\varphi^*(H\boxtimes_{E(I)}K)$ from \eqref{eq: altvfusion}.\end{lemma}

\begin{proof}
Let $\psi_1:S^1\to S^1$ be a diffeomorphism which maps the lower semi-circle $S^1_\bot$ to the lower quarter-circle (drawn here as an edge of a square) and satisfies $\varphi|_{S^1_\top}=\psi_1|_{S^1_\top}$,
let $\psi_2:S^1\to S^1$ be a diffeomorphism which maps the upper semi-circle $S^1_\bot$ to the upper quarter-circle and satisfies $\varphi|_{S^1_\bot}=\psi_2|_{S^1_\bot}$,
and let $u_{\psi_1}$ and $u_{\psi_2}$ be unitaries implementing these diffeomorphisms (these exist by~\cite[Prop.\,1.10]{BDH(1*1)}).
We assume without loss of generality that $\psi_2=j\circ\psi_1\circ j$, where $j$ is the reflection along the horizontal axis of symmetry.
Then $u_{\psi_1}\boxtimes u_{\psi_2}$ maps $H\boxtimes_{E(S^1_\top)}K$ to $H\boxtimes_{E(I)}K$, and is an isomorphism of $D$-$F$-sectors
$H\boxtimes_{E(S^1_\top)}K\cong \varphi^*(H\boxtimes_{E(I)}K)$.
\end{proof}

\bibliographystyle{amsalpha}
\bibliography{../Files/db-cn3}

\newcommand{\noopsort}[1]{}
\providecommand{\bysame}{\leavevmode\hbox to3em{\hrulefill}\thinspace}
\providecommand{\MR}{\relax\ifhmode\unskip\space\fi MR }
\providecommand{\MRhref}[2]{%
  \href{http://www.ams.org/mathscinet-getitem?mr=#1}{#2}
}
\providecommand{\href}[2]{#2}
\begin{thebibliography}{DSPS17}

\bibitem[AF17]{Ayala-Francis:cobordism-hypothesis}
David Ayala and John Francis, \emph{The cobordism hypothesis},
  arXiv:1705.02240, 2017.

\bibitem[BD95]{baezdolan}
John~C. Baez and James Dolan, \emph{Higher-dimensional algebra and topological
  quantum field theory}, J. Math. Phys. \textbf{36} (1995), no.~11, 6073--6105.

\bibitem[BDH14]{BDH(Dualizability+Index-of-subfactors)}
Arthur Bartels, Christopher~L. Douglas, and Andr\'{e} Henriques,
  \emph{Dualizability and index of subfactors}, Quantum Topology \textbf{5}
  (2014), 289--345, arXiv:1110.5671.

\bibitem[BDH15]{BDH(nets)}
\bysame, \emph{Conformal nets {I}: {C}oordinate-free nets}, Int. Math. Res.
  Not. \textbf{13} (2015), 4975--5052, arXiv:1302.2604.

\bibitem[BDH17]{BDH(modularity)}
\bysame, \emph{Conformal nets {II}: {C}onformal blocks}, Comm. Math. Phys.
  \textbf{354} (2017), 393--458, arXiv:1409.8672.

\bibitem[BDH18]{BDH(3-category)}
\bysame, \emph{\textcolor{white}{g}\!\!\!\! {C}onformal nets {IV}: {T}he
  3-category}, Alg. Geom. Top. \textbf{18} ({{\noopsort{2019}}2018}), 897--956,
  arxiv:1605.00662.

\bibitem[BDH19]{BDH(1*1)}
\bysame, \emph{{F}usion of defects (formerly ``{C}onformal nets {III}: {F}usion
  of defects'')}, Mem. Amer. Math. Soc. \textbf{1237}
  ({{\noopsort{2018}}2019}), 1--108, arXiv:1310.8263.

\bibitem[BK01]{Bakalov-Kirillov(Lect-tens-cat+mod-func)}
Bojko Bakalov and Alexander Kirillov, Jr., \emph{Lectures on tensor categories
  and modular functors}, University Lecture Series, vol.~21, American
  Mathematical Society, Providence, RI, 2001. \MR{1797619 (2002d:18003)}

\bibitem[CS15]{Calaque-Scheimbauer:Note-cobordism-category}
Damien Calaque and Claudia Scheimbauer, \emph{A note on the
  $(\infty,n)$-category of cobordisms}, arXiv:1509.08906, 2015.

\bibitem[DSPS17]{DSS}
Christopher~L. Douglas, Christopher Schommer-Pries, and Noah Snyder,
  \emph{Dualizable tensor categories}, Mem. Amer. Math. Soc. (2017),
  arxiv:1312.7188.

\bibitem[Gui18]{BinGui:CategoricalExtensionsOfConformalNets}
Bin Gui, \emph{Categorical extensions of conformal nets}, arXiv:1812.04470,
  2018.

\bibitem[Hen17]{CS(pt)}
Andr\'e Henriques, \emph{What {C}hern--{S}imons theory assigns to a point},
  Proc. Natl. Acad. Sci. USA \textbf{114} (2017), no.~51, 13418--13423.

\bibitem[KLM01]{Kawahigashi-Longo-Mueger(2001multi-interval)}
Yasuyuki Kawahigashi, Roberto Longo, and Michael M{\"u}ger,
  \emph{Multi-interval subfactors and modularity of representations in
  conformal field theory}, Comm. Math. Phys. \textbf{219} (2001), no.~3,
  631--669. \MR{MR1838752 (2002g:81059)}

\bibitem[Lur09]{Lurie(On-classification-TFT)}
Jacob Lurie, \emph{On the classification of topological field theories},
  Current developments in mathematics, 2008, Int. Press, Somerville, MA, 2009,
  pp.~129--280. \MR{2555928 (2010k:57064)}

\bibitem[Lur11]{Lurie-vnalgcourse}
\bysame, \emph{Lecture notes on von {N}eumann algebras}, 2011,
  http://www.math.harvard. edu/$\scriptstyle\sim$lurie/261y.html.

\bibitem[LX04]{Longo-Xu(dichotomy)}
Roberto Longo and Feng Xu, \emph{Topological sectors and a dichotomy in
  conformal field theory}, Comm. Math. Phys. \textbf{251} (2004), no.~2,
  321--364. \MR{2100058 (2005i:81087)}

\end{thebibliography}

\end{document}